\documentclass[12pt]{elsarticle}

\usepackage{hyperref}
\hypersetup{
    colorlinks=true,
    linkcolor=blue,
    filecolor=blue,
    urlcolor=blue,
    citecolor=cyan,
}
\usepackage{amssymb}
\usepackage{color,xcolor}
\usepackage{amsthm,epstopdf}
\usepackage{amsmath,amsfonts,bm}
\usepackage[left=2.54cm,right=2.54cm,top=3.17cm,bottom=3.17cm]{geometry}
\usepackage{algorithm,algpseudocode}
\usepackage{lipsum}

\makeatletter
\newenvironment{breakablealgorithm}
  {
   \begin{center}
     \refstepcounter{algorithm}
     \hrule height.8pt depth0pt \kern2pt
     \renewcommand{\caption}[2][\relax]{
       {\raggedright\textbf{\ALG@name~\thealgorithm} ##2\par}%
       \ifx\relax##1\relax 
         \addcontentsline{loa}{algorithm}{\protect\numberline{\thealgorithm}##2}%
       \else 
         \addcontentsline{loa}{algorithm}{\protect\numberline{\thealgorithm}##1}%
       \fi
       \kern2pt\hrule\kern2pt
     }
  }{
     \kern2pt\hrule\relax
   \end{center}
  }
\makeatother

\newtheorem{example}{Example}

\newtheorem{proposition}{Proposition}
\newtheorem{remark}{Remark}

 \numberwithin{equation}{section}
 \numberwithin{theorem}{section}
  \numberwithin{lemma}{section}
   \numberwithin{remark}{section}
      \numberwithin{example}{section}
         \numberwithin{figure}{section}
 \numberwithin{proposition}{section}

\graphicspath{{pdf/}}
\allowdisplaybreaks[3]

\journal{}

\begin{document}

\begin{frontmatter}



\title{Three   discontinuous Galerkin methods for  one- and two-dimensional nonlinear Dirac equations with a scalar self-interaction}


\author[label1]{Shu-Cun Li}
\ead{lisc26@163.com}
\author[label2]{Huazhong Tang\corref{cor1}}
\ead{hztang@math.pku.edu.cn}
\address[label1]
{School of Mathematics and Information Technology, Xingtai University, Xingtai 054001, P.R. China}
\address[label2]{Center for Applied Physics and Technology, HEDPS and LMAM,
    School of Mathematical Sciences, Peking University, Beijing 100871,
    P.R.China}

\cortext[cor1]{Corresponding author.}

\begin{abstract}This paper develops three high-order accurate discontinuous Galerkin (DG) methods for the one-dimensional (1D) and two-dimensional (2D) nonlinear Dirac (NLD) equations  with a general scalar self-interaction. They are
   the Runge-Kutta DG (RKDG) method and the DG methods with the one-stage fourth-order Lax-Wendroff type time discretizaiton (LWDG) and the two-stage fourth-order accurate time discretization (TSDG). The RKDG method uses the spatial DG approximation
    to discretize the NLD equations and then utilize the explicit multistage high-order Runge-Kutta time discretization for the first-order time derivatives, while the LWDG   and   TSDG methods, on the contrary, first give the one-stage fourth-order Lax-Wendroff type and the two-stage fourth-order time discretizations of the NLD equations, respectively, and then discretize the first- and higher-order spatial derivatives by using the spatial DG approximation.
    The $L^{2}$ stability of the 2D semi-discrete DG approximation
    is proved in the RKDG methods for a general triangulation,
    and
    the computational complexities of three 1D DG methods are  estimated.
    Numerical experiments are conducted to validate the accuracy and the conservative properties of the proposed methods. The interactions of the solitary waves, the standing and travelling waves are  investigated numerically and the 2D breathing pattern is observed.
\end{abstract}
\end{frontmatter}


\section{Introduction}
\label{intro}

The Dirac equation is a relativistic wave equation in particle physics,
and provides  a natural description of an electron \cite{Dirac1928}.
It  predicted the existence of ``negative'' states for the electron and proton, and  thus successfully predicted the existence of antimatter.
After Dirac found the linear equation of the electron, the basic idea of nonlinear description of the elementary particle with spin-1/2 appeared, which made it possible to consider its self-interaction \cite{Ivanenko1938,Finkelstein1951,Finkelstein1956},
and the nonlinear Dirac (NLD) equation was proposed as a possible basis model for a unified field theory \cite{Heisenberg1957}.
The NLD equation allows solitary wave solutions or particle-like solutions (the stable localized solutions with finite energy and charge) \cite{Ranada1983},
that is to say, the particles appear as intense localized regions of field \cite{Weyl1950}.
Around the 1970s and 1980s,
wide interest of
physicists and mathematicians was attracted by different NLD models  with different self-interactions, mainly including the Thirring model \cite{Thirring1958}, the Soler model \cite{Soler1970}, the Gross-Neveu model \cite{Gross1974} (equivalent to the massless Soler model), and the bag model \cite{Mathieu1984} (i.e. the solitary waves with finite compact support),
especially to look for the solitary wave solutions and to investigate the related physical and mathematical properties \cite{Ranada1983}.
%
Since entering the 21st century, the Dirac equation is  used to study the structures and dynamical properties of the two-dimensional (2D) materials such as graphene and graphite \cite{Novoselov2005,Castro2009,Abanin2011,Fefferman2012} and the relativistic effects in molecules in super intense lasers \cite{Fillion-Gourdeau2013} etc.
Moreover, the Bose-Einstein condensates in a honeycomb optical lattice can also be described by a NLD equation in the long wavelength, mean field limit \cite{Haddad2009}. 
 %
Mathematical interests related to the NLD equation are mainly manifested in
 deriving the analytical solitary wave solutions,
 the stability analysis of the NLD solitary waves,
the analysis of global well-posedness and numerical methods etc.
For the 1D NLD equation (i.e. one space dimension), several analytical solitary wave solutions are derived in literature. For example, the solitary wave solutions of the 1D NLD equation with arbitrary nonlinearity was studied in \cite{Cooper2010}.
However, for the high-dimensional NLD equation, there are no explicit solitary wave solutions \cite{Cuevas-Maraver2018}.
The stability of the solitary waves can be found in \cite{Shao2014,Cuevas-Maraver2016,Lakoba2018}.
The readers are referred to  the review in \cite{Xu2013} and references therein.

Numerical method has become one of the important tools to  derive the NLD solitary wave solutions, and to investigate their stability and interaction etc.
%
 The Crank-Nicolson (CN) scheme was first proposed
 for the 1D Soler model and used to simulate the interaction dynamics of the  NLD solitary waves in \cite{Alvarez1981,Alvarez1983}. Such interaction dynamics problem
 was carefully revisited in \cite{Shao2005} by utilizing a fourth-order accurate RKDG method \cite{Shao2006}.
Besides the recovery of the phenomena in \cite{Alvarez1981},
several new ones were observed, e.g. collapse in binary and ternary collisions of two-hump NLD solitary waves \cite{Shao2005}, a long-lived oscillating state formed with an approximate constant frequency in collisions of two standing waves \cite{Shao2006}, and the full repulsion in binary and ternary collisions of out-of-phase waves \cite{Shao2008}. Those numerical results also inferred that the two-hump profile could undermine the stability during the scattering of the NLD solitary waves. It is worth noting that the two-hump profile was first pointed out in \cite{Shao2005} and later gotten noticed by other researchers.
The multi-hump solitary waves  were further studied in \cite{Xu2015} in theory.
There also exist many other numerical schemes for solving the 1D NLD equation:
 the split-step spectral method \cite{Frutos1989},   the linearized CN scheme \cite{Alvarez1992}, the Legendre rational spectral method \cite{Wang2004}, the multisymplectic Runge-Kutta  method  \cite{Hong2006},
 the adaptive mesh method \cite{Wang2007}, the time-splitting methods \cite{Li2017}, and the compact methods \cite{Li2018a} etc.
 A review of the current state-of-the-art of numerical methods for the 1D  {NLD equation} was presented in \cite{Xu2013}. For the 1D {NLD equation} with the scalar and vector self-interaction, the CN schemes, the linearized CN schemes, the odd-even hopscotch scheme, the leapfrog scheme, a semi-implicit finite difference scheme, and the exponential operator splitting  schemes were extendedly proposed  and analyzed in the way of the accuracy and the time reversibility as well as the conservation of the discrete charge, energy and linear momentum.
For the NLD equation in the nonrelativistic limit regime,
the error estimates of  the CN scheme, the exponential wave integrator Fourier pseudospectral method and the time-splitting Fourier pseudospectral method   were studied in \cite{Bao2016}, and a uniformly accurate multiscale time integrator pseudospectral method was proposed in \cite{Cai2018}.

Unfortunately,  the existing work on numerical study of the 2D NLD equation is very limited.
The integrating-factor method was studied in \cite{Hoz2010}, whose
  authors pointed out that ``it is not clear how these schemes (the RKDG methods \cite{Shao2006}) could be generalized to higher spatial dimensions''.
   The standing wave solutions and vortex solutions were discussed by using the Fourier spectral method \cite{Cuevas-Maraver2016,Cuevas-Maraver2018}.
  Two kinds of high-order conservative schemes were obtained with the time-midpoint and the time-splitting methods as well as the operator-compensation method \cite{Li2019}.
 The aim of this paper is to extend the 1D RKDG method \cite{Shao2006} to the 2D NLD equation with a general scalar self-interaction, and to develop the LWDG and TSDG methods, i.e. the discontinuous Galerkin (DG) schemes  with the one-stage fourth-order accurate Lax-Wendroff type time discretization \cite{Qiu2005} and the two-stage fourth-order accurate time discretization \cite{Li2016,Yuan2020}. 
The RKDG method uses
the spatial DG approximation to discretize the NLD equations and then utilize the explicit
multistage high-order Runge-Kutta time discretization for the first-order time derivatives,
while the LWDG and TSDG methods, on the contrary, first give the one-stage
fourth-order Lax-Wendroff type and the two-stage fourth-order time discretizations of the NLD
equations, and then discretize the first- and higher-order spatial derivatives by using the
spatial DG approximation.
The computational complexities of three  1D fully discrete  DG
methods will be theoretically estimated
and their accuracy and performance will be validated by numerical experiments.

This paper is organized as follows.  Section \ref{Sec_preliminary} introduces
 the NLD equations and their  conservation laws, and discusses the standing and travelling wave solutions of the 2D NLD equations. 
 Section \ref{Sec_methods} proposes three high-order accurate DG methods. The $L^{2}$ stability of the 2D semi-discrete DG method is discussed  in the RKDG scheme for
 a general triangulation, and the computational complexities of the 1D fully discrete DG methods are also studied. Section \ref{Sec_results} conducts some experiments  to validate the accuracy   and the conservative properties of our DG methods and to investigate some new phenomena.  Section \ref{Sec_conclusion} draws the conclusion. 

\section{Nonlinear Dirac equations}
\label{Sec_preliminary}

This section introduces the 1D and 2D  NLD equations with a general scalar
 self-interaction and the 2D standing and travelling wave solutions.

This paper    is concerned with numerical methods for
the 1D NLD equation
\begin{equation}\label{1dNLDE}
  \partial_{t} \Psi \left( t,x \right) + \sigma_{1} \partial_{x} \Psi \left( t,x \right)  + \mathrm{i} g \left( s \right) \sigma_{3} \Psi \left( t,x \right)  = 0,\ x \in \mathbb{R},\ t \geq 0,
\end{equation}
with  the spinor unknown $\Psi = \Psi \left( t,x \right) = \left( \psi_{1} \left( t,x \right), \psi_{2} \left( t,x \right) \right) ^{\top}\in \mathbb C^2$,
and the 2D NLD equation
\begin{equation}\label{2dNLDE}
  \partial_{t} \Psi \left( t,x,y \right) + \sigma_{1} \partial_{x} \Psi \left( t,x,y \right) + \sigma_{2} \partial_{y} \Psi \left( t,x,y \right) + \mathrm{i} g \left( s \right) \sigma_{3} \Psi \left( t,x,y \right)  = 0,\ \left( x,y \right) \in \mathbb{R}^{2},\ t \geq 0,
\end{equation}
with $\Psi = \Psi \left( t,x,y \right) = \left( \psi_{1} \left( t,x,y \right), \psi_{2} \left( t,x,y \right) \right) ^{\top} \in \mathbb{C}^{2}$.
Here $\mathrm{i} = \sqrt{-1}$, $s = \Psi^{\ast} \sigma_{3} \Psi$, $g \left( s \right) = m - \left( \kappa + 1 \right) \lambda s^{\kappa}$ with $\kappa > 0$,
  the particle mass $m \geq 0$ and the nonnegative parameter $\lambda$,
  the superscripts $\ast$ and $\top$ denote the complex conjugate transpose
  and the vector transpose, respectively, and
 \begin{equation*}
  \sigma_{1} = \left(
  \begin{array}{cc}
    0 & 1 \\
    1 & 0%
  \end{array}
  \right),\ \sigma_{2} = \left(
  \begin{array}{cc}
    0 & -\mathrm{i} \\
    \mathrm{i} & 0%
  \end{array}
  \right),\ \sigma_{3} = \left(
  \begin{array}{cc}
    1 & 0 \\
    0 & -1%
  \end{array}
  \right) \label{Pauli}
\end{equation*}
are three Pauli matrices.
The term $\left( \kappa + 1 \right) \lambda (\Psi^{\ast} \sigma_{3} \Psi)^{\kappa}$
represents the general scalar self-interaction.
When $\kappa = 1$, Eq. \eqref{2dNLDE} reduces to the  Soler model \cite{Soler1970}. 


For the 1D and 2D NLD equations \eqref{1dNLDE} and \eqref{2dNLDE}, assuming that the solutions are smooth enough, we may derive the following proposition.

\begin{proposition}\label{Pro_conservation}
If  $\lim_{|\vec x| \rightarrow + \infty} \left| \Psi \left( t,\vec x \right) \right| = 0$
  holds uniformly for $t \geq 0$, then the charge $Q$ and the energy $E$ are conservative, i.e. $\frac{\mathrm{d}}{\mathrm{d} t} Q \left( t \right) = 0,\ \frac{\mathrm{d}}{\mathrm{d} t} E \left( t \right) = 0$, where
  \begin{equation*}
    Q \left( t \right) = \int \nolimits_{\mathbb{R}^{d}} \rho_{Q} \left( t,\vec x \right) \mathrm{d} \vec x ,\ E \left( t \right) = \int \nolimits_{\mathbb{R}^{d}} \rho_{E} \left( t,\vec x\right) \mathrm{d} \vec x.
  \end{equation*}
Here the charge density $\rho _{Q}:= \Psi^{\ast} \Psi$,
and $\vec x$ and the energy density  are given by
  \begin{align*}
 \mbox{\rm 1D ($d=1$)}\quad & {\vec x=x},\
 \rho _{E} := \mathrm{Im} \left( \Psi^{\ast} \sigma_{1} \partial_{x} \Psi \right) + m \Psi^{\ast} \sigma_{3} \Psi - \lambda \left( \Psi ^{\ast} \sigma_{3} \Psi \right)^{\kappa + 1},\\
   \mbox{\rm 2D  ($d=2$)}\quad & {\vec x=(x,y)}, \
  \rho _{E} := \mathrm{Im} \left( \Psi^{\ast} \sigma_{1} \partial_{x} \Psi + \Psi^{\ast} \sigma_{2} \partial_{y} \Psi \right) + m \Psi^{\ast} \sigma_{3} \Psi - \lambda \left( \Psi ^{\ast} \sigma_{3} \Psi \right)^{\kappa + 1}.
  \end{align*}
\end{proposition}

\begin{proof} In the following, only the case of $d=2$ is considered.

 (i) From   \eqref{2dNLDE}, one has
  \begin{equation*}
    \Psi^{\ast} \partial_{t} \Psi + \Psi^{\ast} \sigma_{1} \partial_{x} \Psi + \Psi^{\ast} \sigma_{2} \partial_{y} \Psi + \mathrm{i} g \left( s \right) s = 0,
  \end{equation*}
  and its complex conjugate form
  \begin{equation*}
    \left( \partial_{t} \Psi^{\ast} \right) \Psi + \left( \partial_{x} \Psi^{\ast} \right) \sigma_{1} \Psi + \left( \partial_{y} \Psi^{\ast} \right) \sigma_{2} \Psi - \mathrm{i} g \left( s \right) s = 0.
  \end{equation*}
  Summing up them gives
  \begin{equation*}
    \partial_{t} \left( \Psi^{\ast} \Psi \right) + \partial_{x} \left( \Psi^{\ast} \sigma_{1} \Psi \right) + \partial_{y} \left( \Psi^{\ast} \sigma_{2} \Psi \right) = 0.
  \end{equation*}
  Integrating it with respect to $\vec x$ yields the charge conservation law $\frac{\mathrm{d}}{\mathrm{d} t} \int_{\mathbb{R}^{d}} \rho _{Q} \mathrm{d} \vec x= 0$ under the hypothesis.

 (ii) From \eqref{2dNLDE}, one also has
  \begin{equation*}
    \left( \partial_{t} \Psi^{\ast} \right) \partial_{t} \Psi + \left( \partial_{t} \Psi^{\ast} \right) \sigma_{1} \partial_{x} \Psi + \left( \partial_{t} \Psi^{\ast} \right) \sigma_{2} \partial_{y} \Psi + \mathrm{i} g \left( s \right) \left( \partial_{t} \Psi^{\ast} \right) \sigma_{3} \Psi = 0.
  \end{equation*}
  Taking the imaginary part and noticing the definition of $g \left( s \right)$ gives
  \begin{equation}\label{imag}
    \mathrm{Im} \left( \left( \partial_{t} \Psi^{\ast} \right) \sigma_{1} \partial_{x} \Psi + \left( \partial_{t} \Psi^{\ast} \right) \sigma_{2} \partial_{y} \Psi \right) + \frac{1}{2} \partial_{t} \left[ m \left( \Psi^{\ast} \sigma_{3} \Psi \right) - \lambda \left( \Psi ^{\ast} \sigma_{3} \Psi \right)^{\kappa + 1} \right] = 0.
  \end{equation}
On the other hand, using the integration by parts and the hypothesis yields
  \begin{equation*}
    \int_{\mathbb{R}} \left( \partial_{x} \Psi^{\ast} \right) \sigma_{1} \partial_{t} \Psi \mathrm{d} x = - \int_{\mathbb{R}} \Psi^{\ast} \sigma_{1} \left( \partial_{tx} \Psi \right) \mathrm{d} x,
  \end{equation*}
 and then
  \begin{equation*}
    2 \int_{\mathbb{R}^{2}} \mathrm{Im} \left( \left( \partial_{t} \Psi^{\ast} \right) \sigma_{1} \partial_{x} \Psi \right) \mathrm{d} x \mathrm{d} y = \frac{\mathrm{d}}{\mathrm{d} t} \int_{\mathbb{R}^{2}} \mathrm{Im} \left( \Psi^{\ast} \sigma_{1} \partial_{x} \Psi \right) \mathrm{d} x \mathrm{d} y.
  \end{equation*}
  Similarly, one has
  \begin{equation*}
    2 \int_{\mathbb{R}^{2}} \mathrm{Im} \left( \left( \partial_{t} \Psi^{\ast} \right) \sigma_{2} \partial_{y} \Psi \right) \mathrm{d} x \mathrm{d} y = \frac{\mathrm{d}}{\mathrm{d} t} \int_{\mathbb{R}^{2}} \mathrm{Im} \left( \Psi^{\ast} \sigma_{2} \partial_{y} \Psi \right) \mathrm{d} x \mathrm{d} y.
  \end{equation*}
Combining them with \eqref{imag} gets $\frac{\mathrm{d}}{\mathrm{d} t} E \left( t \right) = 0$. The proof is completed.
\end{proof}

\begin{remark}
  The proof of Proposition \ref{Pro_conservation} with $\kappa = 1$ and $d=1$ was given in \cite{Shao2006}, but there exists a difference between those proofs of the energy conservation law.
\end{remark}

The  standing and travelling wave solutions of  the 1D NLD equation  can be found
in \cite{Cooper2010,Shao2006,Xu2013,Xu2015}, so that  they are not presented here to avoid repetition. 
For the 2D NLD equation, the standing wave solution was approximately obtained in \cite{Cuevas-Maraver2018}  by using the spectral method, the fixed point method,
and the following ansatz in the polar coordinates
  \begin{equation}\label{ansatz}
    \Psi^{sw} \left( t,r,\theta \right) =   \begin{pmatrix}
      \varphi \left( r,\omega \right) \mathrm{e}^{\mathrm{i} S \theta} \\
      \mathrm{i} \chi \left( r,\omega \right) \mathrm{e}^{\mathrm{i} \left( S + 1 \right) \theta}
    \end{pmatrix}         \mathrm{e}^{- \mathrm{i} \omega t},
  \end{equation}
where $\varphi$ and $\chi$ are two real-valued functions, and $S$ is the vorticity for the first spinor component.
In fact, first substituting \eqref{ansatz} into \eqref{2dNLDE} gives the following
system of ordinary differential equations
\begin{equation}\label{ansatz2}
  \begin{array}{l}
 \left( \frac{\mathrm{d}}{\mathrm{d} r} + \frac{S + 1}{r} \right) \chi +
 (g \left( \hat{s} \right)-\omega) \varphi=0, \\
      \left( \frac{\mathrm{d}}{\mathrm{d} r} - \frac{S}{r} \right) \varphi +( g \left( \hat{s} \right) +\omega)\chi=0,
  \end{array}
\end{equation}
where $\hat{s} = \varphi^{2} - \chi^{2}$, $r > 0$.
Then using the Chebyshev spectral method to approximate
the ODE system \eqref{ansatz2}
gives the  nonlinear algebraic system and the final approximate solutions are
obtained by the iterative method, e.g. the fixed point method. The readers are referred to   Section 3.2.2 of
\cite{Cuevas-Maraver2018} for the details.
Once we have the standing wave solution $\Psi^{sw} \left( t,r,\theta \right)$,
the travelling wave solution  $\Psi^{tw} \left( t,r,\theta \right)$ can be obtained
by using the Lorentz transformation. For example,
under the  Lorentz transformation with the speed of light $c = 1$
and  the relative velocity $v$ in the $x$-direction
\begin{equation*}
    \tilde{t} = \delta \left( t - v x \right), \ \
    \tilde{x} = \delta \left( x - v t \right), \ \
    \tilde{y} = y,
\end{equation*}
the travelling wave solution is gotten  by
\begin{equation*}
  \Psi^{tw} \left( t,x,y \right) = \left(
  \begin{array}{cc}
    \sqrt{\frac{\delta + 1}{2}} & \mathrm{sign} \left( v \right) \sqrt{\frac{\delta - 1}{2}} \\
    \mathrm{sign} \left( v \right) \sqrt{\frac{\delta - 1}{2}} & \sqrt{\frac{\delta + 1}{2}}
  \end{array}
  \right) \Psi^{sw} \left( \tilde{t},\tilde{x},\tilde{y} \right),
\end{equation*}
where  $\delta = \frac{1}{\sqrt{1 - v^{2}}}$ is the Lorentz factor.
Figure \ref{Fig_symmetry} shows   the standing wave solution and the travelling wave solution with $v =0.5$ at $t = 0$, where $\omega$ is taken as $0.12$ and $S=0$.
The right plot  shows clearly that the charge density $\left| \Psi^{tw} \left( 0,x,y \right) \right|^{2}$ loses  symmetry in the $y$-direction.
The lack of symmetry of $\left| \Psi^{tw} \left( t,x,y \right) \right|^{2}$ in the $y$-direction is caused by the above Lorentz transformation and the difference between $\mathrm{e}^{- \mathrm{i} S \theta}$ and $\mathrm{e}^{- \mathrm{i} \left( S + 1 \right) \theta}$ in \eqref{ansatz}.

\begin{figure}[htbp]
  \centering
  \includegraphics[width=0.45\textwidth]{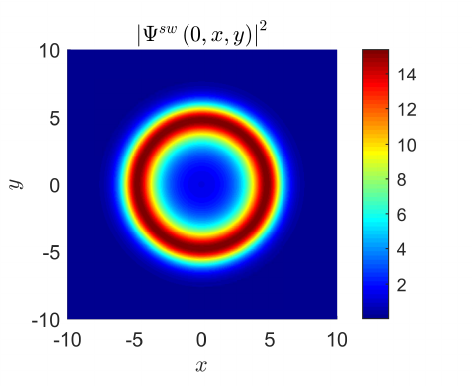}
  \includegraphics[width=0.45\textwidth]{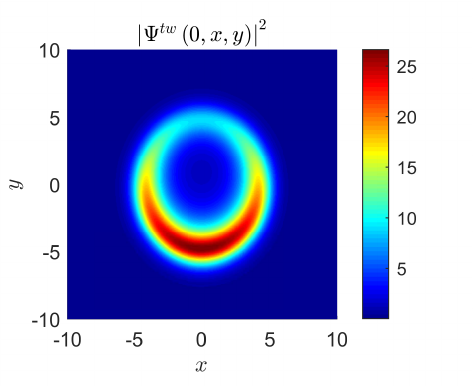} \\
  \caption{The charge densities $\rho_Q$ at $t = 0$.
  Left: the standing wave solution, right: the travelling wave solution ($v=0.5$).}
  \label{Fig_symmetry}
\end{figure}

\section{Numerical methods}
\label{Sec_methods}
This section focuses on developing three high-order accurate DG methods of the 2D  {NLD equation} \eqref{2dNLDE} with three time discretizations on the Cartesian grid.
The 1D RKDG methods have been presented in \cite{Shao2006}, while the 1D
LWDG and TSDG methods can be obtained by removing all the dependence of the wave function $\Psi$ on $y$ from corresponding 2D schemes.

If defining $\bm{u} = \bm{u} \left( t,x,y \right) = \left( u_{1}, u_{2}, u_{3}, u_{4} \right)^{\top} := \left( \psi_{1}^{re}, \psi_{2}^{re}, \psi_{1}^{im}, \psi_{2}^{im} \right)^{\top}$, where $\psi_{\ell}^{re}$ and $\psi_{\ell}^{im}$ denote the real and imaginary parts of $\psi_{\ell}$, respectively, $\ell=1,2$, then
the  {NLD equation} \eqref{2dNLDE} can be rewritten as follows
\begin{equation}\label{2dNLDE4}
  \partial_{t} \bm{u} \left( t,x,y \right) = - \alpha \partial_{x} \bm{u} \left( t,x,y \right) - \beta \partial_{y} \bm{u} \left( t,x,y \right) + g \left( \rho \right) \gamma \bm{u} \left( t,x,y \right),\ \left( x,y \right) \in \Omega,\ t>0, 
\end{equation}
or in the following compact form
\begin{equation}\label{2dNLDE4comp}
  \partial_{t} \bm{u} = -\nabla \cdot \bm{f} \left( \bm{u} \right) + \bm{\mathcal{M}} \left( \bm{u} \right),
\end{equation}
where
$ 
    \bm{f} \left( \bm{u} \right)
    := \left( \alpha \bm{u},\beta \bm{u} \right)$,
    $ (
    \mathcal{M}_{1},    \mathcal{M}_{2},
    \mathcal{M}_{3},
    \mathcal{M}_{4})^{\top}= {\bm{\mathcal{M}}} \left( \bm{u} \right)
    := g \left( \rho \right) \gamma \bm{u}$,
    $\rho= u_{1}^{2} + u_{3}^{2} - u_{2}^{2} - u_{4}^{2}$,
and
\begin{align*}
  \alpha &=
  \begin{pmatrix}
    \sigma_{1} &  0\\
    0 & \sigma_{1}
  \end{pmatrix}, \ \beta =
  \begin{pmatrix}
   0  & \mathrm{i} \sigma_{2} \\
    - \mathrm{i} \sigma_{2} & 0
  \end{pmatrix}, \ \gamma =
  \begin{pmatrix}
    0 & \sigma_{3} \\
    - \sigma_{3} & 0
  \end{pmatrix}.
\end{align*}
%

\subsection{RKDG method}\label{sect3.1:RKDG}

Let $\mathcal{T}_{h}$ be a rectangular partition of the 2D domain $\Omega$ and for each element
 $\mathcal{K} \in \mathcal{T}_{h}$,
  $\mathcal{P}^{q} \left( \mathcal{K} \right)$ represent the space of the real-valued polynomials on $\mathcal{K}$ of degree at most $q$. The RKDG method \cite{Cockburn1989a,Cockburn1998,Shao2006}
  is to seek first each component of the approximate solution $\bm{u}_{h} \left( t,x,y \right) = \left( u _{1,h}, u _{2,h}, u _{3,h}, u _{4,h} \right) ^{\top}$ for any $t$ in the function space
\begin{equation*}
  \mathcal{V}_{h} := \left\{  \phi\in L^2(\Omega) :~ \phi \left( x,y \right) \in \mathcal{P}^{q} \left( \mathcal{K} \right),\ \left( x,y \right) \in \mathcal{K},\ \forall \mathcal{K} \in \mathcal{T}_{h} \right\},
\end{equation*}
such that for each component of any $\bm{v}_{h}$ in $\mathcal{V}_{h}$, one has
\begin{equation}\label{2dweak_rk}
  \int_{\mathcal{K}} \left( \partial_{t} \bm{u}_{h} \right) \circ \bm{v}_{h} \mathrm{d} x \mathrm{d} y = - \sum \limits_{e \in \partial \mathcal{K}} \int_{e} 
  \widehat{\bm{h}}_{e\mathcal{K}} \left(\bm{u}_{h}^{-},\bm{u}_{h}^{+} \right)
  \circ \bm{v}_{h}^{-} \mathrm{d} S
   + \int_{\mathcal{K}} \left[ \bm{f} \left( \bm{u}_{h} \right)   \nabla \bm{v}_{h} + \bm{\mathcal{M}} \left( \bm{u}_{h} \right) \circ \bm{v}_{h} \right] \mathrm{d} x \mathrm{d} y,
\end{equation}
where  
$\bm{f} \left( \bm{u}_{h} \right)   \nabla \bm{v}_{h}
     = \left(
    \bm{f}_{1,h} \cdot \nabla v_{1,h}, \cdots,
    \bm{f}_{4,h} \cdot \nabla v_{4,h}
  \right)^{\top}$,
  $\bm{f}_{\ell}$ is the $\ell$th row vector of $\bm{f}$, $\ell=1,\cdots,4$,
``$\circ$'' represents the Hadamard product, $\partial \mathcal{K}$ denotes the boundary of $\mathcal{K}$, and $\widehat{\bm{h}}_{e\mathcal{K}} \left(\bm{u}_{h}^{-},\bm{u}_{h}^{+} \right)$ is the two-point numerical flux approximating
the flux $\bm{f} \left( \bm{u}_{h} \left( t,x,y \right) \right) \bm{n}_{e\mathcal{K}}^{\top}$ {with $\bm{n}_{e\mathcal{K}}=(n_{e\mathcal{K},1},n_{e\mathcal{K},2})$ the outward unit normal to the edge $e$ of the element $\mathcal{K}$}.
The numerical flux $\widehat{\bm{h}}_{e\mathcal{K}}$ satisfies
\begin{equation}\label{hhat_relation}
  \widehat{\bm{h}}_{e\mathcal{K}} + \widehat{\bm{h}}_{e\mathcal{K}^{\prime}} = 0,
\end{equation}
and  may be chosen as the following Lax-Friedrichs type flux \cite{Shao2006}
\begin{equation}\label{LFflux}
  \widehat{\bm{h}}_{e\mathcal{K}}^{\mbox{\tiny LF}} \left( \bm{u}_{h}^{-},\bm{u}_{h}^{+} \right) = \frac{1}{2} \left[ \bm{f} \left( \bm{u}_{h}^{-} \right) \bm{n}_{e\mathcal{K}}^{\top}  + \bm{f} \left( \bm{u}_{h}^{+} \right) \bm{n}_{e\mathcal{K}}^{\top}  - \left( \bm{u}_{h}^{+} - \bm{u}_{h}^{-} \right) \right].
\end{equation}
Here $\bm{u}_{h}^{\pm} \left( t,x,y \right)$ are the limiting values of $\bm{u}_{h}$ obtained from the interior ($-$) and the exterior ($+$) of $\mathcal{\mathcal{K}}$, i.e.
\begin{equation*}
  \bm{u}_{h}^{-} \left( t,x,y \right) = \lim \limits_{\left( \tilde{x},\tilde{y} \right) \rightarrow \left( x,y \right), \left( \tilde{x},\tilde{y} \right) \in \mathcal{K}} \bm{u}_{h} \left( t,\tilde{x},\tilde{y} \right),
\end{equation*}
\begin{equation*}
  \bm{u}_{h}^{+} \left( t,x,y \right) = \left\{
  \begin{array}{ll}
    \bm{\gamma}_{h} \left( t,x,y \right), & \mbox{if} \left( x,y \right) \in \partial \Omega, \\
    \lim \limits_{\left( \tilde{x},\tilde{y} \right) \rightarrow \left( x,y \right), \left( \tilde{x},\tilde{y} \right) \in \mathcal{K}^{\prime}} \bm{u}_{h} \left( t,\tilde{x},\tilde{y} \right), & \mbox{otherwise},
  \end{array}
  \right.
\end{equation*}
where $\mathcal{K}^{\prime}$ is the neighboring element of $\mathcal{K}$ by
a common edge $e$, as shown in Figure \ref{Fig_ele2}, $\bm{\gamma}_{h} \left( t,x,y \right)$ is
the discrete boundary value of $\bm{u}_{h}$.

\begin{figure}[htbp]
  \centering
  \includegraphics[width=0.45\textwidth]{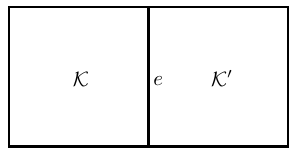} \\
  \caption{A schematic diagram of two elements $\mathcal{K}$ and $\mathcal{K}^{\prime}$ with $\mathcal{K} \bigcap \mathcal{K}^{\prime} = e$.}\label{Fig_ele2}
\end{figure}

Similar to the 1D case \cite{Shao2006}, we can establish the following entropy inequality or $L^{2}$ stability 
of the 2D semi-discrete DG method \eqref{2dweak_rk}, which implies that the discrete total charge
$Q_{h} \left( t \right) = \int_{\Omega} \left| \bm{u}_{h} \right|^{2} \mathrm{d}x \mathrm{d}y = \int_{\Omega} \sum_{\ell=1}^{4} \left| u_{\ell,h} \right|^{2} \mathrm{d}x \mathrm{d}y$
dose not increase with respect to $t$.
\begin{proposition}\label{Pro_entropy}
  If  $Q_{h} \left( 0 \right) < +\infty$, then
  under the homogeneous Dirichlet boundary conditions ($u_{\ell,h}=0$ on $\partial \Omega$), the solution to the scheme \eqref{2dweak_rk} and \eqref{LFflux}  satisfies $\frac{\mathrm{d}}{\mathrm{d}t} Q_{h} \left( t \right) \leq 0$, or $Q_{h} \left( t \right) \leq Q_{h} \left( 0 \right)$ for any $t \geq 0$.
\end{proposition}

\begin{proof}
 For $\ell=1,2,3,4$, because $v_{\ell,h} \in \mathcal{V}_{h}$ is arbitrary,  we may choose $v_{\ell,h}$ in \eqref{2dweak_rk} as $u_{\ell,h}$, and  then  have
  \begin{align*}
    \int_{\mathcal{K}} \left( \partial_{t} u_{\ell,h} \right) u_{\ell,h} \mathrm{d} x \mathrm{d} y = - \sum \limits_{e \in \partial \mathcal{K}} \int_{e} \widehat{h}_{\ell,e\mathcal{K}} u_{\ell,h}^{-} \mathrm{d} S
     + \int_{\mathcal{K}} \left( \bm{f}_{\ell,h} \cdot \nabla u_{\ell,h} + \mathcal{M}_{\ell,h} u_{\ell,h} \right) \mathrm{d} x \mathrm{d} y.
  \end{align*}
  Summing up the above four equations 
    gets
  \begin{equation}\label{P1eq1}
    \frac{\mathrm{d}}{\mathrm{d} t} \int_{\mathcal{K}} \sum_{\ell=1}^{4} u_{\ell,h}^{2} \mathrm{d}x \mathrm{d}y = - \sum_{e \in \partial \mathcal{K}} \int_{e} \sum_{\ell=1}^{4} \left( 2 \widehat{h}_{\ell,e\mathcal{K}} u_{\ell,h}^{-} - \left( \bm{f}_{\ell,h}^{-} \cdot \bm{n}_{e\mathcal{K}} \right) u_{\ell,h}^{-} \right) \mathrm{d} S.
  \end{equation}
  For the edge $e$, see Figure \ref{Fig_ele2}, by noticing \eqref{LFflux}, \eqref{hhat_relation}, and $\bm{n}_{e\mathcal{K}} = -\bm{n}_{e\mathcal{K}^{\prime}}$, we have
  \begin{equation*}
    \begin{aligned}
      &\ \sum_{\ell=1}^{4} \left( 2 \widehat{h}_{\ell,e\mathcal{K}} u_{\ell,h,\mathcal{K}}^{-} - \left( \bm{f}_{\ell,\mathcal{K}}^{-} \cdot \bm{n}_{e\mathcal{K}} \right) u_{\ell,h,\mathcal{K}}^{-} \right) + \left( 2 \widehat{h}_{\ell,e\mathcal{K}^{\prime}} u_{\ell,h,\mathcal{K}^{\prime}}^{-} - \left( \bm{f}_{\ell,\mathcal{K}^{\prime}}^{-} \cdot \bm{n}_{e\mathcal{K}^{\prime}} \right) u_{\ell,h,\mathcal{K}^{\prime}}^{-} \right) \\
      = &\ \sum_{\ell=1}^{4} \left( 2 \widehat{h}_{\ell,e\mathcal{K}} u_{\ell,h,\mathcal{K}}^{-} - \left( \bm{f}_{\ell,\mathcal{K}}^{-} \cdot \bm{n}_{e\mathcal{K}} \right) u_{\ell,h,\mathcal{K}}^{-} \right) - \left( 2 \widehat{h}_{\ell,e\mathcal{K}} u_{\ell,h,\mathcal{K}}^{+} - \left( \bm{f}_{\ell,\mathcal{K}}^{+} \cdot \bm{n}_{e\mathcal{K}} \right) u_{\ell,h,\mathcal{K}}^{+} \right) \\
      = &\ \sum_{\ell=1}^{4} \left( 2 \widehat{h}_{\ell,e\mathcal{K}} \left( u_{\ell,h,\mathcal{K}}^{-} - u_{\ell,h,\mathcal{K}}^{+} \right) - \left( \bm{f}_{\ell,\mathcal{K}}^{-} \cdot \bm{n}_{e\mathcal{K}} \right) u_{\ell,h,\mathcal{K}}^{-} + \left( \bm{f}_{\ell,\mathcal{K}}^{+} \cdot \bm{n}_{e\mathcal{K}} \right) u_{\ell,h,\mathcal{K}}^{+} \right) \\
      = &\ \sum_{\ell=1}^{4} \big[ \left( \bm{f}_{\ell,\mathcal{K}}^{-} \cdot \bm{n}_{e\mathcal{K}} + \bm{f}_{\ell,\mathcal{K}}^{+} \cdot \bm{n}_{e\mathcal{K}} + u_{\ell,h,\mathcal{K}}^{-} - u_{\ell,h,\mathcal{K}}^{+} \right) \left( u_{\ell,h,\mathcal{K}}^{-} - u_{\ell,h,\mathcal{K}}^{+} \right)\\
      &\ - \left( \bm{f}_{\ell,\mathcal{K}}^{-} \cdot \bm{n}_{e\mathcal{K}} \right) u_{\ell,h,\mathcal{K}}^{-} + \left( \bm{f}_{\ell,\mathcal{K}}^{+} \cdot \bm{n}_{e\mathcal{K}} \right) u_{\ell,h,\mathcal{K}}^{+} \big] \\
      = &\ \sum_{p=1}^{4} \left[ \left( \bm{f}_{\ell,\mathcal{K}}^{+} \cdot \bm{n}_{e\mathcal{K}} \right) u_{\ell,h,\mathcal{K}}^{-} - \left( \bm{f}_{\ell,\mathcal{K}}^{-} \cdot \bm{n}_{e\mathcal{K}} \right) u_{\ell,h,\mathcal{K}}^{+} + \left( u_{\ell,h,\mathcal{K}}^{-} - u_{\ell,h,\mathcal{K}}^{+} \right)^{2} \right],
      \end{aligned}
  \end{equation*}
  here $u_{\ell,h,\mathcal{K}}^{-}$ and $u_{\ell,h,\mathcal{K}^{\prime}}^{-}$ are the {limiting} values of $u_{\ell,h}$
   from the interiors ($-$) of $\mathcal{K}$ and $\mathcal{K}^{\prime}$, respectively,
 in order to distinguish the values of $u_{\ell,h}$ from the different elements.
%
  Since $\sum \limits_{\ell=1}^{4} \left[ \left( \bm{f}_{\ell}^{+} \cdot \bm{n}_{e\mathcal{K}} \right) u_{\ell,h}^{-} - \left( \bm{f}_{\ell}^{-} \cdot \bm{n}_{e\mathcal{K}} \right) u_{\ell,h}^{+} \right] = 0$, we have
  \begin{equation*}
    \sum_{\ell=1}^{4} \left\{\left( 2 \widehat{h}_{\ell,e\mathcal{K}} u_{\ell,h,\mathcal{K}}^{-} - \left( \bm{f}_{\ell,\mathcal{K}}^{-} \cdot \bm{n}_{e\mathcal{K}} \right) u_{\ell,h,\mathcal{K}}^{-} \right) + \left( 2 \widehat{h}_{\ell,e\mathcal{K}^{\prime}} u_{\ell,h,\mathcal{K}^{\prime}}^{-} - \left( \bm{f}_{\ell,\mathcal{K}^{\prime}}^{-} \cdot \bm{n}_{e\mathcal{K}^{\prime}} \right) u_{\ell,h,\mathcal{K}^{\prime}}^{-} \right)\right\} \geq 0.
  \end{equation*}
 Combining them, summing up \eqref{P1eq1} on all $\mathcal{K}$, and noticing the boundary condition yields
  \begin{equation*}
    \frac{\mathrm{d}}{\mathrm{d}t} Q_{h} \left( t \right) = \sum \limits_{\mathcal{K} \in \mathcal{T}_{h}} \int_{\mathcal{K}} \sum_{\ell=1}^{4} u_{\ell,h}^{2} \mathrm{d}x \mathrm{d}y = - \sum \limits_{\mathcal{K} \in \mathcal{T}_{h}} \sum \limits_{e \in \partial \mathcal{K}} \int_{e} \sum \limits_{\ell=1}^{4} \left( 2 \widehat{{h}}_{\ell,e\mathcal{K}} u_{\ell,h}^{-} - \left( \bm{f}_{\ell}^{-} \cdot \bm{n}_{e\mathcal{K}} \right) u_{\ell,h}^{-} \right) \mathrm{d} S \leq 0.
  \end{equation*}
  This completes the proof.
\end{proof}

For  the Cartesian grid, 
following \cite{Cockburn1989a}, we choose the following local basis functions
\begin{equation*}
  \begin{array}{l}
    v_{\mathcal{K}}^{\left( 0 \right)} \left( x,y \right) = 1,\ v_{\mathcal{K}}^{\left( 1 \right)} \left( x,y \right) = x - x_{j},\ v_{\mathcal{K}}^{\left( 2 \right)} \left( x,y \right) = y - y_{k}, \\
    v_{\mathcal{K}}^{\left( 3 \right)} \left( x,y \right) = \left( x - x_{j} \right)^{2} - \frac{\Delta x^{2}}{12},\ v_{\mathcal{K}}^{\left( 4 \right)} \left( x,y \right) = \left( x - x_{j} \right) \left( y - y_{k} \right), \\
    v_{\mathcal{K}}^{\left( 5 \right)} \left( x,y \right) = \left( y - y_{k} \right)^{2} - \frac{\Delta y^{2}}{12},\ \cdots,
  \end{array}
\end{equation*}
where $\Delta x$ and $\Delta y$ are the spatial step sizes in $x$
 and $y$ directions, respectively.
 Then  the DG approximate solutions can be expressed as
\begin{equation}\label{2dNumsol}
  \bm{u}_{h} \left( t,x,y \right) = \sum_{l=0}^{\frac{\left( q+1 \right)\left( q+2 \right)}{2} - 1} \bm{u}_{\mathcal{K}}^{\left( l \right)} \left( t \right) v_{\mathcal{K}}^{\left( l \right)} \left( x,y \right),\ \left( x,y \right) \in \mathcal{K},
\end{equation}
where $\bm{u}_{\mathcal{K}}^{\left( l \right)} \left( t \right)$ are the degrees of freedom to be determined.
Substituting them into \eqref{2dweak_rk} gives
the semi-discrete DG scheme  for the degrees of freedom
\begin{equation}\label{2dODEs}
  \frac{\mathrm{d}}{\mathrm{d}t} \bm{u}_{\mathcal{K}}^{\left( l \right)} \left( t \right) = \frac{1}{a_{\mathcal{K}}^{\left( l \right)}} \left[ - \sum_{e \in \partial \mathcal{K}} \int_{e} \widehat{\bm{h}}_{e\mathcal{K}} v_{\mathcal{K}}^{\left( l \right)} \mathrm{d} S + \int_{\mathcal{K}} \bm{f} \left( \bm{u}_{h} \right) \left( \nabla v_{\mathcal{K}}^{\left( l \right)} \right)^{\top} + \bm{\mathcal{M}} \left( \bm{u}_{h} \right) v_{\mathcal{K}}^{\left( l \right)} \mathrm{d} x \mathrm{d} y \right],
\end{equation}
where $l=0,1,\cdots,\frac{\left( q+1 \right) \left( q+2 \right)}{2} - 1$
and $a _{\mathcal{K}}^{\left( l \right)} = \int_{\mathcal{K}} \left( v _{\mathcal{K}}^{\left( l \right)} \left( x,y \right) \right)^{2} \mathrm{d} x \mathrm{d} y$.
The integrals in \eqref{2dODEs} will be calculated by
the $\left( q + 1 \right)$-point Gauss-Legendre
 quadrature in each coordinate direction.

To further discretize the system \eqref{2dODEs} in time,
 let us first rewrite it into a compact form
\begin{equation*}\label{ODE}
  \frac{\mathrm{d}}{\mathrm{d} t} \Phi \left( t \right) = \bm{\mathcal{L}} \left( \Phi \left( t \right) \right),\ t>0,
\end{equation*}
and then approximate it by utilizing the
 fourth-order non-TVD RK method
\begin{equation*}
  \left\{
  \begin{array}{l}
    \bm{\phi}^{\left(1\right)} = \Phi \left( t \right) + \frac{\tau}{2} \bm{\mathcal{L}} \left( \Phi \left( t \right) \right), \\
    \bm{\phi}^{\left(2\right)} = \Phi \left( t \right) + \frac{\tau}{2} \bm{\mathcal{L}} \left( \bm{\phi}^{\left(1\right)} \right), \\
    \bm{\phi}^{\left(3\right)} = \Phi \left( t \right) + \tau \bm{\mathcal{L}} \left( \bm{\phi}^{\left(2\right)} \right), \\
    \Phi \left( t+\tau \right) = \frac{1}{3} \left( \bm{\phi}^{\left(1\right)} + 2\bm{\phi}^{\left(2\right)}+ \bm{\phi}^{\left(3\right)} - \Phi \left( t \right) + \frac{\tau}{2} \bm{\mathcal{L}} \left( \bm{\phi}^{\left(3\right)} \right) \right),
  \end{array}
  \right.
\end{equation*}
or the third-order TVD RK method \cite{Shu1988}
\begin{equation*}
  \left\{
  \begin{array}{l}
    \bm{\phi}^{\left(1\right)} = \Phi \left( t \right) + \tau \bm{\mathcal{L}} \left( \Phi \left( t \right) \right), \\
    \bm{\phi}^{\left(2\right)} = \frac{1}{4} \left( 3 \Phi \left( t \right) + \bm{\phi}^{\left(1\right)} + \tau \bm{\mathcal{L}} \left( \bm{\phi}^{\left(1\right)} \right) \right), \\
    \Phi \left( t+\tau \right) = \frac{1}{3} \left( \Phi \left( t \right) + 2 \bm{\phi}^{\left(2\right)} + 2 \tau \bm{\mathcal{L}} \left( \bm{\phi}^{\left(2\right)} \right) \right),
  \end{array}
  \right.
\end{equation*}
where $\tau$ denotes the time step size.
 In our experiments,
 the above fourth-order RK method is used. %

\subsection{LWDG and TSDG methods}

This section proposes two other DG methods, i.e. the LWDG and TSDG methods.
Different from the above RKDG method, they are derived by first giving the one-stage fourth-order Lax-Wendroff type and the two-stage fourth-order time discretizations of the NLD equations, respectively, and then discretizing the first- and higher-order spatial derivatives by using the spatial DG approximation.
To do that,
assume that the solutions $\bm{u}$ are sufficiently smooth
 and let us not write spatial arguments for the time being so that
  the 2D NLD equation \eqref{2dNLDE4} is rewritten into the following form
\begin{equation}\label{NLDE_ODE}
  \frac{\mathrm{d}}{\mathrm{d} t} \bm{u} \left( t \right) = \bm{\mathcal{N}} \left( \bm{u} \left( t \right) \right),\ t > 0,
\end{equation}
with $\left( \mathcal{N}_{1},\mathcal{N}_{2},\mathcal{N}_{3},\mathcal{N}_{4} \right)^{\top} = \bm{\mathcal{N}} \left( \bm{u} \left( t \right) \right) := - \nabla \cdot \bm{f} \left( \bm{u} \left( t \right) \right) +  \bm{\mathcal{M}} \left( \bm{u} \left( t \right) \right)$.

\subsubsection{Four{th}-order time discretizations}

Using the Taylor series expansion in $t$ gives
\begin{equation}\label{Taylor}
  \bm{u} \left( t + \tau \right) = \bm{u} + \tau \bm{u}_{t} + \frac{\tau^{2}}{2} \bm{u}_{tt} + \frac{\tau^{3}}{6} \bm{u}_{ttt} + \frac{\tau^{4}}{24} \bm{u}_{tttt} + \mathcal{O} \left( \tau^{5} \right),
\end{equation}
where $\bm{u}_{t} = \frac{\mathrm{d} \bm{u}}{\mathrm{d} t}$. Utilizing \eqref{NLDE_ODE} yields
\begin{equation}\label{EQ:LW}
  \bm{u} \left( t + \tau \right) = \bm{u} + \tau \bm{\mathcal{N}} \left( \bm{u} \right) + \frac{\tau^{2}}{2} \bm{\mathcal{N}}_{t} \left( \bm{u} \right) + \frac{\tau^{3}}{6} \bm{\mathcal{N}}_{tt} \left( \bm{u} \right) + \frac{\tau^{4}}{24} \bm{\mathcal{N}}_{ttt} \left( \bm{u} \right) + \mathcal{O} \left( \tau^{5} \right),
\end{equation}
which will give a fourth-order accurate Lax-Wendroff type time discretization
by omitting the term $\mathcal{O} \left( \tau^{5} \right)$ and replacing $\bm{u}$
with the approximate solution.



The two-stage fourth-order accurate time discretizations  are recently studied
in \cite{Li2016,Yuan2020,Yuan2020b} and successfully applied to solving the hyperbolic partial differential equations.

Following \cite{Yuan2020}, \eqref{Taylor} can be written as
\begin{equation*}
  \bm{u} \left( t + \tau \right) = \bm{u} + \tau \bm{u}_{t} + \frac{\vartheta \tau^{2}}{2} \bm{u}_{tt} + \frac{\left( 1-\vartheta \right) \tau^{2}}{2} \left( \bm{u} + \frac{\tau \bm{u}_{t}}{3 \left( 1 - \vartheta \right)} + \frac{\tau^{2} \bm{u}_{tt}}{12 \left( 1 - \vartheta \right)} \right)_{tt} + \mathcal{O} \left( \tau^{5} \right),
\end{equation*}
where $\vartheta \neq 1$. Thanks to \eqref{NLDE_ODE}, one has
\begin{equation}\label{Taylor_TS}
  \bm{u} \left( t + \tau \right) = \bm{u} + \tau \bm{\mathcal{N}} \left( \bm{u} \right) + \frac{\vartheta \tau^{2}}{2} \bm{\mathcal{N}}_{t} \left( \bm{u} \right) + \frac{\left( 1 - \vartheta \right) \tau^{2}}{2} \bm{u}_{tt}^{\ast} + \mathcal{O} \left( \tau^{5} \right),
\end{equation}
where
\begin{equation}\label{Psistar}
  \bm{u}^{\ast} := \bm{u} + \frac{\tau  \bm{\mathcal{N}} \left( \bm{u} \right)}{3 \left( 1 - \vartheta \right)} + \frac{\tau^{2} \bm{\mathcal{N}}_{t} \left( \bm{u} \right)}{12 \left( 1 - \vartheta \right)}.
\end{equation}
The component form of  \eqref{Psistar} reads
\begin{equation}\label{Psistarcomponent}
  u_{\ell}^{\ast} = u_{\ell} + \frac{\tau \mathcal{N}_{\ell} \left( \bm{u} \right)}{3 \left( 1 - \vartheta \right)} + \frac{\tau^{2} \mathcal{N}_{\ell,t} \left( \bm{u} \right)}{12 \left( 1 - \vartheta \right)},\ \ell =1,2,3,4,
\end{equation}
where $(\bullet)_{\ell,t}$ denotes $\frac{\mathrm{d} (\bullet)_{\ell}}{\mathrm{d} t}$.

The rest of the task is to approximate $\bm{u}_{tt}^{\ast}$ in \eqref{Taylor_TS}. {From \eqref{Psistarcomponent},} one has
\begin{equation*}
  u_{\ell,t}^{\ast} = \mathcal{N}_{\ell} \left( \bm{u} \right) + \frac{\tau \mathcal{N}_{\ell,t} \left( \bm{u} \right)}{3 \left( 1 - \vartheta \right)} + \frac{\tau^{2}}{12 \left( 1 - \vartheta \right)} \sum_{j=1}^{4} \left( \frac{\partial \mathcal{N}_{\ell,t} \left( \bm{u} \right)}{\partial u_{j}} \mathcal{N}_{j} \left( \bm{u} \right) \right),\ \ell =1,2,3,4.
\end{equation*}
Moreover, one has
\begin{align}
  u_{\ell,tt}^{\ast} =&\ \mathcal{N}_{\ell,t} \left( \bm{u} \right) + \frac{\tau}{3 \left( 1 - \vartheta \right)} \sum_{j=1}^{4} \left( \frac{\partial \mathcal{N}_{\ell,t} \left( \bm{u} \right)}{\partial u_{j}} \mathcal{N}_{j} \left( \bm{u} \right) \right) \notag \\
  &\ + \frac{\tau^{2}}{12 \left( 1 - \vartheta \right)} \sum_{j,k=1}^{4} \left( \frac{\partial^{2} \mathcal{N}_{\ell,t} \left( \bm{u} \right)}{\partial u_{j} \partial u_{k}} \mathcal{N}_{k} \left( \bm{u} \right) \mathcal{N}_{j} \left( \bm{u} \right) + \frac{\partial \mathcal{N}_{\ell,t} \left( \bm{u} \right)}{\partial u_{j}} \frac{\partial \mathcal{N}_{j} \left( \bm{u} \right)}{\partial u_{k}} \mathcal{N}_{k} \left( \bm{u} \right) \right). \label{Psittstar}
\end{align}
On the other hand, for $\ell =1,\cdots,4$, using the Taylor series expansion gives
\begin{equation*}
  \mathcal{N}_{\ell,t} \left( \bm{u}^{\ast} \right) = \mathcal{N}_{\ell,t} \left( \bm{u} \right) + \sum_{j=1}^{4} \frac{\partial \mathcal{N}_{\ell,t} \left( \bm{u} \right)}{\partial u_{j}} \left( u_{j}^{\ast} - u_{j} \right) + \frac{1}{2} \sum_{j,k=1}^{4} \frac{\partial^{2} \mathcal{N}_{\ell,t} \left( \bm{u} \right)}{\partial u_{j} \partial u_{k}} \left( u_{j}^{\ast} - u_{j} \right) \left( u_{k}^{\ast} - u_{k} \right) + \mathcal{O} \left( \tau^{3} \right).
\end{equation*}
Comparing it to \eqref{Psittstar} {and noticing \eqref{Psistarcomponent}} yields
\begin{align*}
  u_{\ell,tt}^{\ast} - \mathcal{N}_{\ell,t} \left( \bm{u}^{\ast} \right) =&\ \left( \frac{\tau^{2}}{12 \left( 1 - \vartheta \right)} - \frac{\tau^{2}}{18 \left( 1 - \vartheta \right)^{2}} \right) \sum_{j,k=1}^{4} \left( \frac{\partial^{2} \mathcal{N}_{\ell,t} \left( \bm{u} \right)}{\partial u_{j} \partial u_{k}} \mathcal{N}_{k} \left( \bm{u} \right) \mathcal{N}_{j} \left( \bm{u} \right) \right) + \mathcal{O} \left( \tau^{3} \right) \notag \\
  =&\ \frac{\tau^{2}}{18 \left( 1 - \vartheta \right)^{2}} \left( \frac{3 \left( 1 - \vartheta \right)}{2} - 1 \right) \sum_{j,k=1}^{4} \left( \frac{\partial^{2} \mathcal{N}_{\ell,t} \left( \bm{u} \right)}{\partial u_{j} \partial u_{k}} \mathcal{N}_{k} \left( \bm{u} \right) \mathcal{N}_{j} \left( \bm{u} \right) \right) + \mathcal{O} \left( \tau^{3} \right).
\end{align*}
 Hence, if
\begin{equation}\label{TSrelation2}
  \vartheta =\frac13+{\mathcal O}(\hat{\tau}),
\end{equation}
where $\hat{\tau} := \tau^{\nu}$ and $\nu \geq 1$, then
\begin{equation*}
  \bm{u}_{tt}^{\ast} = \bm{\mathcal{N}}_{t} \left( \bm{u}^{\ast} \right) + \mathcal{O} \left( \tau^{3} \right).
\end{equation*}
Substituting it into \eqref{Taylor_TS} gives
\begin{equation*}
  \bm{u} \left( t + \tau \right) = \bm{u} + \tau \bm{\mathcal{N}} + \frac{\vartheta \tau^{2}}{2} \bm{\mathcal{N}}_{t} + \frac{\left( 1 - \vartheta \right) \tau^{2}}{2} \bm{\mathcal{N}}_{t} \left( \bm{u}^{\ast} \right) + \mathcal{O} \left( \tau^{5} \right).
\end{equation*}

Based on the above discussion, an explicit two-stage fourth-order accurate time discretization can be given as follows.
\begin{description}
  \item[Stage 1.] Calculate the intermediate value
  \begin{equation}\label{Taylor_TS_1}
    \bm{u}^{\ast} = \bm{u} \left( t \right) + \frac{\tau}{3 \left( 1 - \vartheta \right)} \bm{\mathcal{N}} \left( \bm{u} \left( t \right) \right) + \frac{\tau^{2}}{12 \left( 1 - \vartheta \right)} \bm{\mathcal{N}}_{t} \left( \bm{u} \left( t \right) \right),
  \end{equation}
  \item[Stage 2.] Compute the solution at time level $t + \tau$, i.e.,
  \begin{equation}\label{Taylor_TS_2}
    \bm{u} \left( t+\tau \right) = \bm{u} \left( t \right) + \tau \bm{\mathcal{N}} \left( \bm{u} \left( t \right) \right) + \frac{\vartheta \tau^{2}}{2} \bm{\mathcal{N}}_{t} \left( \bm{u} \left( t \right) \right) + \frac{\left( 1 - \vartheta \right) \tau^{2}}{2}  \bm{\mathcal{N}}_{t} \left( \bm{u}^{\ast} \right),
  \end{equation}
\end{description}
where $\vartheta = \vartheta \left( \hat{\tau} \right)\neq 1$ satisfies  \eqref{TSrelation2}. 
A more general discussion of the two-stage fourth-order accurate time discretization
can be found in \cite{Yuan2020b}.
Without loss of generality, $\vartheta$ is taken as $\frac{1}{3}$ in the   numerical experiments in Section \ref{Sec_results}.

\subsubsection{Spatial discretizations}

This subsection gives the LWDG and TSDG methods based on the
one-stage fourth-order Lax-Wendroff type time discretization   and the two-stage fourth-order time discretization  for the NLD equation \eqref{2dNLDE4}.

Applying the previous fourth-order Lax-Wendroff type time discretization to the NLD equation \eqref{2dNLDE4} gives
\begin{equation}\label{Taylorlw_NLDE}
  \bm{u} \left( t+\tau,x,y \right) = \bm{u} \left( t,x,y \right) - \tau \left( \nabla \cdot \bm{\mathcal{F}} \left( \bm{u} \left( t,x,y \right) \right) - \bm{\mathcal{G}} \left( \bm{u} \left( t,x,y \right) \right) \right),
\end{equation}
where
\begin{align}
  \bm{\mathcal{F}} \left( \bm{u} \right) &= \bm{f} \left( \bm{u} \right) + \frac{\tau}{2} \partial_{t} \bm{f} \left( \bm{u} \right) + \frac{\tau^{2}}{6} \partial_{tt} \bm{f} \left( \bm{u} \right) + \frac{\tau^{3}}{24} \partial_{ttt} \bm{f} \left( \bm{u} \right), \label{F_u} \\
  \bm{\mathcal{G}} \left( \bm{u} \right) &=  \bm{\mathcal{M}} \left( \bm{u} \right) + \frac{\tau}{2} \partial_{t}  \bm{\mathcal{M}} \left( \bm{u} \right) + \frac{\tau^{2}}{6} \partial_{tt}  \bm{\mathcal{M}} \left( \bm{u} \right) + \frac{\tau^{3}}{24} \partial_{ttt}  \bm{\mathcal{M}} \left( \bm{u} \right). \label{G_u}
\end{align}
%
  To calculate $\bm{\mathcal{F}}$ and $\bm{\mathcal{G}}$, one needs to compute high-order time derivatives of $\bm{f}$ and $ \bm{\mathcal{M}}$. With the help of   \eqref{2dNLDE4comp}, 
  those time derivatives can be replaced with the spatial derivatives of $\bm{u}$, {see \ref{Appendix_details} for details.}

The LWDG method is to seek the approximate solution $\bm{u}_{h} \left( t,x,y \right)$ with each component $u_{\ell,h} \in \mathcal{V}_{h}$   for any $t$, such that for $v_{\ell,h} \in \mathcal{V}_{h}$, $\bm{u}_{h} \left( t,x,y \right)$ satisfies
\begin{align*}
  \int_{\mathcal{K}} \bm{u}_{h} \left( t+\tau,x,y \right) \circ \bm{v}_{h} \mathrm{d} x \mathrm{d} y =&\ \int_{\mathcal{K}} \bm{u}_{h} \circ \bm{v}_{h} \mathrm{d} x \mathrm{d} y - \tau \sum_{e \in \partial \mathcal{K}} \int_{e} \tilde{\bm{h}}_{e\mathcal{K}} \circ \bm{v}_{h} \mathrm{d} S \\
  &\ + \tau \int_{\mathcal{K}} \left( \bm{\mathcal{F}} \left( \bm{u}_{h} \right) \nabla \bm{v_{h}} + \bm{\mathcal{G}} \left( \bm{u}_{h} \right) \circ \bm{v}_{h} \right) \mathrm{d} x \mathrm{d} y,
\end{align*}
where the numerical flux
$\widetilde{\bm{h}}_{e\mathcal{K}} = \widetilde{\bm{h}}_{e\mathcal{K}} \left( \bm{u}_{h}^{-},\bm{u}_{h}^{+} \right)$ is consistent with the continuous flux $\bm{\mathcal{F}} \left( \bm{u}_{h} \right)$ and can be taken as the Lax-Friedrichs  type flux in \eqref{LFflux} {by replacing $\bm{f}$ with $\bm{\mathcal{F}}$}.

Using the Galerkin approximation of $\bm{u}$ in \eqref{2dNumsol} gives
fully discrete LWDG scheme
\begin{equation*}\label{2dODEs_lw}
  \bm{u}_{\mathcal{K}}^{\left( l \right)} \left( t+\tau \right) = \bm{u}_{\mathcal{K}}^{\left( l \right)} \left( t \right) - \frac{\tau}{a _{\mathcal{K}}^{\left( l \right)}} \sum_{e \in \partial \mathcal{K}} \int_{e} \tilde{\bm{h}}_{e\mathcal{K}} v_{\mathcal{K}}^{\left( l \right)} \mathrm{d} S \notag + \frac{\tau}{a _{\mathcal{K}}^{\left( l \right)}} \int_{\mathcal{K}} \left[ \bm{\mathcal{F}} \left( \bm{u}_{h} \right) \left( \nabla v_{\mathcal{K}}^{\left( l \right)} \right)^{\top} + \bm{\mathcal{G}} \left( \bm{u}_{h} \right) v_{\mathcal{K}}^{\left( l \right)} \right] \mathrm{d} x \mathrm{d} y,
\end{equation*}
{where $l=0,1,\cdots,\frac{\left( q+1 \right) \left( q+2 \right)}{2} - 1$}.
%
It is worth mentioning that $\bm{\mathcal{F}} \left( \bm{u}_{h} \right)$ and $\bm{\mathcal{G}} \left( \bm{u}_{h} \right)$ are obtained by replacing $\bm{u}$ with $\bm{u}_{h}$ in  \eqref{F_u} and \eqref{G_u}, and  the spatial derivatives of $\bm{u}$ in  \ref{Appendix_details}  with those of $\bm{u}_h$ in \eqref{2dNumsol}. Moreover, just like the RKDG method in Section \ref{sect3.1:RKDG},
 the integrals in the above equation are calculated by using the $(q + 1)$-point Gauss-Legendre quadrature.

Similarly, the TSDG method is to seek the approximate solutions $\bm{u}_{h} \left( t,x,y \right)$ and $\bm{u}_{h}^{\ast} \left( t,x,y \right)$ with their components belonging to $\mathcal{V}_{h}$, such that they satisfy
\begin{equation*}\label{relation_TS_weak}
  \left\{
  \begin{array}{l}
    \int_{\mathcal{K}} \bm{u}_{h}^{\ast} \circ \bm{v}_{h} \mathrm{d} x \mathrm{d} y = \int_{\mathcal{K}} \bm{u}_{h} \circ \bm{v}_{h} \mathrm{d} x \mathrm{d} y - \frac{\tau}{3 \left( 1 - \vartheta \right)} \mathfrak{T}_{1} - \frac{\tau^{2}}{12 \left( 1 - \vartheta \right)} \mathfrak{T}_{2}, \\
    \int_{\mathcal{K}} \bm{u}_{h} \left( t + \tau,x,y \right) \circ \bm{v}_{h} \mathrm{d} x \mathrm{d} y = \int_{\mathcal{K}} \bm{u}_{h} \circ \bm{v}_{h} \mathrm{d} x \mathrm{d} y - \tau \mathfrak{T}_{1} - \vartheta \tau^{2} \mathfrak{T}_{2} - \frac{\left( 1 - \vartheta \right) \tau^{2}}{2} \mathfrak{T}_{3},
  \end{array}
  \right.
\end{equation*}
where
\begin{equation*}\label{Integralterm}
  \left\{
  \begin{array}{l}
    \mathfrak{T}_{1} = \sum_{e \in \partial \mathcal{K}} \int_{e} \widehat{\bm{h}}_{e\mathcal{K}} \circ \bm{v}_{h} \mathrm{d} S - \int_{\mathcal{K}} \left( \bm{f} \left( \bm{u}_{h} \right) \nabla \bm{v}_{h} +  \bm{\mathcal{M}} \left( \bm{u}_{h} \right) \circ \bm{v}_{h} \right) \mathrm{d} x \mathrm{d} y, \\
    \mathfrak{T}_{2} = \sum_{e \in \partial \mathcal{K}} \int_{e} \overline{\bm{h}}_{e\mathcal{K}} \circ \bm{v}_{h} \mathrm{d} S - \int_{\mathcal{K}} \left( \partial_{t} \bm{f} \left( \bm{u}_{h} \right) \nabla \bm{v}_{h} + \partial_{t}  \bm{\mathcal{M}} \left( \bm{u}_{h} \right) \circ \bm{v}_{h} \right) \mathrm{d} x \mathrm{d} y, \\
    \mathfrak{T}_{3} = \sum_{e \in \partial \mathcal{K}} \int_{e} \underline{\bm{h}}_{e\mathcal{K}} \circ \bm{v}_{h} \mathrm{d} S - \int_{\mathcal{K}} \left( \partial_{t} \bm{f} \left( \bm{u}_{h}^{\ast} \right) \nabla \bm{v}_{h} + \partial_{t}  \bm{\mathcal{M}} \left( \bm{u}_{h}^{\ast} \right) \circ \bm{v}_{h} \right) \mathrm{d} x \mathrm{d} y,
  \end{array}
  \right.
\end{equation*}
and the numerical {fluxes $\overline{\bm{h}}_{e\mathcal{K}} = \overline{\bm{h}}_{e\mathcal{K}} \left( \bm{u}_{h}^{-},\bm{u}_{h}^{+} \right)$ and $\underline{\bm{h}}_{e\mathcal{K}} = \underline{\bm{h}}_{e\mathcal{K}} \left( \bm{u}_{h}^{\ast,-}, \bm{u}_{h}^{\ast,+} \right)$ are taken as the Lax-Friedrichs type fluxes in \eqref{LFflux} by replacing $\bm{f}$ with $\partial_{t} \bm{f}$. The details of $\partial_{t} \bm{f}$ and $\partial_{t}  \bm{\mathcal{M}}$ can be found in \eqref{f_ut} and \eqref{M_ut} in \ref{Appendix_details}, respectively.

Using the Galerkin approximation  \eqref{2dNumsol} gives the
TSDG scheme
\begin{equation*}\label{relation_TS}
  \left\{
  \begin{array}{l}
    \bm{u}_{\mathcal{K}}^{\ast \left( l \right)} = \bm{u}_{\mathcal{K}}^{\left( l \right)} \left( t \right) - \frac{\tau}{3 \left( 1 - \vartheta \right) a _{\mathcal{K}}^{\left( l \right)}} \widetilde{\mathfrak{T}}_{1} - \frac{\tau^{2}}{12 \left( 1 - \vartheta \right) a _{\mathcal{K}}^{\left( l \right)}} \widetilde{\mathfrak{T}}_{2}, \\
  \bm{u}_{\mathcal{K}}^{\left( l \right)} \left( t+\tau \right) = \bm{u}_{\mathcal{K}}^{\left( l \right)} \left( t \right) - \frac{\tau}{a _{\mathcal{K}}^{\left( l \right)}} \widetilde{\mathfrak{T}}_{1} - \frac{\vartheta \tau^{2}}{a _{\mathcal{K}}^{\left( l \right)}} \widetilde{\mathfrak{T}}_{2} - \frac{\left( 1-\vartheta \right) \tau^{2}}{a _{\mathcal{K}}^{\left( l \right)}} \widetilde{\mathfrak{T}}_{3},
  \end{array}
  \right.
\end{equation*}
where $l=0,1,\cdots,\frac{\left( q+1 \right) \left( q+2 \right)}{2} - 1$,
\begin{equation*}\label{Integralterm2}
  \left\{
  \begin{array}{l}
    \widetilde{\mathfrak{T}}_{1} = \sum_{e \in \partial \mathcal{K}} \int_{e} \widehat{\bm{h}}_{e\mathcal{K}} v_{\mathcal{K}}^{\left( l \right)} \mathrm{d} S - \int_{\mathcal{K}} \left( \bm{f} \left( \bm{u}_{h} \right) \left( \nabla v_{\mathcal{K}}^{\left( l \right)} \right)^{\top} +  \bm{\mathcal{M}} \left( \bm{u}_{h} \right) v_{\mathcal{K}}^{\left( l \right)} \right) \mathrm{d} x \mathrm{d} y, \\
    \widetilde{\mathfrak{T}}_{2} = \sum_{e \in \partial \mathcal{K}} \int_{e} \overline{\bm{h}}_{e\mathcal{K}} v_{\mathcal{K}}^{\left( l \right)} \mathrm{d} S - \int_{\mathcal{K}} \left( \partial_{t} \bm{f} \left( \bm{u}_{h} \right) \left( \nabla v_{\mathcal{K}}^{\left( l \right)} \right)^{\top} + \partial_{t}  \bm{\mathcal{M}}\left( \bm{u}_{h} \right) v_{\mathcal{K}}^{\left( l \right)} \right) \mathrm{d} x \mathrm{d} y, \\
    \widetilde{\mathfrak{T}}_{3} = \sum_{e \in \partial \mathcal{K}} \int_{e} \underline{\bm{h}}_{e\mathcal{K}} v_{\mathcal{K}}^{\left( l \right)} \mathrm{d} S - \int_{\mathcal{K}} \left( \partial_{t} \bm{f} \left( \bm{u}_{h}^{\ast} \right) \left( \nabla v_{\mathcal{K}}^{\left( l \right)} \right)^{\top} + \partial_{t}  \bm{\mathcal{M}} \left( \bm{u}_{h}^{\ast} \right) v_{\mathcal{K}}^{\left( l \right)} \right) \mathrm{d} x \mathrm{d} y.
  \end{array}
  \right.
\end{equation*}
Similarly, the integrals in the above equations are calculated by the $(q + 1)$-point Gauss-Legendre quadrature.

\subsection{Computational complexity}
\label{subsction_complexity}

This subsection estimates the computational complexity  of the above three methods in one dimension,  which are denoted by $P^{q}$-LWDG method, $P^{q}$-TSDG method and $P^{q}$-RKDG method for a fix degree $q$, respectively.
A relative discussion was  given in \cite{Qiu2005} for the LWDG and the RKDG methods of the nonlinear hyperbolic conservation laws.

At each one time step, the LWDG method needs only one stage, correspondingly, the TSDG method needs two stages and the RKDG method needs four stages. Does the LWDG method   need the least CPU time, followed by the TSDG method, and is the RKDG method  the most one?
 Table \ref{Table_DG_complexity} lists  the numbers of the operations `$+/-$', `$\times / \div$' and `$=$'
 for the $P^{q}$-DG methods, $q=2,3$, where
   $G_{p} = q+1$ denotes the number of Gauss-points in Gaussian quadrature, $J$ is the number of cells in space and $N_{\tau}$
   is the number of cells in time.
   \ref{Appendix_codes} presents   pseudo codes  of  three 1D $P^{2}$-DG methods   with $\kappa = 1$.
 The results clearly shows  the LWDG method needs the most CPU time. The reason is that  it requests more computational effort in calculating high-order spatial derivatives of $\bm{u}$. 

\begin{table}[htbp]
\centering
\caption{Computational complexities of three DG schemes.}
\begin{tabular}{cccccc}
  \hline
  Schemes       & $P^{2}$-LWDG                                          & $P^{3}$-LWDG                                          \\
  \hline
  $+/-$         & $\left(\left(166G_{p}+270\right)J+220\right) N_{\tau}$ & $\left(\left(186G_{p}+294\right)J+220\right) N_{\tau}$ \\
  $\times/\div$ & $\left(\left(207G_{p}+284\right)J+256\right) N_{\tau}$ & $\left(\left(231G_{p}+308\right)J+257\right) N_{\tau}$ \\
  =             & $\left(\left(95G_{p}+192\right)J+152\right) N_{\tau}$  & $\left(\left(99G_{p}+212\right)J+152\right) N_{\tau}$  \\
  \hline
  Schemes       & $P^{2}$-TSDG                                          & $P^{3}$-TSDG                                          \\
  \hline
  $+/-$         & $\left(\left(122G_{p}+228\right)J+72\right) N_{\tau}$  & $\left(\left(154G_{p}+276\right)J+72\right) N_{\tau}$  \\
  $\times/\div$ & $\left(\left(136G_{p}+156\right)J+60\right) N_{\tau}$  & $\left(\left(168G_{p}+204\right)J+62\right) N_{\tau}$  \\
  =             & $\left(\left(104G_{p}+208\right)J+69\right) N_{\tau}$  & $\left(\left(116G_{p}+240\right)J+79\right) N_{\tau}$  \\
  \hline
  Schemes       & $P^{2}$-RKDG                                          & $P^{3}$-RKDG                                          \\
  \hline
  $+/-$         & $\left(\left(128G_{p}+260\right)J+50\right) N_{\tau}$  & $\left(\left(176G_{p}+320\right)J+50\right) N_{\tau}$  \\
  $\times/\div$ & $\left(\left(148G_{p}+152\right)J+23\right) N_{\tau}$  & $\left(\left(196G_{p}+208\right)J+25\right) N_{\tau}$  \\
  =             & $\left(\left(116G_{p}+208\right)J+57\right) N_{\tau}$  & $\left(\left(132G_{p}+240\right)J+59\right) N_{\tau}$  \\
  \hline
\end{tabular}\label{Table_DG_complexity}
\end{table}

\section{Numerical results}\label{Sec_results}
This section   conducts some 1D and 2D numerical experiments to validate the accuracy and the conservative properties
of the proposed DG methods
and to investigate some new phenomena.
%
%
Unless stated otherwise, the parameters $m$, $\lambda$ and $\kappa$ are taken as 1, $\frac{1}{2}$ and $1$, respectively, and
the 1D and 2D computational domains are  taken as $\left[ -60, 60 \right]$   and $\left[ -15, 15 \right]\times \left[ -15, 15 \right]$, respectively.

\subsection{1D case}
Some 1D examples are first considered and the time step size is given by the  CFL condition \cite{Shao2006}
\begin{equation}\label{CFL}
  \tau = \frac{\mu \Delta x}{2 q + 1},\ q = 1,2,3.
\end{equation}
In practical computations, $\mu$ is taken as 0.25.

\begin{example}[Accuracy test in 1D]\label{Example_DG_1D_accuracy}\rm
This example tests the numerical accuracy, the charge and energy  {conservations}, and the CPU time of the proposed DG methods  for the 1D  {NLD equation}
\begin{equation*}
  \partial_{t}{\Psi} + \sigma_{1}\partial_{x}{\Psi} + \mathrm{i} g \left( {\Psi}^{\ast} \sigma_{3} \Psi \right) \sigma_{3} \Psi = 0,
\end{equation*}
whose exact solutions can be found in \cite{Shao2006,Xu2015,Li2017}. The initial condition is $\Psi \left( 0,x \right) = \Psi^{tw} \left( 0,x - 5 \right)$ with $\omega = \frac{4}{5}$ and $v = -\frac{1}{5}$.
\end{example}

Tables \ref{Tab_order_LW_1d}-\ref{Tab_order_RK_1d} list the errors at $t=50$ and  corresponding convergence rates of
several DG methods. It is seen that our schemes get the theoretical order
accuracy as expected. 

\begin{table}[htbp]
  \centering
  \caption{Example \ref{Example_DG_1D_accuracy}: Accuracy test of the 1D LWDG methods.}
  \begin{tabular}{cccccc}
    \hline
    Schemes & $J$ & $L^{2}$ error & order & $L^{\infty }$ error & order \\
    \hline
    $P^{1}$-LWDG &  200 & 1.3599e-01 & -    & 6.7381e-02 & -    \\
                 &  400 & 2.0073e-02 & 2.76 & 1.0415e-02 & 2.69 \\
                 &  800 & 3.2345e-03 & 2.63 & 1.7312e-03 & 2.59 \\
                 & 1600 & 6.3411e-04 & 2.35 & 3.3198e-04 & 2.38 \\
    \hline
    $P^{2}$-LWDG &  100 & 4.6881e-02 & -    & 2.1944e-02 & -    \\
                 &  200 & 4.5786e-03 & 3.36 & 2.1371e-03 & 3.36 \\
                 &  400 & 5.2971e-04 & 3.11 & 2.4760e-04 & 3.11 \\
                 &  800 & 6.4815e-05 & 3.03 & 3.0277e-05 & 3.03 \\
    \hline
    $P^{3}$-LWDG &  100 & 1.6810e-03 & -    & 7.7977e-04 & -    \\
                 &  200 & 6.5653e-05 & 4.68 & 3.1363e-05 & 4.64 \\
                 &  400 & 2.7559e-06 & 4.57 & 1.3168e-06 & 4.57 \\
                 &  800 & 1.3926e-07 & 4.31 & 6.6914e-08 & 4.30 \\
    \hline
  \end{tabular}
  \label{Tab_order_LW_1d}
\end{table}
\begin{table}[htbp]
  \centering
  \caption{Example \ref{Example_DG_1D_accuracy}: Accuracy test of the 1D TSDG methods.}
  \begin{tabular}{cccccc}
    \hline
    Schemes & $J$ & $L^{2}$ error & order & $L^{\infty }$ error & order \\
    \hline
    $P^{1}$-TSDG &  200 & 1.4119e-01 & -    & 6.9361e-02 & -    \\
                 &  400 & 2.0359e-02 & 2.79 & 1.0457e-02 & 2.73 \\
                 &  800 & 3.1239e-03 & 2.70 & 1.6739e-03 & 2.64 \\
                 & 1600 & 5.7759e-04 & 2.44 & 3.0768e-04 & 2.44 \\
    \hline
    $P^{2}$-TSDG &  100 & 4.4493e-02 & -    & 2.0827e-02 & -    \\
                 &  200 & 4.2889e-03 & 3.37 & 2.0050e-03 & 3.38 \\
                 &  400 & 4.9332e-04 & 3.12 & 2.3088e-04 & 3.12 \\
                 &  800 & 6.0271e-05 & 3.03 & 2.8196e-05 & 3.03 \\
    \hline
    $P^{3}$-TSDG &  100 & 1.4870e-03 & -    & 6.9071e-04 & -    \\
                 &  200 & 5.6135e-05 & 4.73 & 2.6922e-05 & 4.68 \\
                 &  400 & 2.3230e-06 & 4.59 & 1.1161e-06 & 4.59 \\
                 &  800 & 1.1616e-07 & 4.32 & 5.7802e-08 & 4.27 \\
    \hline
  \end{tabular}
  \label{Tab_order_TS_1d}
\end{table}
\begin{table}[htbp]
  \centering
  \caption{Example \ref{Example_DG_1D_accuracy}: Accuracy test of the 1D RKDG methods.}
  \begin{tabular}{cccccc}
    \hline
    Schemes & $J$ & $L^{2}$ error & order & $L^{\infty }$ error & order \\
    \hline
    $P^{1}$-RKDG &  200 & 1.6410e-01 & -    & 7.7104e-02 & -    \\
                 &  400 & 2.1730e-02 & 2.92 & 1.0370e-02 & 2.89 \\
                 &  800 & 2.7526e-03 & 2.98 & 1.3517e-03 & 2.94 \\
                 & 1600 & 3.4993e-04 & 2.98 & 1.8119e-04 & 2.90 \\
    \hline
    $P^{2}$-RKDG &  100 & 1.8467e-02 & -    & 8.7084e-03 & -    \\
                 &  200 & 6.6117e-04 & 4.80 & 3.4908e-04 & 4.64 \\
                 &  400 & 3.1685e-05 & 4.38 & 1.9083e-05 & 4.19 \\
                 &  800 & 3.1187e-06 & 3.34 & 1.9758e-06 & 3.27 \\
    \hline
    $P^{3}$-RKDG &  100 & 2.7938e-04 & -    & 1.7751e-04 & -    \\
                 &  200 & 8.9008e-06 & 4.97 & 6.4287e-06 & 4.79 \\
                 &  400 & 5.4281e-07 & 4.04 & 3.9350e-07 & 4.03 \\
                 &  800 & 3.3929e-08 & 4.00 & 2.4663e-08 & 4.00 \\
    \hline
  \end{tabular}
  \label{Tab_order_RK_1d}
\end{table}

 Let us further investigate the performance of the numerical schemes in the charge and energy conservations.
%
Figure \ref{Figure_DG_1D_conservation} shows the time evolution of the relative charge  and  energy differences  defined by
\begin{equation*}
  Q_{\mathrm{rela}} \left( t \right) = \left| \frac{Q \left( t \right) - Q \left( 0 \right)}{Q \left( 0 \right)} \right|,\ E_{\mathrm{rela}} \left( t \right) = \left| \frac{E \left( t \right) - E \left( 0 \right)}{E \left( 0 \right)} \right|,
\end{equation*}
 {with} $J = 1000$. The results show that the present three DG methods can conserve the discrete charge and the discrete energy approximately.

\begin{figure}[htbp]
  \centering
  \includegraphics[width=0.4\textwidth]{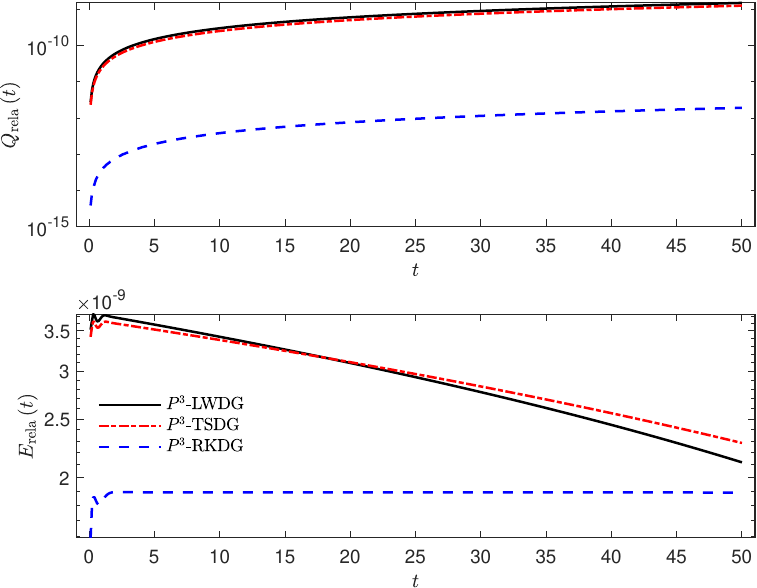} \\
  \caption{Example \ref{Example_DG_1D_accuracy}: Time evolution of the relative charge  and energy differences.}
  \label{Figure_DG_1D_conservation}
\end{figure}


 Finally, we test the CPU time executed by MATLAB and C++ according to the pseudo codes.
 The parameters are the same as those in the accuracy test except for $\mu = 0.5$.
%
Tables \ref{Table_DG_1D_CPU_MATLAB} and   \ref{Table_DG_1D_CPU_cpp} record the CPU times for different schemes by MATLAB and C++ respectively, those data are the average values of five calculations in order to reduce the error, and the output time is taken as $t = 0.005$. It is seen  that when $J$ is large enough, those numerical results are consistent with the analysis in Section \ref{subsction_complexity}.

\begin{table}[htbp]
  \centering
  \caption{Example \ref{Example_DG_1D_accuracy}: CPU times (second) executed by MATLAB for different schemes.}
  \begin{tabular}{ccccc}
    \hline
    Schemes$\setminus J$ & 100000  & 200000 & 400000  & 800000 \\
    \hline
    $P^{2}$-LWDG & 4.77 & 23.07 & 92.11 & 369.73 \\
    $P^{2}$-TSDG & 4.06 & 20.41 & 82.37 & 330.87 \\
    $P^{2}$-RKDG & 3.99 & 21.21 & 85.95 & 343.80 \\
    \hline
    $P^{3}$-LWDG & 8.76 & 39.68 & 158.02 & 633.73 \\
    $P^{3}$-TSDG & 7.60 & 37.57 & 150.40 & 600.27 \\
    $P^{3}$-RKDG & 7.59 & 38.62 & 155.49 & 622.40 \\
    \hline
  \end{tabular}\label{Table_DG_1D_CPU_MATLAB}
\end{table}

\begin{table}[htbp]
  \centering
  \caption{Example \ref{Example_DG_1D_accuracy}: CPU times (second) executed by C++ for different schemes.}
  \begin{tabular}{ccccc}
    \hline
    Schemes$\setminus J$ & 100000  & 200000 & 400000  & 800000 \\
    \hline
    $P^{2}$-LWDG & 11.46 & 46.11 & 183.62 & 727.79 \\
    $P^{2}$-TSDG & 9.21  & 37.51 & 150.48 & 593.24 \\
    $P^{2}$-RKDG & 9.67  & 39.29 & 156.26 & 619.45 \\
    \hline
    $P^{3}$-LWDG & 21.24 & 84.61 & 335.62 & 1344.08 \\
    $P^{3}$-TSDG & 19.75 & 79.00 & 314.68 & 1244.57 \\
    $P^{3}$-RKDG & 20.85 & 83.40 & 332.78 & 1324.90 \\
    \hline
  \end{tabular}\label{Table_DG_1D_CPU_cpp}
\end{table}

\begin{example}[Error history]\label{Example_DG_errorhistory}\rm
  The $L^{\infty}$-error history is investigated in this example.
  The   {NLD equation} and parameters are the same as those in Example \ref{Example_DG_1D_accuracy} except for $J = 500$
  and
  the initial condition  $\Psi \left( 0,x \right) = \Psi^{sw} \left( 0,x \right)$.
\end{example}

Figure \ref{Fig_history} shows the time evolution of the $L^{\infty}$-errors for different methods from $t = 0$ to $ 3000$,
where  we fit the curves linearly and list the resulting slopes.
 Relatively speaking, the RKDG methods perform better than the other two in a long time simulation.
 A similar result between the LWDG  and   RKDG methods for the linear advection problem is observed in \cite{Guo2015}.
\begin{figure}[htbp]
  \centering
  \includegraphics[width=0.4\textwidth]{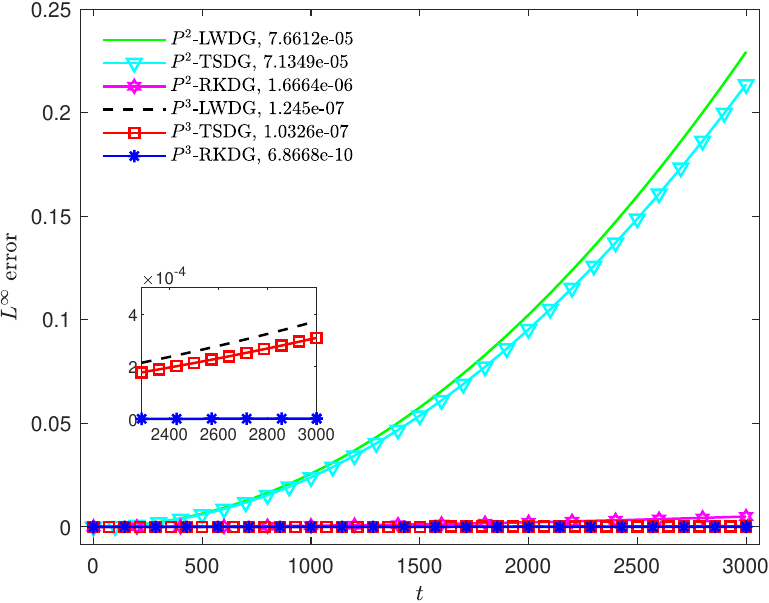}\\
  \caption{Example \ref{Example_DG_errorhistory}: $L^{\infty}$-errors from $t=0$ to $3000$ for different numerical methods.}
  \label{Fig_history}
\end{figure}

As shown in  Figure \ref{Fig_history}, the RKDG method performs better relatively than the other two methods in a long time simulation, so we use the $P^{3}$-RKDG method to simulate the following examples except for Example \ref{Example_DG_2D_accuracy}.

\begin{example}[Collision]\label{Example_DG_1D_collision}
  The inelastic interaction in the binary collision and the ternary collision has been observed in \cite{Alvarez1981,Shao2006}. This example  tries to observe the inelastic interaction in the quaternary collision. The initial data are taken as the linear superposition of four waves, that is, $\Psi \left( 0,x\right) = \Psi^{tw} \left( 0,x + 15 \right) + \Psi^{tw} \left( 0,x + 5 \right) + \Psi^{tw} \left( 0,x - 5 \right) + \Psi^{tw} \left( 0,x - 15 \right)$. Table \ref{Table_DG_1D_quaternary} lists the parameters. The spatial domain is taken as $\left[ -70,70 \right]$, divided into $J = 1400$ cells.
\end{example}
%
\begin{table}[htbp]
  \centering
  \caption{Example \ref{Example_DG_1D_collision}: Parameters in the quaternary collision.}
  \begin{tabular}{cccccc} \hline
  & $\Psi^{tw} \left( 0,x + 15 \right)$ & $\Psi^{tw} \left( 0,x + 5 \right)$ & $\Psi^{tw} \left( 0,x - 5 \right)$ & $\Psi^{tw} \left(0,x - 15 \right)$ \\
  \hline
  $v$      & $1/5$ & $1/10$ & $-1/10$ & $-1/5$ \\
  $\omega$ & $3/5$ & $4/5$  & $4/5$   & $3/5$  \\
  \hline
  \end{tabular}
  \label{Table_DG_1D_quaternary}
\end{table}

Numerical results in Figure \ref{Figure_DG_1D_collision} shows   the inelastic interaction   in the quaternary collision
and the charge decreasing property   with time. The latter  is consistent with Proposition \ref{Pro_entropy}.

\begin{figure}[htbp]
  \centering
  \includegraphics[width=0.4\textwidth]{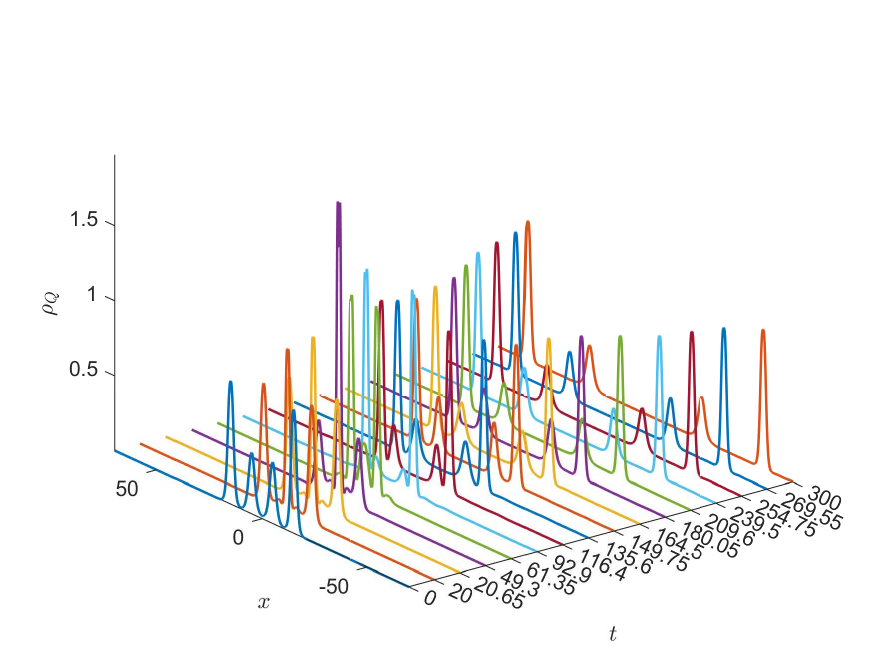}
  \includegraphics[width=0.4\textwidth]{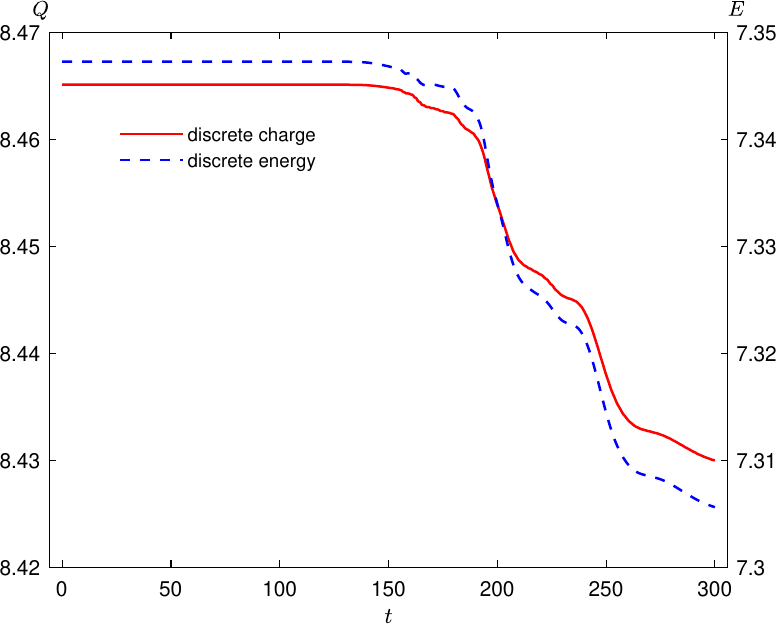} \\
  \caption{Example \ref{Example_DG_1D_collision}: Inelastic interaction in the quaternary collision. Time evolution of the charge density (left) and the discrete charge and the discrete energy (right).}
  \label{Figure_DG_1D_collision}
\end{figure}

\subsection{2D case}
This section investigates some 2D examples for the case of $S = 0$.
%
%
Unless stated otherwise,   we set $\Delta x = \Delta y = 0.2$, and the time step size is given by the following condition
\begin{equation*}
  \tau = \frac{\mu \min\{ h_{x},h_{y} \}}{2 \left( 2 q + 1 \right)},\ q = 1,2,3,
\end{equation*}
with $\mu=  0.25$ for the $P^{3}$-LWDG method and 0.5 for the other methods.

\begin{example}[Accuracy test]\label{Example_DG_2D_accuracy}\rm
This example tests the accuracy   of our DG methods. 
 In order to do that, we add a source term $R = R \left( t,x,y \right) = \left( r_{1} \left( t,x,y \right),r_{2} \left( t,x,y \right) \right)^{\top} \in \mathbb{C}^{2}$ into the 2D  {NLD equation} \eqref{2dNLDE} as
\begin{equation}\label{2dNLDE_accoracy}
  \partial_{t} \Psi + \sigma _{1} \partial_{x} \Psi + \sigma _{2} \partial_{y} \Psi + \mathrm{i} g \left( \Psi^{\ast} \sigma_{3} \Psi \right) \sigma_{3} \Psi = R,
\end{equation}
 so that the exact solutions of \eqref{2dNLDE_accoracy} can be taken as $\psi_{p} = c_{p} \varphi \left( t,x,y \right)$ with the constant complex number  $c_{p}$, $p=1,2$. It is the so-called   method of manufactured solutions.
 The initial condition is taken as $\left[ c_{1} \varphi \left( 0,x,y \right), c_{2} \varphi \left( 0,x,y \right) \right]^{\top}$ with $c_{1} = 1$, $c_{2} = 2$, and the exact solution being $\varphi \left( t,x,y \right) = t^{4} \mathrm{e}^{-5 \left( x^{2} + y^{2} \right)}$,
 and the spatial domain is taken as $\left[ -2,2 \right]^{2}$.
\end{example}

Tables \ref{Table_LWDG_2D_accuracy}-\ref{Table_RKDG_2D_accuracy} list the errors at $t = 0.2$ and   corresponding convergence rates, which  are consistent with the expected.

\begin{table}[htbp]
  \centering
  \caption{Example \ref{Example_DG_2D_accuracy}: Accuracy test of the 2D LWDG methods.}
  \begin{tabular}{cccccc}
    \hline
    Schemes & $J \times K$ & $L^{2}$ error & order & $L^{\infty }$ error & order \\
    \hline
    $P^{1}$-LWDG &  40$\times$ 40 & 9.1886e-03 & -    & 1.9059e-02 & -    \\
                 &  80$\times$ 80 & 2.3082e-03 & 1.99 & 4.7100e-03 & 2.02 \\
                 & 160$\times$160 & 5.7818e-04 & 2.00 & 1.1876e-03 & 1.99 \\
                 & 320$\times$320 & 1.4469e-04 & 2.00 & 2.9750e-04 & 2.00 \\
    \hline
    $P^{2}$-LWDG &  20$\times$ 20 & 3.9147e-03 & -    & 7.5917e-03 & -    \\
                 &  40$\times$ 40 & 4.8257e-04 & 3.02 & 9.4199e-04 & 3.01 \\
                 &  80$\times$ 80 & 5.9930e-05 & 3.01 & 1.1967e-04 & 2.98 \\
                 & 160$\times$160 & 7.4856e-06 & 3.00 & 1.5132e-05 & 2.98 \\
    \hline
    $P^{3}$-LWDG &  20$\times$ 20 & 4.6905e-04 & -    & 1.7709e-03 & -    \\
                 &  40$\times$ 40 & 3.2684e-05 & 3.84 & 1.4145e-04 & 3.65 \\
                 &  80$\times$ 80 & 2.1343e-06 & 3.94 & 9.8153e-06 & 3.85 \\
                 & 160$\times$160 & 1.4042e-07 & 3.93 & 6.2605e-07 & 3.97 \\
    \hline
  \end{tabular}
  \label{Table_LWDG_2D_accuracy}
\end{table}
\begin{table}[htbp]
  \centering
  \caption{Example \ref{Example_DG_2D_accuracy}: Accuracy test of the 2D TSDG methods.}
  \begin{tabular}{cccccc}
    \hline
    Schemes & $J \times K$ & $L^{2}$ error & order & $L^{\infty }$ error & order \\
    \hline
    $P^{1}$-TSDG &  40$\times$ 40 & 9.1237e-03 & -    & 1.8786e-02 & -    \\
                 &  80$\times$ 80 & 2.2918e-03 & 1.99 & 4.5573e-03 & 2.04 \\
                 & 160$\times$160 & 5.7418e-04 & 2.00 & 1.1485e-03 & 1.99 \\
                 & 320$\times$320 & 1.4371e-04 & 2.00 & 2.8793e-04 & 2.00 \\
    \hline
    $P^{2}$-TSDG &  20$\times$ 20 & 3.9203e-03 & -    & 7.7629e-03 & -    \\
                 &  40$\times$ 40 & 4.8322e-04 & 3.02 & 9.4513e-04 & 3.04 \\
                 &  80$\times$ 80 & 6.0058e-05 & 3.01 & 1.2176e-04 & 2.96 \\
                 & 160$\times$160 & 7.4967e-06 & 3.00 & 1.5362e-05 & 2.99 \\
    \hline
    $P^{3}$-TSDG &  20$\times$ 20 & 4.6357e-04 & -    & 1.6581e-03 & -    \\
                 &  40$\times$ 40 & 3.2458e-05 & 3.84 & 1.3339e-04 & 3.64 \\
                 &  80$\times$ 80 & 2.1220e-06 & 3.94 & 9.3929e-06 & 3.83 \\
                 & 160$\times$160 & 1.3493e-07 & 3.98 & 6.0154e-07 & 3.96 \\
    \hline
  \end{tabular}
  \label{Table_TSDG_2D_accuracy}
\end{table}
\begin{table}[htbp]
  \centering
  \caption{Example \ref{Example_DG_2D_accuracy}: Accuracy test of the 2D RKDG methods.}
  \begin{tabular}{cccccccc}
    \hline
    Schemes & $J \times K$ & $L^{2}$ error & order & $L^{\infty }$ error & order \\
    \hline
    $P^{1}$-RKDG &  40$\times$ 40 & 9.1862e-03 & -    & 1.8488e-02 & -    \\
                 &  80$\times$ 80 & 2.2727e-03 & 2.02 & 4.2929e-03 & 2.11 \\
                 & 160$\times$160 & 5.6610e-04 & 2.01 & 1.0160e-03 & 2.08 \\
                 & 320$\times$320 & 1.4138e-04 & 2.00 & 2.4606e-04 & 2.05 \\
    \hline
    $P^{2}$-RKDG &  20$\times$ 20 & 4.2264e-03 & -    & 9.6981e-03 & -    \\
                 &  40$\times$ 40 & 5.2404e-04 & 3.01 & 1.2033e-03 & 3.01 \\
                 &  80$\times$ 80 & 6.5573e-05 & 3.00 & 1.4619e-04 & 3.04 \\
                 & 160$\times$160 & 8.1943e-06 & 3.00 & 1.8016e-05 & 3.02 \\
    \hline
    $P^{3}$-RKDG &  20$\times$ 20 & 4.7686e-04 & -    & 1.9168e-03 & -    \\
                 &  40$\times$ 40 & 3.3019e-05 & 3.85 & 1.5175e-04 & 3.66 \\
                 &  80$\times$ 80 & 2.1497e-06 & 3.94 & 1.0473e-05 & 3.86 \\
                 & 160$\times$160 & 1.3604e-07 & 3.98 & 6.6968e-07 & 3.97 \\
    \hline
  \end{tabular}
  \label{Table_RKDG_2D_accuracy}
\end{table}

\begin{example}[Standing wave solutions]\label{Example_DG_2D_SW}\rm
  This example considers two standing wave solutions \cite{Cuevas-Maraver2018} of the 2D  {NLD equation}: (i) $\omega = 0.8$, (ii) $\omega = 0.12$.
\end{example}

Figures \ref{Figure_DG_2D_sw_1} and \ref{Figure_DG_2D_sw_2} show 
the charge densities of those standing wave solutions  at several different times obtained with $\mu = 0.7$.
 Figure \ref{Figure_DG_2D_sw_error} records the maximum deviations of the charge density  defined by
\begin{equation*}
  \mathrm{dev}_{\rho_{Q}} := \max_{x,y} \left\{ \left| \left| \Psi \left( t,x,y \right) \right|^{2} - \left| \Psi \left( 0,x,y \right) \right|^{2} \right| \right\}.
\end{equation*}
We see that the charge density in the first case is almost unchanged for a very long time, while for the second one, it rotates around the center from about $t = 1800$, and reaches the maximum amplitude at about $t = 2299$.
It is an interesting phenomenon  that after
a certain time, the charge density changes periodically with ``circular ring-elliptical ring-circular ring'', but  has not been observed in the literature. 

\begin{figure}[htbp]
  \centering
  \includegraphics[width=0.2\textwidth]{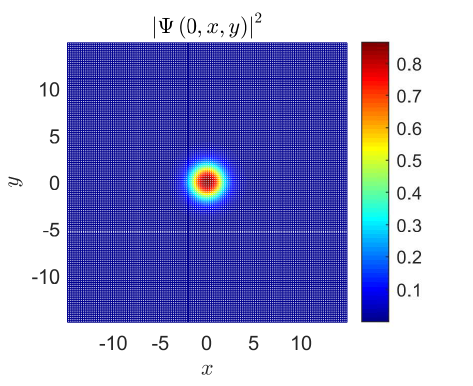}
  \includegraphics[width=0.2\textwidth]{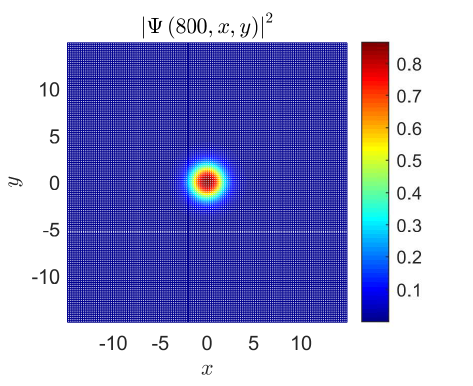}
  \includegraphics[width=0.2\textwidth]{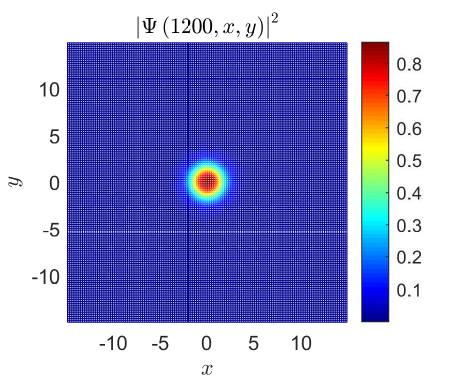}
  \includegraphics[width=0.2\textwidth]{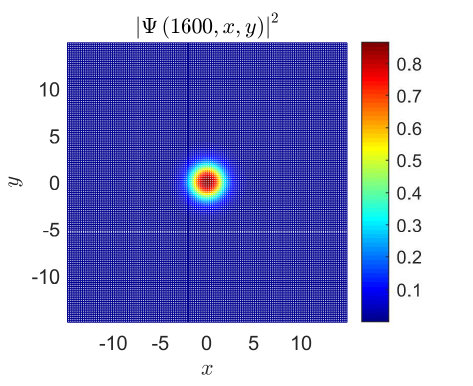} \\
  \includegraphics[width=0.2\textwidth]{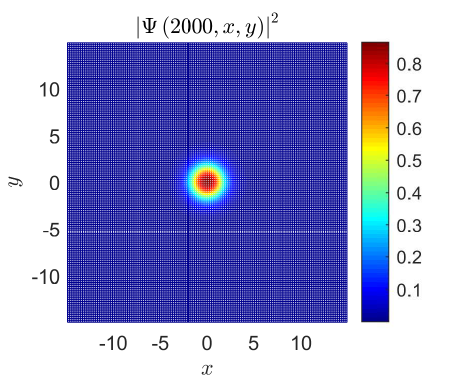}
  \includegraphics[width=0.2\textwidth]{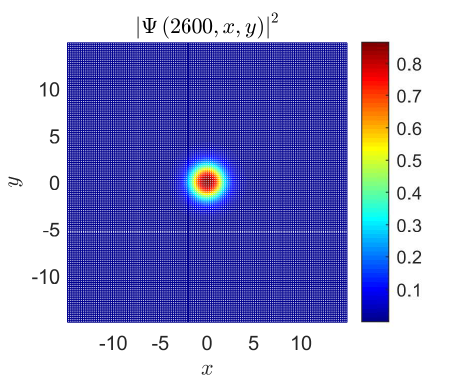}
  \includegraphics[width=0.2\textwidth]{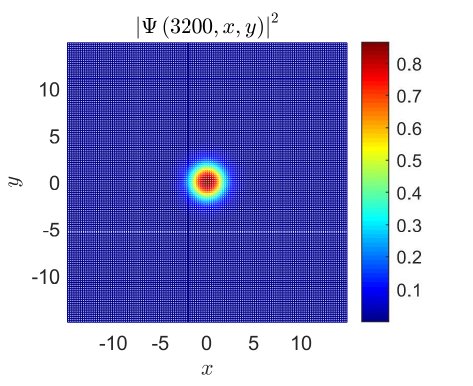}
  \includegraphics[width=0.2\textwidth]{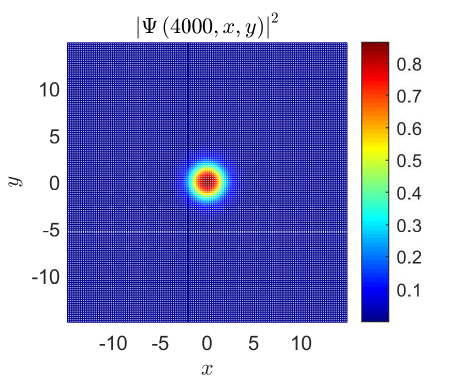} \\
  \caption{Example \ref{Example_DG_2D_SW}: Charge densities at $t = 0,800,1200,1600,2000,2600,3200,4000$, with $\omega = 0.8$.}
  \label{Figure_DG_2D_sw_1}
\end{figure}
\begin{figure}[htbp]
  \centering
  \includegraphics[width=0.2\textwidth]{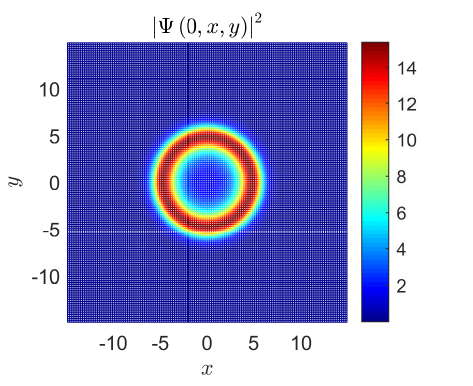}
  \includegraphics[width=0.2\textwidth]{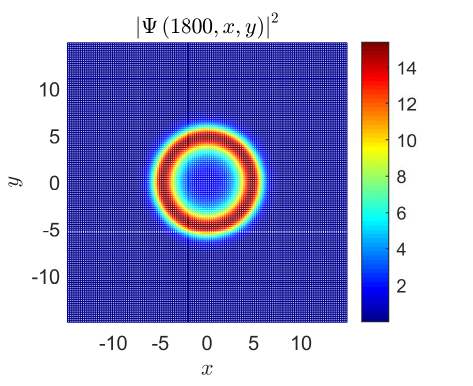}
  \includegraphics[width=0.2\textwidth]{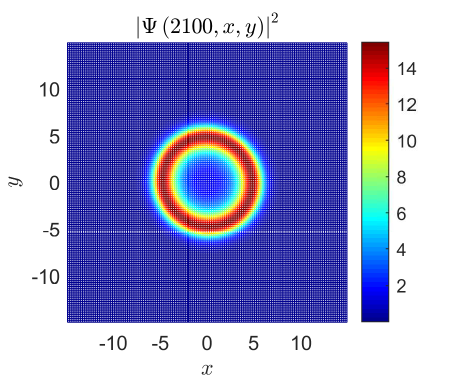}
  \includegraphics[width=0.2\textwidth]{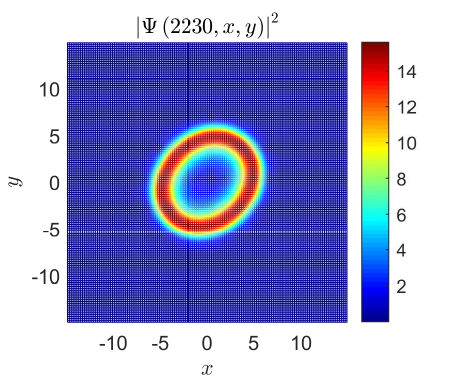} \\
  \includegraphics[width=0.2\textwidth]{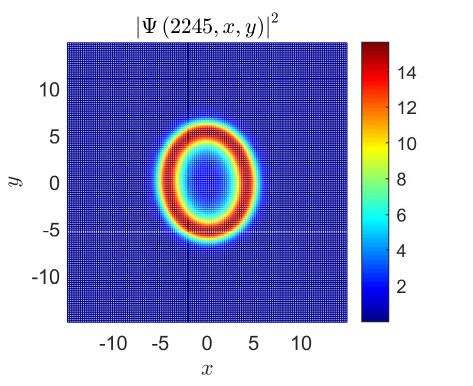}
  \includegraphics[width=0.2\textwidth]{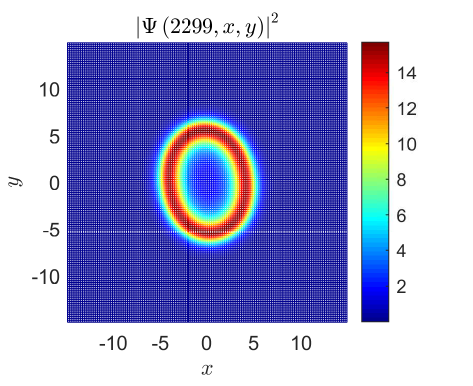}
  \includegraphics[width=0.2\textwidth]{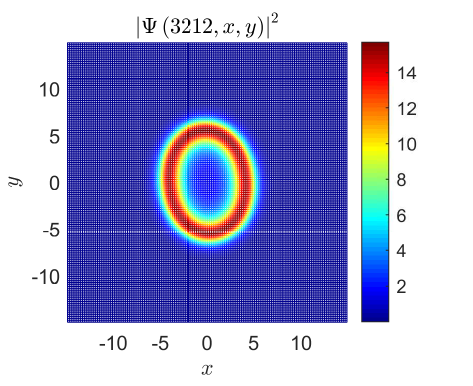}
  \includegraphics[width=0.2\textwidth]{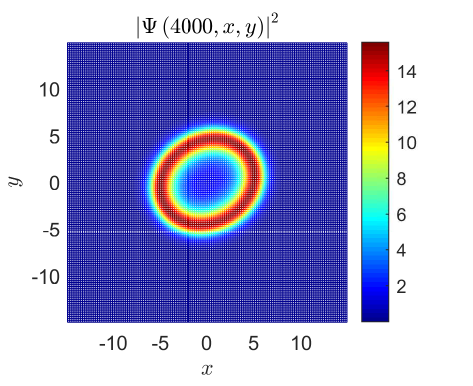} \\
  \caption{Example \ref{Example_DG_2D_SW}: Charge densities at $t = 0,1800,2100,2230,2245,2299,3212,4000$, with $\omega = 0.12$.}
  \label{Figure_DG_2D_sw_2}
\end{figure}

\begin{figure}[htbp]
  \centering
  \includegraphics[width=0.4\textwidth]{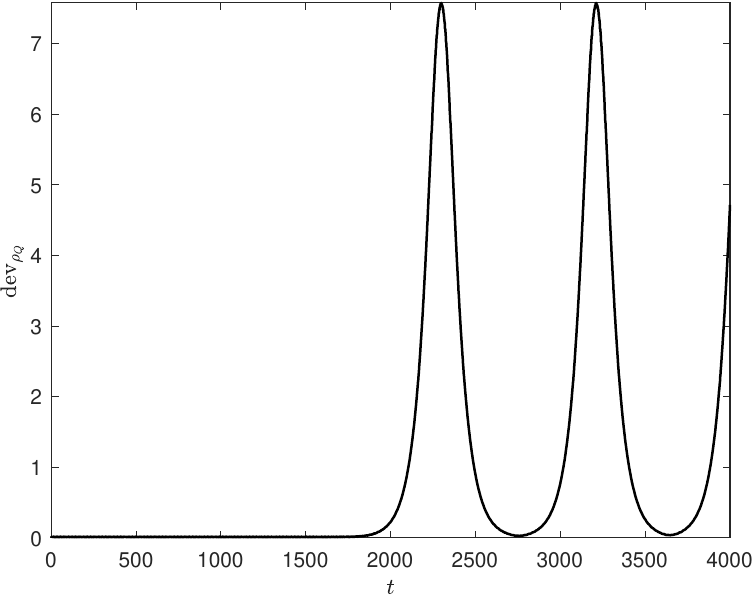}
  \includegraphics[width=0.4\textwidth]{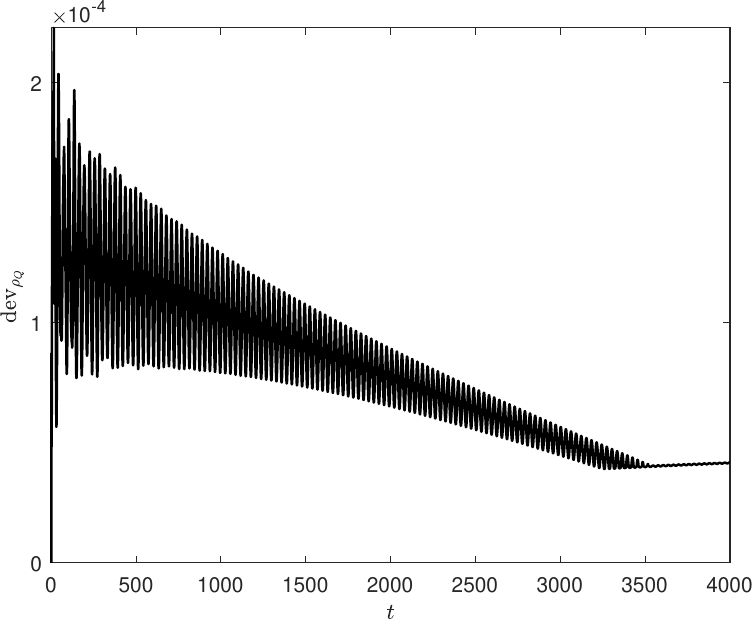} \\
  \caption{Example \ref{Example_DG_2D_SW}: Time evolutions of the maximum deviation of charge density for the case of $\omega = 0.12$ (left) and $\omega = 0.8$ (right).}
  \label{Figure_DG_2D_sw_error}
\end{figure}

\begin{example}[Oscillation state]\label{Example_DG_2D_oscillation}\rm
  This example investigates the interaction of two standing waves of the 2D  {NLD equation}. The initial condition is taken as the linear superposition of two standing waves, i.e., $\Psi \left( 0,x,y \right) = \Psi^{sw} \left( 0,x-2,y \right) + \Psi^{sw} \left( 0,x+2,y \right)$. The computational domain is taken as $\left[ -25.5,25.5 \right]^{2}$.
\end{example}

 Figure \ref{Figure_DG_2D_sw_3} gives the charge densities at $t = 0,7,124,234,426,433,578,600$, with $\omega = 0.8$.
  Figure \ref{Figure_masscenter_2d_3} plots the charge density at $\left( x,y \right) = \left( 0,0 \right)$ with respect to $t$.
 One can see that a long-lived oscillation state is observed.

\begin{figure}[htbp]
  \centering
  \includegraphics[width=0.2\textwidth]{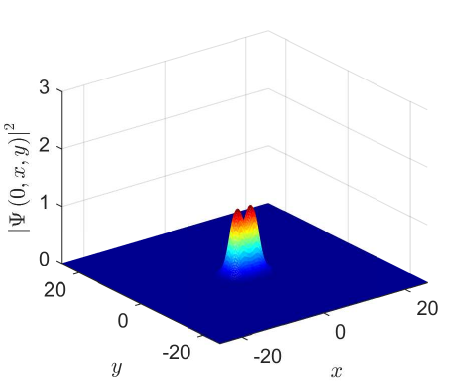}
  \includegraphics[width=0.2\textwidth]{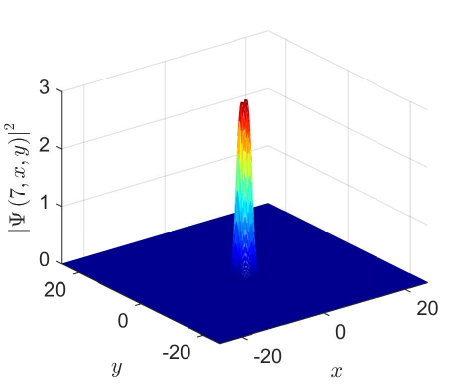}
  \includegraphics[width=0.2\textwidth]{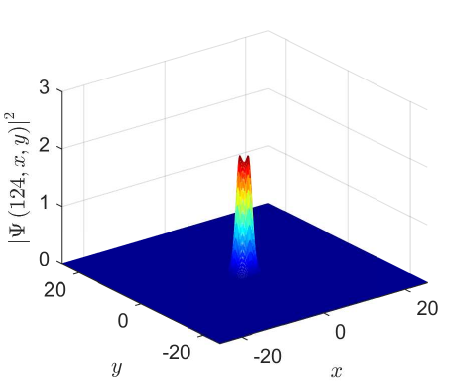}
  \includegraphics[width=0.2\textwidth]{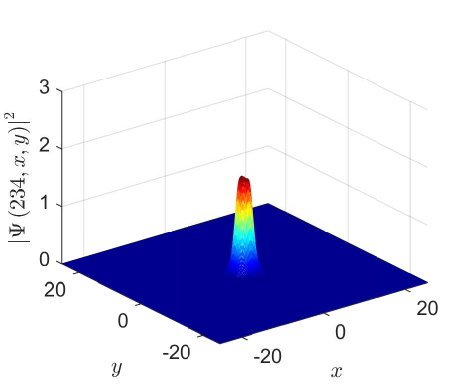} \\
  \includegraphics[width=0.2\textwidth]{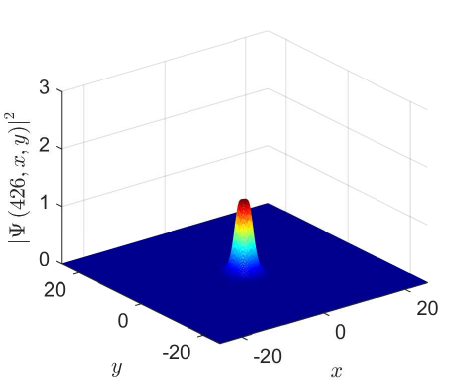}
  \includegraphics[width=0.2\textwidth]{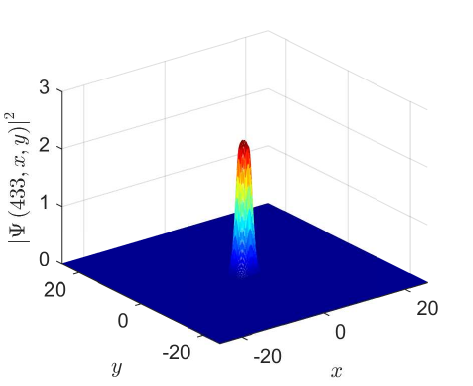}
  \includegraphics[width=0.2\textwidth]{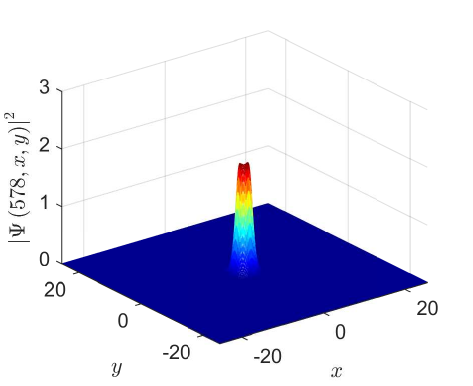}
  \includegraphics[width=0.2\textwidth]{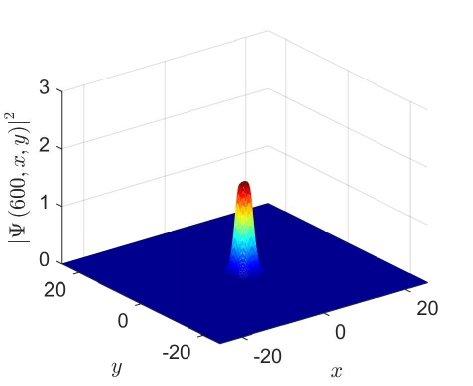} 
  \caption{Example \ref{Example_DG_2D_oscillation}: Charge densities at $t = 0,7,124,234,426,433,578,600$, with $\omega = 0.8$.}
  \label{Figure_DG_2D_sw_3}
\end{figure}

\begin{figure}[htbp]
  \centering
  \includegraphics[width=0.4\textwidth]{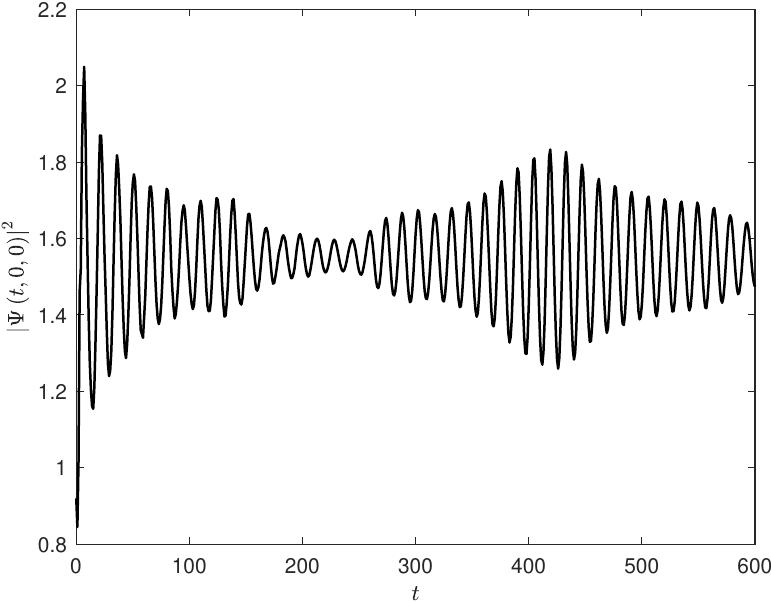} \\
  \caption{Example \ref{Example_DG_2D_oscillation}: Time evolution of the charge density $\left| \Psi \left( t,0,0 \right) \right|^{2}$.}
  \label{Figure_masscenter_2d_3}
\end{figure}

\begin{example}[Travelling wave solutions]\label{Example_DG_2D_TW}\rm
This example  simulate  two travelling wave solutions of the 2D  {NLD equation}.
The computational domain is taken as $\left[ -20,20 \right]^{2}$.
\end{example}

Figure \ref{Fig_tw_2d} shows the charge densities at several different times, reflecting the motion of the travelling waves, where
the first and second rows are for the cases of $\omega = 0.8$ with $v = - \frac{1}{10}$ and  $\omega = 0.12$ with $v = \frac{1}{10}$, respectively.

\begin{figure}[htbp]
  \centering
  \includegraphics[width=0.2\textwidth]{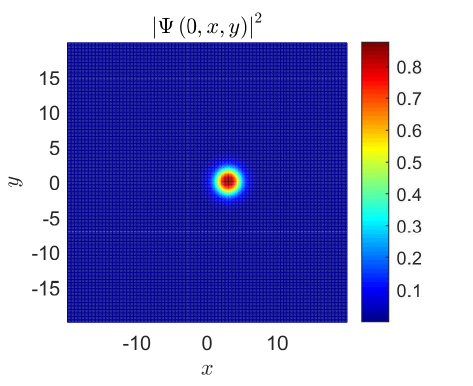}
  \includegraphics[width=0.2\textwidth]{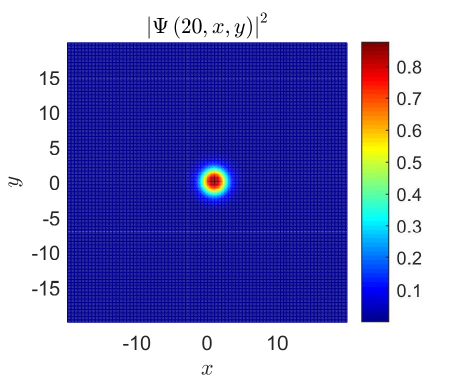}
  \includegraphics[width=0.2\textwidth]{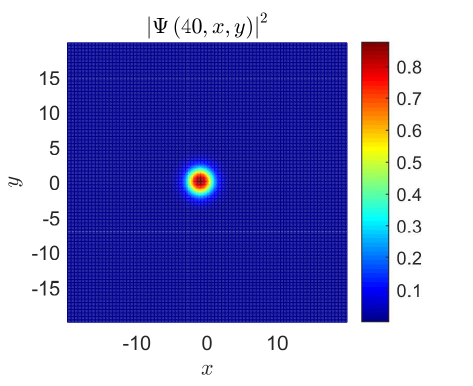}
  \includegraphics[width=0.2\textwidth]{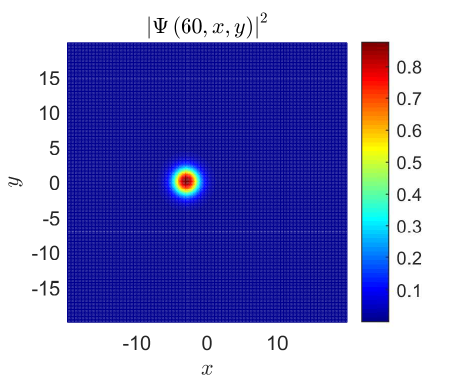} \\
  \includegraphics[width=0.2\textwidth]{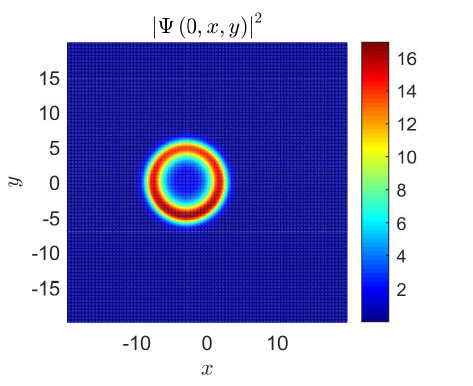}
  \includegraphics[width=0.2\textwidth]{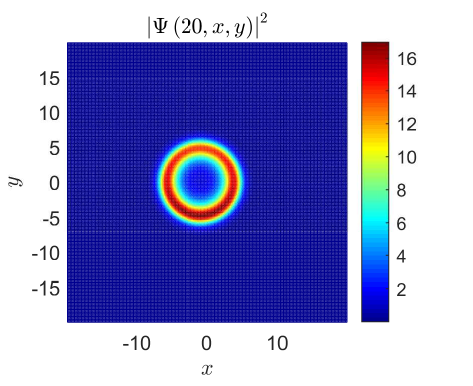}
  \includegraphics[width=0.2\textwidth]{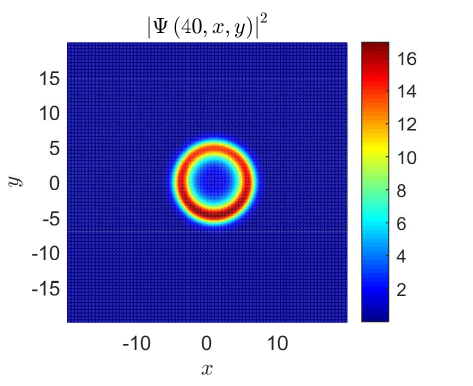}
  \includegraphics[width=0.2\textwidth]{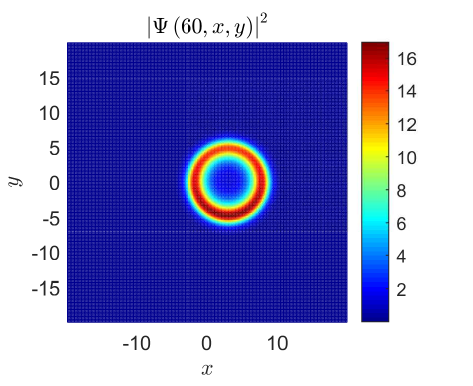} \\
  \caption{Example \ref{Example_DG_2D_TW}: Charge densities at $t = 0,20,40,60$ with $\omega = 0.8$, $v = -\frac{1}{10}$ (up) and $\omega = 0.12$, $v = \frac{1}{10}$ (down).}
  \label{Fig_tw_2d}
\end{figure}

\begin{example}[Breathing pattern]\label{Example_DG_2D_iso}\rm
  The last example investigates the influence of $\kappa$ on the standing wave solution of  {the} 2D {NLD equation}. 
\end{example}

The left plot of Figure \ref{Fig_sw_iso_2d} shows the isosurface (with the value of 0.1) of the charge density from $t = 0$ to $300$ with $\kappa = 2$ and
 $\omega = 0.94$. It presents a breathing pattern. For comparison,
 the right plot of Figure \ref{Fig_sw_iso_2d} gives corresponding result for the case of $\kappa = 1$, where no  breathing pattern is observed.
\begin{figure}[htbp]
  \centering
  \includegraphics[width=0.4\textwidth]{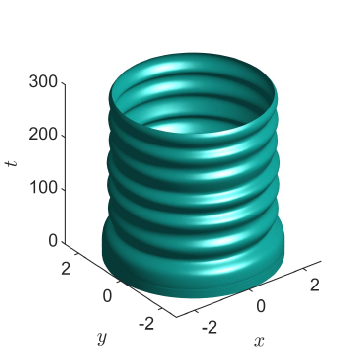}
  \includegraphics[width=0.4\textwidth]{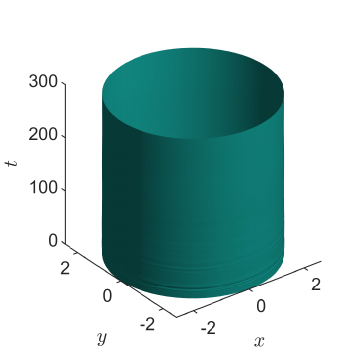} \\
  \caption{Example \ref{Example_DG_2D_iso}: Isosurfaces of the charge density with $\omega = 0.94$. Left: $\kappa = 2$; right: $\kappa = 1$.}
  \label{Fig_sw_iso_2d}
\end{figure}


\section{Conclusion}
\label{Sec_conclusion}
Based on the Runge-Kutta time discretization, the Lax-Wendroff type time discretization, and the two-stage fourth-order time discretization,
this paper developed
three high-order accurate DG methods  for the 1D and 2D  {NLD equations} with a general scalar self-interaction, denoted respectively by RKDG, LWDG and TSDG.
The RKDG method used the spatial DG approximation
    to discretize the NLD equations and then utilized the explicit multistage Runge-Kutta time discretization for the first-order time derivatives, while the LWDG   and   TSDG methods, on the contrary, first gave the one-stage fourth-order Lax-Wendroff type and the two-stage fourth-order time discretizations of the NLD equations, respectively, and then discretized the first- and higher-order spatial derivatives by using the spatial DG approximation.
    For the 2D semi-discrete DG methods with the Lax-Friedrichs flux, we proved the $L^{2}$ stability, that is, the total charge does not increase.
    Moreover,  those three DG methods was compared:
(1) The estimation of their computational complexities in the 1D case showed
that  the computational complexity of the one-stage LWDG
 method was higher than the other two schemes. It was also verified by ous numerical experiments with MATLAB and C++. The main reason was that  the LWDG method needed to calculate the high-order spatial derivatives of the solution and  the nonlinear term, while  the TSDG method  only calculated the first-order derivatives of the nonlinear term and the RKDG method did not require to calculate the derivatives of the nonlinear term.
 (2)  {Recording} the $L^{\infty}$ error in a long time simulation  showed that the RKDG method performed relatively better  than the other two methods.
Several numerical examples were given to verify the above findings, and  the accuracy and the conservative properties of the proposed methods.
In addition, we also simulated the interaction of the 1D solitary waves and the 2D standing and travelling wave solutions. Specially, the breathing pattern was observed clearly in the case of $\kappa = 2$.
 To conduct the 2D numerical experiments, the travelling wave solutions of the 2D  {NLD equation} were given,  according to the standing wave solutions obtained in \cite{Cuevas-Maraver2018} and the Lorentz transformation. Unlike the standing wave solution, the travelling wave solution was not centrosymmetric. 

\section*{Acknowledgement}
The second author was partially supported by
the National Natural Science Foundation of China (No. 11421101).

\appendix

\section{Calculation of $\mathcal{F}$ and $\mathcal{G}$ in \eqref{F_u} and \eqref{G_u}}
\label{Appendix_details}


To calculate $\bm{\mathcal{F}} \left( \bm{u} \right)$, one needs to compute high-order (up to third-order) time derivatives of $\bm{f} \left( \bm{u} \right)$. Those time derivatives can be replaced with the spatial derivatives of $\bm{u}$, thanks to the NLD equation \eqref{2dNLDE4comp}.

Using the definition of $\bm{f} \left( \bm{u} \right)$ in  \eqref{2dNLDE4comp} gives
\begin{equation}\label{f_ut}
  \partial_{t} \bm{f} \left( \bm{u} \right) = \left( \alpha \partial_{t} \bm{u},\beta \partial_{t} \bm{u} \right),
\end{equation}
where $\partial_{t} \bm{u}$ is calculated  from the NLD equation \eqref{2dNLDE4comp} directly.
%
Using \eqref{f_ut} gives
\begin{equation}\label{f_utt}
  \partial_{tt} \bm{f} \left( \bm{u} \right) = \left( \alpha \partial_{tt} \bm{u},\beta \partial_{tt} \bm{u} \right),
\end{equation}
where  $\partial_{tt} \bm{u}$ is calculated by using the NLD equation \eqref{2dNLDE4comp} as follows
\begin{equation}\label{u_tt}
  \partial_{tt} \bm{u} = - \alpha \partial_{tx} \bm{u} - \beta \partial_{ty} \bm{u} + \partial_{t}  \bm{\mathcal{M}} \left( \bm{u} \right).
\end{equation}
Here
\begin{align}
  \partial_{tx} \bm{u} =&\ - \alpha \partial_{xx} \bm{u} - \beta \partial_{xy} \bm{u} + \partial_{x}  \bm{\mathcal{M}} \left( \bm{u} \right) \notag \\
  =&\ - \alpha \partial_{xx} \bm{u} - \beta \partial_{xy} \bm{u} + \partial_{x} g \left( \rho \right) \gamma \bm{u} + g \left( \rho \right) \gamma \partial_{x} \bm{u}, \label{u_tx} \\
  \partial_{ty} \bm{u} =&\ - \alpha \partial_{xy} \bm{u} - \beta \partial_{yy} \bm{u} + \partial_{y}  \bm{\mathcal{M}} \left( \bm{u} \right) \notag \\
  =&\ - \alpha \partial_{xy} \bm{u} - \beta \partial_{yy} \bm{u} + \partial_{y} g \left( \rho \right) \gamma \bm{u} + g \left( \rho \right) \gamma \partial_{y} \bm{u}, \label{u_ty} \\
  \partial_{t}  \bm{\mathcal{M}} \left( \bm{u} \right) =&\ \partial_{t} g \left( \rho \right) \gamma \bm{u} + g \left( \rho \right) \gamma \partial_{t} \bm{u}, \label{M_ut}
\end{align}
and
\begin{align}
  \partial_{t} g \left( \rho \right) =&\ - \left( \kappa + 1 \right) \kappa \lambda \rho^{\kappa - 1} \partial_{t} \rho,\ \partial_{t} \rho = 2 \sum_{p=1}^{4} \left( -1 \right)^{p - 1} \left( u_{p} \partial_{t} u_{p} \right), \label{grho_t} \\
  \partial_{x} g \left( \rho \right) =&\ - \left( \kappa + 1 \right) \kappa \lambda \rho^{\kappa - 1} \partial_{x} \rho,\ \partial_{x} \rho = 2 \sum_{p=1}^{4} \left( -1 \right)^{p - 1} \left( u_{p} \partial_{x} u_{p} \right), \label{grho_x} \\
  \partial_{y} g \left( \rho \right) =&\ - \left( \kappa + 1 \right) \kappa \lambda \rho^{\kappa - 1} \partial_{y} \rho,\ \partial_{y} \rho = 2 \sum_{p=1}^{4} \left( -1 \right)^{p - 1} \left( u_{p} \partial_{y} u_{p} \right). \label{grho_y}
\end{align}

Using \eqref{f_utt}  {further} gives
\begin{equation}\label{f_uttt}
  \partial_{ttt} \bm{f} \left( \bm{u} \right) = \left( \alpha \partial_{ttt} \bm{u},\beta \partial_{ttt} \bm{u} \right),
\end{equation}
where $\partial_{ttt} \bm{u}$ is  computed  from \eqref{2dNLDE4comp} or \eqref{u_tt} as follows
\begin{equation}\label{u_ttt}
  \partial_{ttt} \bm{u} = - \alpha \partial_{ttx} \bm{u} - \beta \partial_{tty} \bm{u} + \partial_{tt}  \bm{\mathcal{M}} \left( \bm{u} \right).
\end{equation}
Here
\begin{align}
  \partial_{ttx} \bm{u} =&\ - \alpha \partial_{txx} \bm{u} - \beta \partial_{txy} \bm{u} + \partial_{tx}  \bm{\mathcal{M}} \left( \bm{u} \right), \label{u_ttx} \\
  \partial_{tty} \bm{u} =&\ - \alpha \partial_{txy} \bm{u} - \beta \partial_{tyy} \bm{u} + \partial_{ty}  \bm{\mathcal{M}} \left( \bm{u} \right), \label{u_tty} \\
  \partial_{tt}  \bm{\mathcal{M}} \left( \bm{u} \right) =&\ \partial_{tt} g \left( \rho \right) \gamma \bm{u} + 2 \partial_{t} g \left( \rho \right) \gamma \partial_{t} \bm{u} + g \left( \rho \right) \gamma \partial_{tt} \bm{u}, \label{M_utt}
\end{align}
with 
\begin{align*}
  \partial_{txx} \bm{u} =&\ - \alpha \partial_{xxx} \bm{u} - \beta \partial_{xxy} \bm{u} + \partial_{xx}  \bm{\mathcal{M}} \left( \bm{u} \right), \\
  \partial_{txy} \bm{u} =&\ - \alpha \partial_{xxy} \bm{u} - \beta \partial_{xyy} \bm{u} + \partial_{xy}  \bm{\mathcal{M}} \left( \bm{u} \right), \\
  \partial_{tyy} \bm{u} =&\ - \alpha \partial_{xyy} \bm{u} - \beta \partial_{yyy} \bm{u} + \partial_{yy} \bm{\mathcal{M}}\left( \bm{u} \right), \\
  \partial_{tx}  \bm{\mathcal{M}} \left( \bm{u} \right) =&\ \partial_{tx} g \left( \rho \right) \gamma \bm{u} + \partial_{t} g \left( \rho \right) \gamma \partial_{x} \bm{u} + \partial_{x} g \left( \rho \right) \gamma \partial_{t} \bm{u} + g \left( \rho \right) \gamma \partial_{tx} \bm{u}, \\
  \partial_{ty}  \bm{\mathcal{M}} \left( \bm{u} \right) =&\ \partial_{ty} g \left( \rho \right) \gamma \bm{u} + \partial_{t} g \left( \rho \right) \gamma \partial_{y} \bm{u} + \partial_{y} g \left( \rho \right) \gamma \partial_{t} \bm{u} + g \left( \rho \right) \gamma \partial_{ty} \bm{u},
\end{align*}
and
\begin{align*}
  \partial_{xx}  \bm{\mathcal{M}} \left( \bm{u} \right) =&\ \partial_{xx} g \left( \rho \right) \gamma \bm{u} + 2 \partial_{x} g \left( \rho \right) \gamma \partial_{x} \bm{u} + g \left( \rho \right) \gamma \partial_{xx} \bm{u}, \\
  \partial_{xy}  \bm{\mathcal{M}} \left( \bm{u} \right) =&\ \partial_{xy} g \left( \rho \right) \gamma \bm{u} + \partial_{x} g \left( \rho \right) \gamma \partial_{y} \bm{u} + \partial_{y} g \left( \rho \right) \gamma \partial_{x} \bm{u} + g \left( \rho \right) \gamma \partial_{xy} \bm{u}, \\
  \partial_{yy}  \bm{\mathcal{M}} \left( \bm{u} \right) =&\ \partial_{yy} g \left( \rho \right) \gamma \bm{u} + 2 \partial_{y} g \left( \rho \right) \gamma \partial_{y} \bm{u} + g \left( \rho \right) \gamma \partial_{yy} \bm{u}, \\
  \partial_{tx} g \left( \rho \right) =&\ - \left( \kappa + 1 \right) \kappa \lambda \left[ \left( \kappa -1 \right) \rho^{\kappa - 2} \partial_{x} \rho \partial_{t} \rho + \rho^{\kappa - 1} \partial_{tx} \rho \right], \\
  \partial_{tx} \rho =&\ 2 \sum_{p=1}^{4} \left( -1 \right)^{p - 1} \left( \partial_{x} u_{p,h} \partial_{t} u_{p,h} + u_{p,h} \partial_{tx} u_{p,h} \right), \\
  \partial_{ty} g \left( \rho \right) =&\ - \left( \kappa + 1 \right) \kappa \lambda \left[ \left( \kappa -1 \right) \rho^{\kappa - 2} \partial_{y} \rho \partial_{t} \rho + \rho^{\kappa - 1} \partial_{ty} \rho \right], \\
  \partial_{ty} \rho =&\ 2 \sum_{p=1}^{4} \left( -1 \right)^{p - 1} \left( \partial_{y} u_{p,h} \partial_{t} u_{p,h} + u_{p,h} \partial_{ty} u_{p,h} \right), \\
  \partial_{xx} g \left( \rho \right) =&\ - \left( \kappa + 1 \right) \kappa \lambda \left[ \left( \kappa - 1 \right) \rho^{\kappa - 2} \left( \partial_{x} \rho \right)^{2} + \rho^{\kappa - 1} \partial_{xx} \rho \right], \\
  \partial_{xx} \rho =&\ 2 \sum_{p=1}^{4} \left( -1 \right)^{p - 1} \left( \left( \partial_{x} u_{p,h} \right)^{2} + u_{p,h} \partial_{xx} u_{p,h} \right), \\
  \partial_{xy} g \left( \rho \right) =&\ - \left( \kappa + 1 \right) \kappa \lambda \left[ \left( \kappa -1 \right) \rho^{\kappa - 2} \partial_{y} \rho \partial_{x} \rho + \rho^{\kappa - 1} \partial_{xy} \rho \right], \notag \\
  \partial_{xy} \rho =&\ 2 \sum_{p=1}^{4} \left( -1 \right)^{p - 1} \left( \partial_{y} u_{p,h} \partial_{x} u_{p,h} + u_{p,h} \partial_{xy} u_{p,h} \right), \\
  \partial_{yy} g \left( \rho \right) =&\ - \left( \kappa + 1 \right) \kappa \lambda \left[ \left( \kappa - 1 \right) \rho^{\kappa - 2} \left( \partial_{y} \rho \right)^{2} + \rho^{\kappa - 1} \partial_{yy} \rho \right], \\
  \partial_{yy} \rho =&\ 2 \sum_{p=1}^{4} \left( -1 \right)^{p - 1} \left( \left( \partial_{y} u_{p,h} \right)^{2} + u_{p,h} \partial_{yy} u_{p,h} \right).
\end{align*}
The terms $\partial_{tt} g$ and $\partial_{tt} \rho$ in \eqref{M_utt}
are calculated as follows
\begin{align}
  \partial_{tt} g \left( \rho \right) =&\ - \left( \kappa + 1 \right) \kappa \lambda \left[ \left( \kappa - 1 \right) \rho^{\kappa - 2} \left( \partial_{t} \rho \right)^{2} + \rho^{\kappa - 1} \partial_{tt} \rho \right], \label{grho_tt} \\
  \partial_{tt} \rho =&\ 2 \sum_{p=1}^{4} \left( -1 \right)^{p - 1} \left( \left( \partial_{t} u_{p,h} \right)^{2} + u_{p,h} \partial_{tt} u_{p,h} \right). \label{rho_tt}
\end{align}
In the above equations, $\partial_{t} \bm{u}$, $\partial_{tt} \bm{u}$, $\partial_{tx} \bm{u}$, $\partial_{ty} \bm{u}$, $\partial_{t} g \left( \rho \right)$, $\partial_{t} \rho$, $\partial_{x} g \left( \rho \right)$, $\partial_{x} \rho$, $\partial_{y} g \left( \rho \right)$, and $\partial_{y} \rho$ can be obtained by  \eqref{2dNLDE4comp}, \eqref{u_tt}, \eqref{u_tx}, \eqref{u_ty}, \eqref{grho_t}, \eqref{grho_x} and \eqref{grho_y}.
Finally, substituting \eqref{f_ut}, \eqref{f_utt} and \eqref{f_uttt} into Eq. \eqref{F_u} gives $\bm{\mathcal{F}} \left( \bm{u} \right)$.

Let us calculate
  $\bm{\mathcal{G}} \left( \bm{u} \right)$. One needs to compute high-order (up to third-order) time derivatives of $ \bm{\mathcal{M}} \left( \bm{u} \right)$.
In fact, we have obtained $\partial_{t}  \bm{\mathcal{M}} \left( \bm{u} \right)$ and $\partial_{tt}  \bm{\mathcal{M}} \left( \bm{u} \right)$ in the above calculations, see   \eqref{M_ut} and \eqref{M_utt}. For $\partial_{ttt}  \bm{\mathcal{M}} \left( \bm{u} \right)$, one has
\begin{equation}\label{M_uttt}
  \partial_{ttt}  \bm{\mathcal{M}} \left( \bm{u} \right) = \partial_{ttt} g \left( \rho \right) \gamma \bm{u} + 3 \partial_{tt} g \left( \rho \right) \gamma \partial_{t} \bm{u}+ 3 \partial_{t} g \left( \rho \right) \gamma \partial_{tt} \bm{u} + g \left( \rho \right) \gamma \partial_{ttt} \bm{u},
\end{equation}
where
\begin{align*}
  \partial_{ttt} g \left( \rho \right) =&\ - \left( \kappa + 1 \right) \kappa \lambda \left[ \left( \kappa - 1 \right) \left( \kappa - 2 \right) \rho^{\kappa - 3} \partial_{t} \rho + 2 \left( \kappa -1 \right) \rho^{\kappa - 2} \partial_{tt} \rho + \rho^{\kappa - 1} \partial_{ttt} \rho \right], \\
  \partial_{ttt} \rho =&\ 2 \sum_{p=1}^{4} \left( -1 \right)^{p - 1} \left( 3 \partial_{t} u_{p,h} \partial_{tt} u_{p,h} + u_{p,h} \partial_{ttt} u_{p,h} \right).
\end{align*}

We remark here that $\partial_{t} \bm{u}$, $\partial_{tt} \bm{u}$, $\partial_{ttt} \bm{u}$, 
$\partial_{t} g \left( \rho \right)$, $\partial_{t} \rho$, $\partial_{tt} g \left( \rho \right)$ and $\partial_{tt} \rho$ in the above equations can be obtained by using \eqref{2dNLDE4comp}, \eqref{u_tt}, \eqref{grho_t}, \eqref{u_ttt}, \eqref{grho_tt} and \eqref{rho_tt}. Finally, substituting \eqref{M_ut}, \eqref{M_utt} and \eqref{M_uttt} into Eq. \eqref{G_u} gives $\bm{\mathcal{G}} \left( \bm{u} \right)$.

\section{Pseudo codes}
\label{Appendix_codes}

The pseudo codes are given here for executing  three 1D DG methods. The numbers in each line represent the needed amount of the corresponding operations.

{\scriptsize
\begin{breakablealgorithm}
  \caption{Pseudo codes for $P^{2}$-LWDG}
  \begin{algorithmic}[1]
  \Require The given initial data $u_{p,j}^{\left(l\right)} \left( t_{0}=0 \right),\  {p=1,2,3,4},\ l=0,1,2;$
  \Ensure $u_{p,j}^{\left(l\right)} \left( T \right),$ $T$: the final time;
  \State Set $a_{j}^{\left(0\right)} =a_{0} = \Delta x;\ a_{j}^{\left(1\right)} = a_{1} = \frac{\Delta x^{3}}{12};\ a_{j}^{\left(2\right)} = a_{2} = \frac{\Delta x^{5}}{180};\ C_{1}=\frac{\Delta x}{2};\ C_{2}=\frac{\Delta x^{2}}{6};\ C_{6}=\frac{C_{2}}{2};\ \Lambda=2 \lambda;\ \widehat{\Lambda}=4 \lambda;\ \tau = \frac{\mu \Delta x}{2*2+1};$
  \For{$k=1:P$} ({$P$: number of Gaussian points})
  \State $D_{1k}=C_{1} \tilde{x}_{k};\ D_{2k}=D_{1k}^{2} - C_{6};\ D_{4k}=2D_{1k};\ W_{1k}=C_{1} \omega_{k};$
  \EndFor

  \State Set time$=0$; $k=-1$;
  \While{time$<T$}
  \State \textcolor[rgb]{1.00,0.00,0.00}{$1+,0\times$} $k=k+1$;
  \If{time$+\tau>T$}
  \State $\tau = T-$time;
  \EndIf
  \State \textcolor[rgb]{1.00,0.00,0.00}{$\ 0+,\ 8\times$} $t_{1}=\frac{\tau}{2};\ t_{2}=\frac{\tau^{2}}{6};\ t_{3}=\frac{t_{1}t_{2}}{2};\ T_{0}=\frac{\tau}{a_{0}};\ T_{1}=\frac{\tau}{a_{1}};\ T_{2}=\frac{\tau}{a_{2}};$
  \State \textcolor[rgb]{1.00,0.00,0.00}{Compute the left and right limits at cell interface.}
  \For{$j = 1:J$}
  \State \textcolor[rgb]{1.00,0.00,1.00}{$\ 12+,\ 8\times$} $L_{p} = u_{p,j}^{\left(0\right)} \left(t_{k} \right) + C_{2} u_{p,j}^{\left(2\right)} \left(t_{k} \right);\  R_{p} = C_{1} u_{p,j}^{\left(1\right)} \left(t_{k} \right);\ u_{p,j+\frac{1}{2}}^{-} = L_{p} + R_{p};\ u_{p,j-\frac{1}{2}}^{+} = L_{p} - R_{p};$
  \State \textcolor[rgb]{1.00,0.00,1.00}{$\ 8+,\ 4\times$} $R_{p} = \Delta x u_{p,j}^{\left(2\right)} \left(t_{k} \right);\ \left(u_{x}\right)_{p,j+\frac{1}{2}}^{-} = u_{p,j}^{\left(1\right)} \left(t_{k} \right) + R_{p};\ \left(u_{x}\right)_{p,j-\frac{1}{2}}^{+} = u_{p,j}^{\left(1\right)} \left(t_{k} \right) - R_{p};$
  \State \textcolor[rgb]{1.00,0.00,1.00}{$\ 0+,\ 4\times$} $L_{p} = 2 u_{p,j}^{\left(2\right)} \left(t_{k} \right);\ \left(u_{xx}\right)_{p,j+\frac{1}{2}}^{-} = L_{p};\ \left(u_{xx}\right)_{p,j-\frac{1}{2}}^{+} = L_{p};\ \left(u_{xxx}\right)_{p,j+\frac{1}{2}}^{-} = 0;\ \left(u_{xxx}\right)_{p,j-\frac{1}{2}}^{+} = 0;$
  \EndFor
  \State \textcolor[rgb]{1.00,0.00,0.00}{Remark: Boundary conditions}
  \State $u_{p\frac{1}{2}}^{-} = 0;\ u_{pJ+\frac{1}{2}}^{+} = 0;\ \left(u_{x}\right)_{p\frac{1}{2}}^{-} = 0;\ \left(u_{x}\right)_{pJ+\frac{1}{2}}^{+} = 0;\ \left(u_{xx}\right)_{p\frac{1}{2}}^{-} = 0;\ \left(u_{xx}\right)_{pJ+\frac{1}{2}}^{+} = 0;\ \left(u_{xxx}\right)_{p\frac{1}{2}}^{-} = 0;\ \left(u_{xxx}\right)_{pJ+\frac{1}{2}}^{+} = 0;$
  \State \textcolor[rgb]{1.00,0.00,0.00}{Compute the flux at cell interface.}
  \For{$j = 0:J$}
  \State \textcolor[rgb]{1.00,0.00,0.00}{$\ 8+,10 \times$} $\widetilde{\rho}^{\pm} = m - 2\lambda \rho_{j+\frac{1}{2}}^{\pm} = m - \Lambda \sum \limits_{p=1}^{4} \left( -1 \right)^{p+1} \left( u_{p,j+\frac{1}{2}}^{\pm} \right)^{2};$
  \State \textcolor[rgb]{1.00,0.00,0.00}{$\ 6+,10 \times$} $\widetilde{\rho}_{x}^{\pm} = 2\lambda \left(\rho_{x}\right)_{j+\frac{1}{2}}^{\pm} = \widehat{\Lambda} \sum \limits_{p=1}^{4} \left( -1 \right)^{p+1} \left( u_{p,j+\frac{1}{2}}^{\pm} \left(u_{x}\right)_{p,j+\frac{1}{2}}^{\pm} \right);$
  \State \textcolor[rgb]{1.00,0.00,0.00}{$\ 4+,\ 4\times$} $\left(u_{t}\right)_{1,j+\frac{1}{2}}^{\pm} = - \left(u_{x}\right)_{2,j+\frac{1}{2}}^{\pm} + \widetilde{\rho}^{\pm} u_{3,j+\frac{1}{2}}^{\pm};\ \left(u_{t}\right)_{2,j+\frac{1}{2}}^{\pm} = - \left(u_{x}\right)_{1,j+\frac{1}{2}}^{\pm} - \widetilde{\rho}^{\pm} u_{4,j+\frac{1}{2}}^{\pm};$
  \State \textcolor[rgb]{1.00,0.00,0.00}{$\ 4+,\ 4\times$} $\left(u_{t}\right)_{3,j+\frac{1}{2}}^{\pm} = - \left(u_{x}\right)_{4,j+\frac{1}{2}}^{\pm} - \widetilde{\rho}^{\pm} u_{1,j+\frac{1}{2}}^{\pm};\ \left(u_{t}\right)_{4,j+\frac{1}{2}}^{\pm} = - \left(u_{x}\right)_{3,j+\frac{1}{2}}^{\pm} + \widetilde{\rho}^{\pm} u_{2,j+\frac{1}{2}}^{\pm};$
  \State \textcolor[rgb]{1.00,0.00,0.00}{$\ 6+,10 \times$} $\widetilde{\rho}_{t}^{\pm} = 2\lambda \left(\rho_{t}\right)_{j+\frac{1}{2}}^{\pm} = \widehat{\Lambda} \sum \limits_{p=1}^{4} \left( -1 \right)^{p+1} \left( u_{p,j+\frac{1}{2}}^{\pm} \left(u_{t}\right)_{p,j+\frac{1}{2}}^{\pm} \right);$
  \State \textcolor[rgb]{1.00,0.00,0.00}{$\ 4+,\ 8\times$} $M_{1,x}^{\pm} = - \widetilde{\rho}_{x}^{\pm} u_{3,j+\frac{1}{2}}^{\pm} + \widetilde{\rho}^{\pm} \left( u_{x}\right)_{3,j+\frac{1}{2}}^{\pm};\ M_{2,x}^{\pm} = \widetilde{\rho}_{x}^{\pm} u_{4,j+\frac{1}{2}}^{\pm} - \widetilde{\rho}^{\pm} \left( u_{x}\right)_{4,j+\frac{1}{2}}^{\pm};$
  \State \textcolor[rgb]{1.00,0.00,0.00}{$\ 4+,\ 8\times$} $M_{3,x}^{\pm} = \widetilde{\rho}_{x}^{\pm} u_{1,j+\frac{1}{2}}^{\pm} - \widetilde{\rho}^{\pm} \left( u_{x}\right)_{1,j+\frac{1}{2}}^{\pm};\ M_{4,x}^{\pm} = - \widetilde{\rho}_{x}^{\pm} u_{2,j+\frac{1}{2}}^{\pm} + \widetilde{\rho}^{\pm} \left( u_{x}\right)_{2,j+\frac{1}{2}}^{\pm};$
  \State \textcolor[rgb]{1.00,0.00,0.00}{$\ 4+,\ 8\times$} $M_{1,t}^{\pm} = - \widetilde{\rho}_{t}^{\pm} u_{3,j+\frac{1}{2}}^{\pm} + \widetilde{\rho}^{\pm} \left( u_{t}\right)_{3,j+\frac{1}{2}}^{\pm};\ M_{2,t}^{\pm} = \widetilde{\rho}_{t}^{\pm} u_{4,j+\frac{1}{2}}^{\pm} - \widetilde{\rho}^{\pm} \left( u_{t}\right)_{4,j+\frac{1}{2}}^{\pm};$
  \State \textcolor[rgb]{1.00,0.00,0.00}{$\ 4+,\ 8\times$} $M_{3,t}^{\pm} = \widetilde{\rho}_{t}^{\pm} u_{1,j+\frac{1}{2}}^{\pm} - \widetilde{\rho}^{\pm} \left( u_{t}\right)_{1,j+\frac{1}{2}}^{\pm};\ M_{4,t}^{\pm} = - \widetilde{\rho}_{t}^{\pm} u_{2,j+\frac{1}{2}}^{\pm} + \widetilde{\rho}^{\pm} \left( u_{t}\right)_{2,j+\frac{1}{2}}^{\pm};$
  \State \textcolor[rgb]{1.00,0.00,0.00}{$\ 4+,\ 0\times$} $\left(u_{tx}\right)_{1,j+\frac{1}{2}}^{\pm} = - \left(u_{xx}\right)_{2,j+\frac{1}{2}}^{\pm} + M_{1,x}^{\pm};\ \left(u_{tx}\right)_{2,j+\frac{1}{2}}^{\pm} = - \left(u_{xx}\right)_{1,j+\frac{1}{2}}^{\pm} + M_{2,x}^{\pm};$
  \State \textcolor[rgb]{1.00,0.00,0.00}{$\ 4+,\ 0\times$} $\left(u_{tx}\right)_{3,j+\frac{1}{2}}^{\pm} = - \left(u_{xx}\right)_{4,j+\frac{1}{2}}^{\pm} + M_{3,x}^{\pm};\ \left(u_{tx}\right)_{4,j+\frac{1}{2}}^{\pm} = - \left(u_{xx}\right)_{3,j+\frac{1}{2}}^{\pm} + M_{4,x}^{\pm};$
  \State \textcolor[rgb]{1.00,0.00,0.00}{$\ 4+,\ 0\times$} $\left(u_{tt}\right)_{1,j+\frac{1}{2}}^{\pm} = - \left(u_{tx}\right)_{2,j+\frac{1}{2}}^{\pm} + M_{1,t}^{\pm};\ \left(u_{tt}\right)_{2,j+\frac{1}{2}}^{\pm} = - \left(u_{tx}\right)_{1,j+\frac{1}{2}}^{\pm} + M_{2,t}^{\pm};$
  \State \textcolor[rgb]{1.00,0.00,0.00}{$\ 4+,\ 0\times$} $\left(u_{tt}\right)_{3,j+\frac{1}{2}}^{\pm} = - \left(u_{tx}\right)_{4,j+\frac{1}{2}}^{\pm} + M_{3,t}^{\pm};\ \left(u_{tt}\right)_{4,j+\frac{1}{2}}^{\pm} = - \left(u_{tx}\right)_{3,j+\frac{1}{2}}^{\pm} + M_{4,t}^{\pm};$
  \State \textcolor[rgb]{1.00,0.00,0.00}{$14 +,18 \times$} $\widetilde{\rho}_{xx}^{\pm} = 2\lambda \left(\rho_{xx}\right)_{j+\frac{1}{2}}^{\pm} = \widehat{\Lambda} \sum \limits_{p=1}^{4} \left( -1 \right)^{p+1} \left( \left( \left(u_{x}\right)_{p,j+\frac{1}{2}}^{\pm} \right)^{2} + u_{p,j+\frac{1}{2}}^{\pm} \left(u_{xx}\right)_{p,j+\frac{1}{2}}^{\pm} \right);$
  \State \textcolor[rgb]{1.00,0.00,0.00}{$14 +,18 \times$} $\widetilde{\rho}_{tx}^{\pm} = 2\lambda \left(\rho_{tx}\right)_{j+\frac{1}{2}}^{\pm} = \widehat{\Lambda} \sum \limits_{p=1}^{4} \left( -1 \right)^{p+1} \left( \left(u_{x}\right)_{p,j+\frac{1}{2}}^{\pm} \left(u_{t}\right)_{p,j+\frac{1}{2}}^{\pm} + u_{p,j+\frac{1}{2}}^{\pm} \left(u_{tx}\right)_{p,j+\frac{1}{2}}^{\pm} \right);$
  \State \textcolor[rgb]{1.00,0.00,0.00}{$14 +,18 \times$} $\widetilde{\rho}_{tt}^{\pm} = 2\lambda \left(\rho_{tt}\right)_{j+\frac{1}{2}}^{\pm} = \widehat{\Lambda} \sum \limits_{p=1}^{4} \left( -1 \right)^{p+1} \left( \left( \left(u_{t}\right)_{p,j+\frac{1}{2}}^{\pm} \right)^{2} + u_{p,j+\frac{1}{2}}^{\pm} \left(u_{tt}\right)_{p,j+\frac{1}{2}}^{\pm} \right);$
  \State \textcolor[rgb]{1.00,0.00,0.00}{$\ 4+,\ 8\times$} $M_{1,xx}^{\pm} = - \widetilde{\rho}_{xx}^{\pm} u_{3,j+\frac{1}{2}}^{\pm} - 2\widetilde{\rho}_{x}^{\pm} \left( u_{x}\right)_{3,j+\frac{1}{2}}^{\pm} + \widetilde{\rho}^{\pm} \left( u_{xx}\right)_{3,j+\frac{1}{2}}^{\pm};$
  \State \textcolor[rgb]{1.00,0.00,0.00}{$\ 4+,\ 8\times$} $M_{2,xx}^{\pm} = \widetilde{\rho}_{xx}^{\pm} u_{4,j+\frac{1}{2}}^{\pm} + 2\widetilde{\rho}_{x}^{\pm} \left( u_{x}\right)_{4,j+\frac{1}{2}}^{\pm} - \widetilde{\rho}^{\pm} \left( u_{xx}\right)_{4,j+\frac{1}{2}}^{\pm};$
  \State \textcolor[rgb]{1.00,0.00,0.00}{$\ 4+,\ 8\times$} $M_{3,xx}^{\pm} = \widetilde{\rho}_{xx}^{\pm} u_{1,j+\frac{1}{2}}^{\pm} + 2\widetilde{\rho}_{x}^{\pm} \left( u_{x}\right)_{1,j+\frac{1}{2}}^{\pm} - \widetilde{\rho}^{\pm} \left( u_{xx}\right)_{1,j+\frac{1}{2}}^{\pm};$
  \State \textcolor[rgb]{1.00,0.00,0.00}{$\ 4+,\ 8\times$} $M_{4,xx}^{\pm} = - \widetilde{\rho}_{xx}^{\pm} u_{2,j+\frac{1}{2}}^{\pm} - 2\widetilde{\rho}_{x}^{\pm} \left( u_{x}\right)_{2,j+\frac{1}{2}}^{\pm} + \widetilde{\rho}^{\pm} \left( u_{xx}\right)_{2,j+\frac{1}{2}}^{\pm};$
  \State \textcolor[rgb]{1.00,0.00,0.00}{$\ 6+,\ 8 \times$} $M_{1,tx}^{\pm} = - \widetilde{\rho}_{tx}^{\pm} u_{3,j+\frac{1}{2}}^{\pm} - \widetilde{\rho}_{t}^{\pm} \left( u_{x}\right)_{3,j+\frac{1}{2}}^{\pm} - \widetilde{\rho}_{x}^{\pm} \left( u_{t}\right)_{3,j+\frac{1}{2}}^{\pm} + \widetilde{\rho}^{\pm} \left( u_{tx}\right)_{3,j+\frac{1}{2}}^{\pm};$
  \State \textcolor[rgb]{1.00,0.00,0.00}{$\ 6+,\ 8 \times$} $M_{2,tx}^{\pm} = \widetilde{\rho}_{tx}^{\pm} u_{4,j+\frac{1}{2}}^{\pm} + \widetilde{\rho}_{t}^{\pm} \left( u_{x}\right)_{4,j+\frac{1}{2}}^{\pm} + \widetilde{\rho}_{x}^{\pm} \left( u_{t}\right)_{4,j+\frac{1}{2}}^{\pm} - \widetilde{\rho}^{\pm} \left( u_{tx}\right)_{4,j+\frac{1}{2}}^{\pm};$
  \State \textcolor[rgb]{1.00,0.00,0.00}{$\ 6+,\ 8 \times$} $M_{3,tx}^{\pm} = \widetilde{\rho}_{tx}^{\pm} u_{1,j+\frac{1}{2}}^{\pm} + \widetilde{\rho}_{t}^{\pm} \left( u_{x}\right)_{1,j+\frac{1}{2}}^{\pm} + \widetilde{\rho}_{x}^{\pm} \left( u_{t}\right)_{1,j+\frac{1}{2}}^{\pm} - \widetilde{\rho}^{\pm} \left( u_{tx}\right)_{1,j+\frac{1}{2}}^{\pm};$
  \State \textcolor[rgb]{1.00,0.00,0.00}{$\ 6+,\ 8 \times$} $M_{4,tx}^{\pm} = - \widetilde{\rho}_{tx}^{\pm} u_{2,j+\frac{1}{2}}^{\pm} - \widetilde{\rho}_{t}^{\pm} \left( u_{x}\right)_{2,j+\frac{1}{2}}^{\pm} - \widetilde{\rho}_{x}^{\pm} \left( u_{t}\right)_{2,j+\frac{1}{2}}^{\pm} + \widetilde{\rho}^{\pm} \left( u_{tx}\right)_{2,j+\frac{1}{2}}^{\pm};$
  \State \textcolor[rgb]{1.00,0.00,0.00}{$\ 4+,\ 8\times$} $M_{1,tt}^{\pm} = - \widetilde{\rho}_{tt}^{\pm} u_{3,j+\frac{1}{2}}^{\pm} - 2\widetilde{\rho}_{t}^{\pm} \left( u_{t}\right)_{3,j+\frac{1}{2}}^{\pm} + \widetilde{\rho}^{\pm} \left( u_{tt}\right)_{3,j+\frac{1}{2}}^{\pm};$
  \State \textcolor[rgb]{1.00,0.00,0.00}{$\ 4+,\ 8\times$} $M_{2,tt}^{\pm} = \widetilde{\rho}_{tt}^{\pm} u_{4,j+\frac{1}{2}}^{\pm} + 2\widetilde{\rho}_{t}^{\pm} \left( u_{t}\right)_{4,j+\frac{1}{2}}^{\pm} - \widetilde{\rho}^{\pm} \left( u_{tt}\right)_{4,j+\frac{1}{2}}^{\pm};$
  \State \textcolor[rgb]{1.00,0.00,0.00}{$\ 4+,\ 8\times$} $M_{3,tt}^{\pm} = \widetilde{\rho}_{tt}^{\pm} u_{1,j+\frac{1}{2}}^{\pm} + 2\widetilde{\rho}_{t}^{\pm} \left( u_{t}\right)_{1,j+\frac{1}{2}}^{\pm} - \widetilde{\rho}^{\pm} \left( u_{tt}\right)_{1,j+\frac{1}{2}}^{\pm};$
  \State \textcolor[rgb]{1.00,0.00,0.00}{$\ 4+,\ 8\times$} $M_{4,tt}^{\pm} = - \widetilde{\rho}_{tt}^{\pm} u_{2,j+\frac{1}{2}}^{\pm} - 2\widetilde{\rho}_{t}^{\pm} \left( u_{t}\right)_{2,j+\frac{1}{2}}^{\pm} + \widetilde{\rho}^{\pm} \left( u_{tt}\right)_{2,j+\frac{1}{2}}^{\pm};$
  \State \textcolor[rgb]{1.00,0.00,0.00}{$\ 4+,\ 0 \times$} $\left(u_{txx}\right)_{1,j+\frac{1}{2}}^{\pm} = - \left(u_{xxx}\right)_{2,j+\frac{1}{2}}^{\pm} + M_{1,xx}^{\pm};\ \left(u_{txx}\right)_{2,j+\frac{1}{2}}^{\pm} = - \left(u_{xxx}\right)_{1,j+\frac{1}{2}}^{\pm} + M_{2,xx}^{\pm};$
  \State \textcolor[rgb]{1.00,0.00,0.00}{$\ 4+,\ 0 \times$} $\left(u_{txx}\right)_{3,j+\frac{1}{2}}^{\pm} = - \left(u_{xxx}\right)_{4,j+\frac{1}{2}}^{\pm} + M_{3,xx}^{\pm};\ \left(u_{txx}\right)_{4,j+\frac{1}{2}}^{\pm} = - \left(u_{xxx}\right)_{3,j+\frac{1}{2}}^{\pm} + M_{4,xx}^{\pm};$
  \State \textcolor[rgb]{1.00,0.00,0.00}{$\ 4+,\ 0 \times$} $\left(u_{ttx}\right)_{1,j+\frac{1}{2}}^{\pm} = - \left(u_{txx}\right)_{2,j+\frac{1}{2}}^{\pm} + M_{1,tx}^{\pm};\ \left(u_{ttx}\right)_{2,j+\frac{1}{2}}^{\pm} = - \left(u_{txx}\right)_{1,j+\frac{1}{2}}^{\pm} + M_{2,tx}^{\pm};$
  \State \textcolor[rgb]{1.00,0.00,0.00}{$\ 4+,\ 0 \times$} $\left(u_{ttx}\right)_{3,j+\frac{1}{2}}^{\pm} = - \left(u_{txx}\right)_{4,j+\frac{1}{2}}^{\pm} + M_{3,tx}^{\pm};\ \left(u_{ttx}\right)_{4,j+\frac{1}{2}}^{\pm} = - \left(u_{txx}\right)_{3,j+\frac{1}{2}}^{\pm} + M_{4,tx}^{\pm};$
  \State \textcolor[rgb]{1.00,0.00,0.00}{$\ 4+,\ 0 \times$} $\left(u_{ttt}\right)_{1,j+\frac{1}{2}}^{\pm} = - \left(u_{ttx}\right)_{2,j+\frac{1}{2}}^{\pm} + M_{1,tt}^{\pm};\ \left(u_{ttt}\right)_{2,j+\frac{1}{2}}^{\pm} = - \left(u_{ttx}\right)_{1,j+\frac{1}{2}}^{\pm} + M_{2,tt}^{\pm};$
  \State \textcolor[rgb]{1.00,0.00,0.00}{$\ 4+,\ 0 \times$} $\left(u_{ttt}\right)_{3,j+\frac{1}{2}}^{\pm} = - \left(u_{ttx}\right)_{4,j+\frac{1}{2}}^{\pm} + M_{3,tt}^{\pm};\ \left(u_{ttt}\right)_{4,j+\frac{1}{2}}^{\pm} = - \left(u_{ttx}\right)_{3,j+\frac{1}{2}}^{\pm} + M_{4,tt}^{\pm};$
  \State \textcolor[rgb]{1.00,0.00,0.00}{$\ 6+,\ 6\times$} $\mathcal{F}_{1,j+\frac{1}{2}}^{\pm} = u_{2,j+\frac{1}{2}}^{\pm} + t_{1} \left(u_{t}\right)_{2,j+\frac{1}{2}}^{\pm} + t_{2} \left(u_{tt}\right)_{2,j+\frac{1}{2}}^{\pm} + t_{3} \left(u_{ttt}\right)_{2,j+\frac{1}{2}}^{\pm};$
  \State \textcolor[rgb]{1.00,0.00,0.00}{$\ 6+,\ 6\times$} $\mathcal{F}_{2,j+\frac{1}{2}}^{\pm} = u_{1,j+\frac{1}{2}}^{\pm} + t_{1} \left(u_{t}\right)_{1,j+\frac{1}{2}}^{\pm} + t_{2} \left(u_{tt}\right)_{1,j+\frac{1}{2}}^{\pm} + t_{3} \left(u_{ttt}\right)_{1,j+\frac{1}{2}}^{\pm};$
  \State \textcolor[rgb]{1.00,0.00,0.00}{$\ 6+,\ 6\times$} $\mathcal{F}_{3,j+\frac{1}{2}}^{\pm} = u_{4,j+\frac{1}{2}}^{\pm} + t_{1} \left(u_{t}\right)_{4,j+\frac{1}{2}}^{\pm} + t_{2} \left(u_{tt}\right)_{4,j+\frac{1}{2}}^{\pm} + t_{3} \left(u_{ttt}\right)_{4,j+\frac{1}{2}}^{\pm};$
  \State \textcolor[rgb]{1.00,0.00,0.00}{$\ 6+,\ 6\times$} $\mathcal{F}_{4,j+\frac{1}{2}}^{\pm} = u_{3,j+\frac{1}{2}}^{\pm} + t_{1} \left(u_{t}\right)_{3,j+\frac{1}{2}}^{\pm} + t_{2} \left(u_{tt}\right)_{3,j+\frac{1}{2}}^{\pm} + t_{3} \left(u_{ttt}\right)_{3,j+\frac{1}{2}}^{\pm};$
  \State \textcolor[rgb]{1.00,0.00,1.00}{$12+,\ 4\times$} $\widehat{\mathcal{F}}_{p,j+\frac{1}{2}} = \frac{1}{2} \left[ \mathcal{F}_{p,j+\frac{1}{2}}^{-} + \mathcal{F}_{p,j+\frac{1}{2}}^{+} - \left( u_{p,j+\frac{1}{2}}^{+} - u_{p,j+\frac{1}{2}}^{-}\right) \right];$
  \EndFor
  \State Remark: Gaussian quadrature, $\tilde{x}_{k}$-Gaussian points, $\omega_{k}$-weights
  \For{$j=1:J$}
  \State \textcolor[rgb]{1.00,0.00,1.00}{$\ 8+,\ 0\times$} $Q_{1,p,j} = \widehat{\mathcal{F}}_{p,j+\frac{1}{2}} - \widehat{\mathcal{F}}_{p,j-\frac{1}{2}},\ Q_{2,p,j} = \widehat{\mathcal{F}}_{p,j+\frac{1}{2}} + \widehat{\mathcal{F}}_{p,j-\frac{1}{2}};$
  \State $F_{p,j}^{\left(l\right)} = 0,\ l=0,1,2;$
  \For{$k=1:P$}
  \State \textcolor[rgb]{1.00,0.00,1.00}{$\ 8+,\ 8\times$} $u_{p,j,k} = u_{p,j}^{\left(0\right)} \left(t_{k} \right) + u_{p,j}^{\left(1\right)} \left(t_{k} \right) D_{1k} + u_{p,j}^{\left(2\right)} \left(t_{k} \right) D_{2k};$
  \State \textcolor[rgb]{1.00,0.00,1.00}{$\ 4+,\ 8\times$} $\left(u_{x}\right)_{p,j,k} = u_{p,j}^{\left(1\right)} \left(t_{k} \right) + u_{p,j}^{\left(2\right)} \left(t_{k} \right) D_{4k};\ \left(u_{xx}\right)_{p,j,k} = 2 u_{p,j}^{\left(2\right)} \left(t_{k} \right);\ \left(u_{xxx}\right)_{p,j,k} = 0;$
  \State \textcolor[rgb]{1.00,0.00,0.00}{$\ 4+,\ 5\times$} $\widetilde{\rho} = m - 2\lambda \rho_{j,k} = m - \Lambda \left( \left( u_{1,j,k} \right)^{2} + \left( u_{3,j,k} \right)^{2} - \left( u_{2,j,k} \right)^{2} - \left( u_{4,j,k} \right)^{2} \right);$
  \State \textcolor[rgb]{1.00,0.00,0.00}{$\ 0+,\ 4\times$} $\widetilde{M}_{1} = \widetilde{\rho} u_{3,j,k};\ \widetilde{M}_{2} = -\widetilde{\rho} u_{4,j,k};\ \widetilde{M}_{3} = -\widetilde{\rho} u_{1,j,k};\ \widetilde{M}_{4} = \widetilde{\rho} u_{2,j,k};$
  \State \textcolor[rgb]{1.00,0.00,0.00}{$\ 2+,\ 0\times$} $\left(u_{t}\right)_{1,j,k} = - \left(u_{x}\right)_{2,j,k} + \widetilde{M}_{1};\ \left(u_{t}\right)_{2,j,k} = - \left(u_{x}\right)_{1,j,k} + \widetilde{M}_{2};$
  \State \textcolor[rgb]{1.00,0.00,0.00}{$\ 2+,\ 0\times$} $\left(u_{t}\right)_{3,j,k} = - \left(u_{x}\right)_{4,j,k} + \widetilde{M}_{3};\ \left(u_{t}\right)_{4,j,k} = - \left(u_{x}\right)_{3,j,k} + \widetilde{M}_{4};$
  \State \textcolor[rgb]{1.00,0.00,0.00}{$\ 3+,\ 5\times$} $\widetilde{\rho}_{x} = 2\lambda \left(\rho_{x}\right)_{j,k} = \widehat{\Lambda} \sum \limits_{p=1}^{4} \left( -1 \right)^{p+1} \left( u_{p,j,k} \left(u_{x}\right)_{p,j,k} \right);$
  \State \textcolor[rgb]{1.00,0.00,0.00}{$\ 3+,\ 5\times$} $\widetilde{\rho}_{t} = 2\lambda \left(\rho_{t}\right)_{j,k} = \widehat{\Lambda} \sum \limits_{p=1}^{4} \left( -1 \right)^{p+1} \left( u_{p,j,k} \left(u_{t}\right)_{p,j,k} \right);$
  \State \textcolor[rgb]{1.00,0.00,0.00}{$\ 2+,\ 4\times$} $\widetilde{M}_{1,x} = - \widetilde{\rho}_{x} u_{3,j,k} + \widetilde{\rho} \left(u_{x}\right)_{3,j,k};\ \widetilde{M}_{2,x} = \widetilde{\rho}_{x} u_{4,j,k} - \widetilde{\rho} \left(u_{x}\right)_{4,j,k};$
  \State \textcolor[rgb]{1.00,0.00,0.00}{$\ 2+,\ 4\times$} $\widetilde{M}_{3,x} = \widetilde{\rho}_{x} u_{1,j,k} - \widetilde{\rho} \left(u_{x}\right)_{1,j,k};\ \widetilde{M}_{4,x} = - \widetilde{\rho}_{x} u_{2,j,k} + \widetilde{\rho} \left(u_{x}\right)_{2,j,k};$
  \State \textcolor[rgb]{1.00,0.00,0.00}{$\ 2+,\ 4\times$} $\widetilde{M}_{1,t} = - \widetilde{\rho}_{t} u_{3,j,k} + \widetilde{\rho} \left(u_{t}\right)_{3,j,k};\widetilde{M}_{2,t} = \widetilde{\rho}_{t} u_{4,j,k} - \widetilde{\rho} \left(u_{t}\right)_{4,j,k};$
  \State \textcolor[rgb]{1.00,0.00,0.00}{$\ 2+,\ 4\times$} $\widetilde{M}_{3,t} = \widetilde{\rho}_{t} u_{1,j,k} - \widetilde{\rho} \left(u_{t}\right)_{1,j,k};\ \widetilde{M}_{4,t} = - \widetilde{\rho}_{t} u_{2,j,k} + \widetilde{\rho} \left(u_{t}\right)_{2,j,k};$
  \State \textcolor[rgb]{1.00,0.00,0.00}{$\ 2+,\ 0\times$} $\left(u_{tx}\right)_{1,j,k} = - \left(u_{xx}\right)_{2,j,k} + \widetilde{M}_{1,x};\ \left(u_{tx}\right)_{2,j,k} = - \left(u_{xx}\right)_{1,j,k} + \widetilde{M}_{2,x};$
  \State \textcolor[rgb]{1.00,0.00,0.00}{$\ 2+,\ 0\times$} $\left(u_{tx}\right)_{3,j,k} = - \left(u_{xx}\right)_{4,j,k} + \widetilde{M}_{3,x};\ \left(u_{tx}\right)_{4,j,k} = - \left(u_{xx}\right)_{3,j,k} + \widetilde{M}_{4,x};$
  \State \textcolor[rgb]{1.00,0.00,0.00}{$\ 2+,\ 0\times$} $\left(u_{tt}\right)_{1,j,k} = - \left(u_{tx}\right)_{2,j,k} + \widetilde{M}_{1,t};\ \left(u_{tt}\right)_{2,j,k} = - \left(u_{tx}\right)_{1,j,k} + \widetilde{M}_{2,t};$
  \State \textcolor[rgb]{1.00,0.00,0.00}{$\ 2+,\ 0\times$} $\left(u_{tt}\right)_{3,j,k} = - \left(u_{tx}\right)_{4,j,k} + \widetilde{M}_{3,t};\ \left(u_{tt}\right)_{4,j,k} = - \left(u_{tx}\right)_{3,j,k} + \widetilde{M}_{4,t};$
  \State \textcolor[rgb]{1.00,0.00,0.00}{$\ 7+,\ 9\times$} $\widetilde{\rho}_{xx} = 2\lambda \left(\rho_{tt}\right)_{j,k} = \widehat{\Lambda} \sum \limits_{p=1}^{4} \left( -1 \right)^{p+1} \left( \left( \left(u_{x}\right)_{p,j,k} \right)^{2} + u_{p,j,k} \left(u_{xx}\right)_{p,j,k} \right);$
  \State \textcolor[rgb]{1.00,0.00,0.00}{$\ 7+,\ 9\times$} $\widetilde{\rho}_{tx} = 2\lambda \left(\rho_{tx}\right)_{j,k} = \widehat{\Lambda} \sum \limits_{p=1}^{4} \left( -1 \right)^{p+1} \left( \left(u_{x}\right)_{p,j,k} \left(u_{t}\right)_{p,j,k} + u_{p,j,k} \left(u_{tx}\right)_{p,j,k} \right);$
  \State \textcolor[rgb]{1.00,0.00,0.00}{$\ 7+,\ 9\times$} $\widetilde{\rho}_{tt} = 2\lambda \left(\rho_{xx}\right)_{j,k} = \widehat{\Lambda} \sum \limits_{p=1}^{4} \left( -1 \right)^{p+1} \left( \left( \left(u_{t}\right)_{p,j,k} \right)^{2} + u_{p,j,k} \left(u_{tt}\right)_{p,j,k} \right);$
  \State \textcolor[rgb]{1.00,0.00,0.00}{$\ 4+,\ 8\times$} $\widetilde{M}_{1,xx} = - \widetilde{\rho}_{xx} u_{3,j,k} - 2\widetilde{\rho}_{x} \left(u_{x}\right)_{3,j,k} + \widetilde{\rho} \left(u_{xx}\right)_{3,j,k};\ \widetilde{M}_{2,xx} = \widetilde{\rho}_{xx} u_{4,j,k} + 2\widetilde{\rho}_{x} \left(u_{x}\right)_{4,j,k} - \widetilde{\rho} \left(u_{xx}\right)_{4,j,k};$
  \State \textcolor[rgb]{1.00,0.00,0.00}{$\ 4+,\ 8\times$} $\widetilde{M}_{3,xx} = \widetilde{\rho}_{xx} u_{1,j,k} + 2\widetilde{\rho}_{x} \left(u_{x}\right)_{1,j,k} - \widetilde{\rho} \left(u_{xx}\right)_{1,j,k};\ \widetilde{M}_{4,xx} = - \widetilde{\rho}_{xx} u_{2,j,k} - 2\widetilde{\rho}_{x} \left(u_{x}\right)_{2,j,k} + \widetilde{\rho} \left(u_{xx}\right)_{2,j,k};$
  \State \textcolor[rgb]{1.00,0.00,0.00}{$\ 3+,\ 4\times$} $\widetilde{M}_{1,tx} = - \widetilde{\rho}_{tx} u_{3,j,k} - \widetilde{\rho}_{t} \left(u_{x}\right)_{3,j,k} - \widetilde{\rho}_{x} \left(u_{t}\right)_{3,j,k} + \widetilde{\rho} \left(u_{tx}\right)_{3,j,k};$
  \State \textcolor[rgb]{1.00,0.00,0.00}{$\ 3+,\ 4\times$} $\widetilde{M}_{2,tx} = \widetilde{\rho}_{tx} u_{4,j,k} + \widetilde{\rho}_{t} \left(u_{x}\right)_{4,j,k} + \widetilde{\rho}_{x} \left(u_{t}\right)_{4,j,k} - \widetilde{\rho} \left(u_{tx}\right)_{4,j,k};$
  \State \textcolor[rgb]{1.00,0.00,0.00}{$\ 3+,\ 4\times$} $\widetilde{M}_{3,tx} = \widetilde{\rho}_{tx} u_{1,j,k} + \widetilde{\rho}_{t} \left(u_{x}\right)_{1,j,k} + \widetilde{\rho}_{x} \left(u_{t}\right)_{1,j,k} - \widetilde{\rho} \left(u_{tx}\right)_{1,j,k};$
  \State \textcolor[rgb]{1.00,0.00,0.00}{$\ 3+,\ 4\times$} $\widetilde{M}_{4,tx} = - \widetilde{\rho}_{tx} u_{2,j,k} - \widetilde{\rho}_{t} \left(u_{x}\right)_{2,j,k} - \widetilde{\rho}_{x} \left(u_{t}\right)_{2,j,k} + \widetilde{\rho} \left(u_{tx}\right)_{2,j,k};$
  \State \textcolor[rgb]{1.00,0.00,0.00}{$\ 4+,\ 8\times$} $\widetilde{M}_{1,tt} = - \widetilde{\rho}_{tt} u_{3,j,k} - 2\widetilde{\rho}_{t} \left(u_{t}\right)_{3,j,k} + \widetilde{\rho} \left(u_{tt}\right)_{3,j,k};\ \widetilde{M}_{2,tt} = \widetilde{\rho}_{tt} u_{4,j,k} + 2\widetilde{\rho}_{t} \left(u_{t}\right)_{4,j,k} - \widetilde{\rho} \left(u_{tt}\right)_{4,j,k};$
  \State \textcolor[rgb]{1.00,0.00,0.00}{$\ 4+,\ 8\times$} $\widetilde{M}_{3,tt} = \widetilde{\rho}_{tt} u_{1,j,k} + 2\widetilde{\rho}_{t} \left(u_{t}\right)_{1,j,k} - \widetilde{\rho} \left(u_{tt}\right)_{1,j,k};\ \widetilde{M}_{4,tt} = - \widetilde{\rho}_{tt} u_{2,j,k} - 2\widetilde{\rho}_{t} \left(u_{t}\right)_{2,j,k} + \widetilde{\rho} \left(u_{tt}\right)_{2,j,k};$
  \State \textcolor[rgb]{1.00,0.00,0.00}{$\ 2+,\ 0\times$} $\left(u_{txx}\right)_{1,j,k} = - \left(u_{xxx}\right)_{2,j,k} + \widetilde{M}_{1,xx};\ \left(u_{txx}\right)_{2,j,k} = - \left(u_{xxx}\right)_{1,j,k} + \widetilde{M}_{2,xx};$
  \State \textcolor[rgb]{1.00,0.00,0.00}{$\ 2+,\ 0\times$} $\left(u_{txx}\right)_{3,j,k} = - \left(u_{xxx}\right)_{4,j,k} + \widetilde{M}_{3,xx};\ \left(u_{txx}\right)_{4,j,k} = - \left(u_{xxx}\right)_{3,j,k} + \widetilde{M}_{4,xx};$
  \State \textcolor[rgb]{1.00,0.00,0.00}{$\ 2+,\ 0\times$} $\left(u_{ttx}\right)_{1,j,k} = - \left(u_{txx}\right)_{2,j,k} + \widetilde{M}_{1,tx};\ \left(u_{ttx}\right)_{2,j,k} = - \left(u_{txx}\right)_{1,j,k} + \widetilde{M}_{2,tx};$
  \State \textcolor[rgb]{1.00,0.00,0.00}{$\ 2+,\ 0\times$} $\left(u_{ttx}\right)_{3,j,k} = - \left(u_{txx}\right)_{4,j,k} + \widetilde{M}_{3,tx};\ \left(u_{ttx}\right)_{4,j,k} = - \left(u_{txx}\right)_{3,j,k} + \widetilde{M}_{4,tx};$
  \State \textcolor[rgb]{1.00,0.00,0.00}{$\ 2+,\ 0\times$} $\left(u_{ttt}\right)_{1,j,k} = - \left(u_{ttx}\right)_{2,j,k} + \widetilde{M}_{1,tt};\ \left(u_{ttt}\right)_{2,j,k} = - \left(u_{ttx}\right)_{1,j,k} + \widetilde{M}_{2,tt};$
  \State \textcolor[rgb]{1.00,0.00,0.00}{$\ 2+,\ 0\times$} $\left(u_{ttt}\right)_{3,j,k} = - \left(u_{ttx}\right)_{4,j,k} + \widetilde{M}_{3,tt};\ \left(u_{ttt}\right)_{4,j,k} = - \left(u_{ttx}\right)_{3,j,k} + \widetilde{M}_{4,tt};$
  \State \textcolor[rgb]{1.00,0.00,0.00}{$\ 7+,13 \times$} $\widetilde{\rho}_{ttt} = 2\lambda \left(\rho_{ttt}\right)_{j,k} = \widehat{\Lambda} \sum \limits_{p=1}^{4} \left( -1 \right)^{p+1} \left( 3 \left(u_{t}\right)_{p,j,k} \left(u_{tt}\right)_{p,j,k} + u_{p,j,k} \left(u_{ttt}\right)_{p,j,k} \right);$
  \State \textcolor[rgb]{1.00,0.00,0.00}{$\ 3+,\ 6\times$} $\widetilde{M}_{1,ttt} = - \widetilde{\rho}_{ttt} u_{3,j,k} - 3\widetilde{\rho}_{tt} \left(u_{t}\right)_{3,j,k} - 3\widetilde{\rho}_{t} \left(u_{tt}\right)_{3,j,k} + \widetilde{\rho} \left(u_{ttt}\right)_{3,j,k};$
  \State \textcolor[rgb]{1.00,0.00,0.00}{$\ 3+,\ 6\times$} $\widetilde{M}_{2,ttt} = \widetilde{\rho}_{ttt} u_{4,j,k} + 3\widetilde{\rho}_{tt} \left(u_{t}\right)_{4,j,k} + 3\widetilde{\rho}_{t} \left(u_{tt}\right)_{4,j,k} - \widetilde{\rho} \left(u_{ttt}\right)_{4,j,k};$
  \State \textcolor[rgb]{1.00,0.00,0.00}{$\ 3+,\ 6\times$} $\widetilde{M}_{3,ttt} = \widetilde{\rho}_{ttt} u_{1,j,k} + 3\widetilde{\rho}_{tt} \left(u_{t}\right)_{1,j,k} + 3\widetilde{\rho}_{t} \left(u_{tt}\right)_{1,j,k} - \widetilde{\rho} \left(u_{ttt}\right)_{1,j,k};$
  \State \textcolor[rgb]{1.00,0.00,0.00}{$\ 3+,\ 6\times$} $\widetilde{M}_{4,ttt} = - \widetilde{\rho}_{ttt} u_{2,j,k} - 3\widetilde{\rho}_{tt} \left(u_{t}\right)_{2,j,k} - 3\widetilde{\rho}_{t} \left(u_{tt}\right)_{2,j,k} + \widetilde{\rho} \left(u_{ttt}\right)_{2,j,k};$
  \State \textcolor[rgb]{1.00,0.00,0.00}{$\ 3+,\ 4\times$} $\mathrm{temp}_{1,1} = W_{1k} \left(u_{2,j,k} + t_{1} \left(u_{t}\right)_{2,j,k} + t_{2} \left(u_{tt}\right)_{2,j,k} + t_{3} \left(u_{ttt}\right)_{2,j,k}\right);$
  \State \textcolor[rgb]{1.00,0.00,0.00}{$\ 3+,\ 4\times$} $\mathrm{temp}_{1,2} = W_{1k} \left(u_{1,j,k} + t_{1} \left(u_{t}\right)_{1,j,k} + t_{2} \left(u_{tt}\right)_{1,j,k} + t_{3} \left(u_{ttt}\right)_{1,j,k}\right);$
  \State \textcolor[rgb]{1.00,0.00,0.00}{$\ 3+,\ 4\times$} $\mathrm{temp}_{1,3} = W_{1k} \left(u_{4,j,k} + t_{1} \left(u_{t}\right)_{4,j,k} + t_{2} \left(u_{tt}\right)_{4,j,k} + t_{3} \left(u_{ttt}\right)_{4,j,k}\right);$
  \State \textcolor[rgb]{1.00,0.00,0.00}{$\ 3+,\ 4\times$} $\mathrm{temp}_{1,4} = W_{1k} \left(u_{3,j,k} + t_{1} \left(u_{t}\right)_{3,j,k} + t_{2} \left(u_{tt}\right)_{3,j,k} + t_{3} \left(u_{ttt}\right)_{3,j,k}\right);$
  \State \textcolor[rgb]{1.00,0.00,1.00}{$12 +,16 \times$} $\mathrm{temp}_{2,p} = W_{1k}\left(\widetilde{M}_{p} + t_{1} \widetilde{M}_{p,t} + t_{2} \widetilde{M}_{p,tt} + t_{3} \widetilde{M}_{p,ttt}\right);$
  \State \textcolor[rgb]{1.00,0.00,1.00}{$\ 4+,\ 0\times$} $F_{p,j}^{\left(0\right)} = F_{p,j}^{\left(0\right)} + \mathrm{temp}_{2,p};$
  \State \textcolor[rgb]{1.00,0.00,1.00}{$\ 8+,\ 4\times$} $F_{p,j}^{\left(1\right)} = F_{p,j}^{\left(1\right)} + \mathrm{temp}_{1,p} + D_{1k} \mathrm{temp}_{2,p};$
  \State \textcolor[rgb]{1.00,0.00,1.00}{$\ 8+,\ 8\times$} $F_{p,j}^{\left(2\right)} = F_{p,j}^{\left(2\right)} + D_{4k} \mathrm{temp}_{1,p} + D_{2k} \mathrm{temp}_{2,p};$
  \EndFor
  \State \textcolor[rgb]{1.00,0.00,1.00}{$\ 8+,\ 4\times$} $u_{p,j}^{\left(0\right)} \left( t_{k+1} \right) = u_{p,j}^{\left(0\right)} \left( t_{k} \right) + T_{0} \left[ F_{p,j}^{\left(0\right)} - Q_{1,p,j} \right]$
  \State \textcolor[rgb]{1.00,0.00,1.00}{$\ 8+,\ 8\times$} $u_{p,j}^{\left(1\right)} \left( t_{k+1} \right) = u_{p,j}^{\left(1\right)} \left( t_{k} \right) + T_{1} \left[ F_{p,j}^{\left(1\right)} - Q_{2,p,j} C_{1} \right];$
  \State \textcolor[rgb]{1.00,0.00,1.00}{$\ 8+,\ 8\times$} $u_{p,j}^{\left(2\right)} \left( t_{k+1} \right) = u_{p,j}^{\left(2\right)} \left( t_{k} \right) + T_{2} \left[ F_{p,j}^{\left(2\right)} - Q_{1,p,j} C_{2} \right];$
  \EndFor
  \State \textcolor[rgb]{1.00,0.00,0.00}{$1+,0\times$} time=time+$\tau$;
  \EndWhile
  \end{algorithmic}
\end{breakablealgorithm}

\begin{breakablealgorithm}
  \caption{Pseudo codes for $P^{2}$-TSDG}
  \begin{algorithmic}[1]
  \Require The given initial data $u_{p,j}^{\left(l\right)} \left( t_{0}=0 \right),\ \textcolor[rgb]{1.00,0.00,0.00}{p=1,2,3,4},\ l=0,1,2;$
  \Ensure $u_{p,j}^{\left(l\right)} \left( T \right),$ $T$: the final time;
  \State Set $a_{j}^{\left(0\right)} =a_{0} = \Delta x;\ a_{j}^{\left(1\right)} = a_{1} = \frac{\Delta x^{3}}{12};\ a_{j}^{\left(2\right)} = a_{2} = \frac{\Delta x^{5}}{180};\ C_{1}=\frac{\Delta x}{2};\ C_{2}=\frac{\Delta x^{2}}{6};\ C_{6}=\frac{C_{2}}{2};\ \Lambda=2 \lambda;\ \widehat{\Lambda}=4 \lambda;\ \tau = \frac{\mu \Delta x}{2*2+1};$
  \For{$k=1:P$}
  \State $D_{1k}=C_{1} \tilde{x}_{k};\ D_{2k}=D_{1k}^{2} - C_{6};\ D_{4k}=2D_{1k};\ W_{1k}=C_{1} \omega_{k};$
  \EndFor
  \State Set time$=0$; $k=-1$;
  \While{time$<T$}
  \State \textcolor[rgb]{1.00,0.00,0.00}{$1+,0\times$} $k=k+1$;
  \If{time$+\tau>T$}
  \State $\tau = T-$time;
  \EndIf
  \State \textcolor[rgb]{1.00,0.00,0.00}{$\ 2+,\ 6\times$} $\theta = \frac{1}{3},\ t_{1}=\frac{\tau}{4};\ t_{2}=\frac{\tau}{3\left(1-\theta\right)};\ t_{3}=\frac{\theta \tau}{2};\ t_{4} = \frac{\tau}{2}-t_{3};$
  \State \textcolor[rgb]{1.00,0.00,0.00}{$\ 0+,\ 6\times$} $T_{0}=\frac{t_{2}}{a_{0}};\ T_{1}=\frac{t_{2}}{a_{1}};\ T_{2}=\frac{t_{2}}{a_{2}};\ \widetilde{T}_{0}=\frac{\tau}{a_{0}};\ \widetilde{T}_{1}=\frac{\tau}{a_{1}};\ \widetilde{T}_{2}=\frac{\tau}{a_{2}};$
  \State \textcolor[rgb]{1.00,0.00,0.00}{Stage 1}
  \For{$j = 1:J$}
  \State \textcolor[rgb]{1.00,0.00,1.00}{$12 +,\ 8\times$} $L_{p} = u_{p,j}^{\left(0\right)} \left(t_{k} \right) + C_{2} u_{p,j}^{\left(2\right)} \left(t_{k} \right);\  R_{p} = C_{1} u_{p,j}^{\left(1\right)} \left(t_{k} \right);\ u_{p,j+\frac{1}{2}}^{-} = L_{p} + R_{p};\ u_{p,j-\frac{1}{2}}^{+} = L_{p} - R_{p};$
  \State \textcolor[rgb]{1.00,0.00,1.00}{$\ 8+,\ 4\times$} $L_{p} = u_{p,j}^{\left(1\right)} \left(t_{k} \right);\  R_{p} = \Delta x u_{p,j}^{\left(2\right)} \left(t_{k} \right);\ \left(u_{x}\right)_{p,j+\frac{1}{2}}^{-} = L_{p} + R_{p};\ \left(u_{x}\right)_{p,j-\frac{1}{2}}^{+} = L_{p} - R_{p};$
  \EndFor
  \State $u_{p,\frac{1}{2}}^{-} = 0;\ u_{p,J+\frac{1}{2}}^{+} = 0;\ \left(u_{x}\right)_{p,\frac{1}{2}}^{-} = 0;\ \left(u_{x}\right)_{p,J+\frac{1}{2}}^{+} = 0;$
  \For{$j = 0:J$}
  \State \textcolor[rgb]{1.00,0.00,0.00}{$\ 8+,10 \times$} $\widetilde{\rho}^{\pm} = m - 2\lambda \rho_{j+\frac{1}{2}}^{\pm} = m - \Lambda \sum \limits_{p=1}^{4} \left( -1 \right)^{p+1} \left( u_{p,j+\frac{1}{2}}^{\pm} \right)^{2};$
  \State \textcolor[rgb]{1.00,0.00,0.00}{$\ 4+,\ 4\times$} $\left(u_{t}\right)_{1,j+\frac{1}{2}}^{\pm} = - \left(u_{x}\right)_{2,j+\frac{1}{2}}^{\pm} + \widetilde{\rho}^{\pm} u_{3,j+\frac{1}{2}}^{\pm};\ \left(u_{t}\right)_{2,j+\frac{1}{2}}^{\pm} = - \left(u_{x}\right)_{1,j+\frac{1}{2}}^{\pm} - \widetilde{\rho}^{\pm} u_{4,j+\frac{1}{2}}^{\pm};$
  \State \textcolor[rgb]{1.00,0.00,0.00}{$\ 4+,\ 4\times$} $\left(u_{t}\right)_{3,j+\frac{1}{2}}^{\pm} = - \left(u_{x}\right)_{4,j+\frac{1}{2}}^{\pm} - \widetilde{\rho}^{\pm} u_{1,j+\frac{1}{2}}^{\pm};\ \left(u_{t}\right)_{4,j+\frac{1}{2}}^{\pm} = - \left(u_{x}\right)_{3,j+\frac{1}{2}}^{\pm} + \widetilde{\rho}^{\pm} u_{2,j+\frac{1}{2}}^{\pm};$
  \State \textcolor[rgb]{1.00,0.00,0.00}{$\ 3+,\ 1\times$} $\widehat{\mathcal{F}}_{1,1,j+\frac{1}{2}} = \frac{1}{2} \left[ u_{2,j+\frac{1}{2}}^{-} + u_{2,j+\frac{1}{2}}^{+} - \left( u_{1,j+\frac{1}{2}}^{+} - u_{1,j+\frac{1}{2}}^{-}\right) \right];$
  \State \textcolor[rgb]{1.00,0.00,0.00}{$\ 3+,\ 1\times$} $\widehat{\mathcal{F}}_{1,2,j+\frac{1}{2}} = \frac{1}{2} \left[ u_{1,j+\frac{1}{2}}^{-} + u_{1,j+\frac{1}{2}}^{+} - \left( u_{2,j+\frac{1}{2}}^{+} - u_{2,j+\frac{1}{2}}^{-}\right) \right];$
  \State \textcolor[rgb]{1.00,0.00,0.00}{$\ 3+,\ 1\times$} $\widehat{\mathcal{F}}_{1,3,j+\frac{1}{2}} = \frac{1}{2} \left[ u_{4,j+\frac{1}{2}}^{-} + u_{4,j+\frac{1}{2}}^{+} - \left( u_{3,j+\frac{1}{2}}^{+} - u_{3,j+\frac{1}{2}}^{-}\right) \right];$
  \State \textcolor[rgb]{1.00,0.00,0.00}{$\ 3+,\ 1\times$} $\widehat{\mathcal{F}}_{1,4,j+\frac{1}{2}} = \frac{1}{2} \left[ u_{3,j+\frac{1}{2}}^{-} + u_{3,j+\frac{1}{2}}^{+} - \left( u_{4,j+\frac{1}{2}}^{+} - u_{4,j+\frac{1}{2}}^{-}\right) \right];$
  \State \textcolor[rgb]{1.00,0.00,0.00}{$\ 3+,\ 1\times$} $\widehat{\mathcal{F}}_{2,1,j+\frac{1}{2}} = \frac{1}{2} \left[ \left(u_{t}\right)_{2,j+\frac{1}{2}}^{-} + \left(u_{t}\right)_{2,j+\frac{1}{2}}^{+} - \left( u_{1,j+\frac{1}{2}}^{+} - u_{1,j+\frac{1}{2}}^{-}\right) \right];$
  \State \textcolor[rgb]{1.00,0.00,0.00}{$\ 3+,\ 1\times$} $\widehat{\mathcal{F}}_{2,2,j+\frac{1}{2}} = \frac{1}{2} \left[ \left(u_{t}\right)_{1,j+\frac{1}{2}}^{-} + \left(u_{t}\right)_{1,j+\frac{1}{2}}^{+} - \left( u_{2,j+\frac{1}{2}}^{+} - u_{2,j+\frac{1}{2}}^{-}\right) \right];$
  \State \textcolor[rgb]{1.00,0.00,0.00}{$\ 3+,\ 1\times$} $\widehat{\mathcal{F}}_{2,3,j+\frac{1}{2}} = \frac{1}{2} \left[ \left(u_{t}\right)_{4,j+\frac{1}{2}}^{-} + \left(u_{t}\right)_{4,j+\frac{1}{2}}^{+} - \left( u_{3,j+\frac{1}{2}}^{+} - u_{3,j+\frac{1}{2}}^{-}\right) \right];$
  \State \textcolor[rgb]{1.00,0.00,0.00}{$\ 3+,\ 1\times$} $\widehat{\mathcal{F}}_{2,4,j+\frac{1}{2}} = \frac{1}{2} \left[ \left(u_{t}\right)_{3,j+\frac{1}{2}}^{-} + \left(u_{t}\right)_{3,j+\frac{1}{2}}^{+} - \left( u_{4,j+\frac{1}{2}}^{+} - u_{4,j+\frac{1}{2}}^{-}\right) \right];$
  \EndFor
  \For{$j=1:J$}
  \State \textcolor[rgb]{1.00,0.00,1.00}{$\ 8+,\ 0\times$} $Q_{1,p,j} = \widehat{\mathcal{F}}_{1,p,j+\frac{1}{2}} - \widehat{\mathcal{F}}_{1,p,j-\frac{1}{2}},\ Q_{2,p,j} = \widehat{\mathcal{F}}_{1,p,j+\frac{1}{2}} + \widehat{\mathcal{F}}_{1,p,j-\frac{1}{2}};$
  \State \textcolor[rgb]{1.00,0.00,1.00}{$\ 8+,\ 0\times$} $Q_{3,p,j} = \widehat{\mathcal{F}}_{p,j+\frac{1}{2}} - \widehat{\mathcal{F}}_{p,j-\frac{1}{2}},\ Q_{4,p,j} = \widehat{\mathcal{F}}_{p,j+\frac{1}{2}} + \widehat{\mathcal{F}}_{p,j-\frac{1}{2}};$
  \State $F_{1,p,j}^{\left(l\right)} = 0,\ l=0,1,2;\ F_{2,p,j}^{\left(l\right)} = 0,\ l=0,1,2;$
  \For{$k=1:P$}
  \State \textcolor[rgb]{1.00,0.00,1.00}{$12 +,12 \times$} $u_{p,j,k} = u_{p,j}^{\left(0\right)} \left(t_{k} \right) + u_{p,j}^{\left(1\right)} \left(t_{k} \right) D_{1k} + u_{p,j}^{\left(2\right)} \left(t_{k} \right) D_{2k};\ \left(u_{x}\right)_{p,j,k} = u_{p,j}^{\left(1\right)} \left(t_{k} \right) + u_{p,j}^{\left(2\right)} \left(t_{k} \right) D_{4k};$
  \State \textcolor[rgb]{1.00,0.00,0.00}{$\ 4+,\ 5\times$} $\widetilde{\rho} = m - 2\lambda \rho_{j,k} = m - \Lambda \left( \left( u_{1,j,k} \right)^{2} + \left( u_{3,j,k} \right)^{2} - \left( u_{2,j,k} \right)^{2} - \left( u_{4,j,k} \right)^{2} \right);$
  \State \textcolor[rgb]{1.00,0.00,0.00}{$\ 0+,\ 4\times$} $\widetilde{M}_{1} = \widetilde{\rho} u_{3,j,k};\ \widetilde{M}_{2} = -\widetilde{\rho} u_{4,j,k};\ \widetilde{M}_{3} = -\widetilde{\rho} u_{1,j,k};\ \widetilde{M}_{4} = \widetilde{\rho} u_{2,j,k};$
  \State \textcolor[rgb]{1.00,0.00,0.00}{$\ 2+,\ 0\times$} $\left(u_{t}\right)_{1,j,k} = - \left(u_{x}\right)_{2,j,k} + \widetilde{M}_{1};\ \left(u_{t}\right)_{2,j,k} = - \left(u_{x}\right)_{1,j,k} + \widetilde{M}_{2};$
  \State \textcolor[rgb]{1.00,0.00,0.00}{$\ 2+,\ 0\times$} $\left(u_{t}\right)_{3,j,k} = - \left(u_{x}\right)_{4,j,k} + \widetilde{M}_{3};\ \left(u_{t}\right)_{4,j,k} = - \left(u_{x}\right)_{3,j,k} + \widetilde{M}_{4};$
  \State \textcolor[rgb]{1.00,0.00,0.00}{$\ 3+,\ 5\times$} $\widetilde{\rho}_{t} = 2\lambda \left(\rho_{t}\right)_{j,k} = \widehat{\Lambda} \sum \limits_{p=1}^{4} \left( -1 \right)^{p+1} \left( u_{p,j,k} \left(u_{t}\right)_{p,j,k} \right);$
  \State \textcolor[rgb]{1.00,0.00,0.00}{$\ 2+,\ 4\times$} $\widetilde{M}_{1,t} = - \widetilde{\rho}_{t} u_{3,j,k} + \widetilde{\rho} \left(u_{t}\right)_{3,j,k};\ \widetilde{M}_{2,t} = \widetilde{\rho}_{t} u_{4,j,k} - \widetilde{\rho} \left(u_{t}\right)_{4,j,k};$
  \State \textcolor[rgb]{1.00,0.00,0.00}{$\ 2+,\ 4\times$} $\widetilde{M}_{3,t} = \widetilde{\rho}_{t} u_{1,j,k} - \widetilde{\rho} \left(u_{t}\right)_{1,j,k};\ \widetilde{M}_{4,t} = - \widetilde{\rho}_{t} u_{2,j,k} + \widetilde{\rho} \left(u_{t}\right)_{2,j,k};$
  \State \textcolor[rgb]{1.00,0.00,0.00}{$\ 0+,\ 4\times$} $\mathrm{temp}_{1,1} = W_{1k} u_{2,j,k};\ \mathrm{temp}_{1,2} = W_{1k} u_{1,j,k};\ \mathrm{temp}_{1,3} = W_{1k} u_{4,j,k};\ \mathrm{temp}_{1,4} = W_{1k} u_{3,j,k};$
  \State \textcolor[rgb]{1.00,0.00,1.00}{$\ 4+,\ 4\times$} $\mathrm{temp}_{2,p} = W_{1k}\widetilde{M}_{p};\ F_{1,p,j}^{\left(0\right)} = F_{1,p,j}^{\left(0\right)} + \mathrm{temp}_{2,p};$
  \State \textcolor[rgb]{1.00,0.00,1.00}{$16 +,12 \times$} $F_{1,p,j}^{\left(1\right)} = F_{1,p,j}^{\left(1\right)} + \mathrm{temp}_{1,p} + D_{1k} \mathrm{temp}_{2,p};\ F_{1,p,j}^{\left(2\right)} = F_{1,p,j}^{\left(2\right)} + D_{4k} \mathrm{temp}_{1,p} + D_{2k} \mathrm{temp}_{2,p};$
  \State \textcolor[rgb]{1.00,0.00,0.00}{$\ 0+,\ 4\times$} $\mathrm{temp}_{1,1} = W_{1k} \left(u_{t}\right)_{2,j,k};\ \mathrm{temp}_{1,2} = W_{1k} \left(u_{t}\right)_{1,j,k};\ \mathrm{temp}_{1,3} = W_{1k} \left(u_{t}\right)_{4,j,k};\ \mathrm{temp}_{1,4} = W_{1k} \left(u_{t}\right)_{3,j,k};$
  \State \textcolor[rgb]{1.00,0.00,1.00}{$\ 4+,\ 4\times$} $\mathrm{temp}_{2,p} = W_{1k} \widetilde{M}_{p,t};\ F_{2,p,j}^{\left(0\right)} = F_{2,p,j}^{\left(0\right)} + \mathrm{temp}_{2,p};$
  \State \textcolor[rgb]{1.00,0.00,1.00}{$16 +,12 \times$} $F_{2,p,j}^{\left(1\right)} = F_{2,p,j}^{\left(1\right)} + \mathrm{temp}_{1,p} + D_{1k} \mathrm{temp}_{2,p};\ F_{2,p,j}^{\left(2\right)} = F_{2,p,j}^{\left(2\right)} + D_{4k} \mathrm{temp}_{1,p} + D_{2k} \mathrm{temp}_{2,p};$
  \EndFor
  \State \textcolor[rgb]{1.00,0.00,1.00}{$12+,\ 8\times$} $I_{p,0}= F_{1,p,j}^{\left(0\right)} - Q_{1,p,j};\ I_{p,1}= F_{1,p,j}^{\left(1\right)} - Q_{2,p,j}C_{1};\ I_{p,2}= F_{1,p,j}^{\left(2\right)} - Q_{1,p,j}C_{2};$
  \State \textcolor[rgb]{1.00,0.00,1.00}{$12+,\ 8\times$} $\widehat{I}_{p,0}= F_{2,p,j}^{\left(0\right)}- Q_{3,p,j};\ \widehat{I}_{p,1}= F_{2,p,j}^{\left(1\right)} - Q_{4,p,j}C_{1};\ \widehat{I}_{p,2}= F_{2,p,j}^{\left(2\right)} - Q_{3,p,j}C_{2};;$
  \State \textcolor[rgb]{1.00,0.00,1.00}{$24+,24\times$} $v_{p,j}^{\left(l\right)} = u_{p,j}^{\left(l\right)} \left( t_{k} \right) + T_{l} \left( I_{p,l} + t_{1} \widehat{I}_{p,l} \right),\ l=0,1,2;$
  \EndFor
  \State \textcolor[rgb]{1.00,0.00,0.00}{Stage 2}
  \For{$j = 1:J$}
  \State \textcolor[rgb]{1.00,0.00,1.00}{$12 +,\ 8\times$} $L_{p} = v_{p,j}^{\left(0\right)} + C_{2} v_{p,j}^{\left(2\right)};\  R_{p} = C_{1} v_{p,j}^{\left(1\right)};\ v_{p,j+\frac{1}{2}}^{-} = L_{p} + R_{p};\ v_{p,j-\frac{1}{2}}^{+} = L_{p} - R_{p};$
  \State \textcolor[rgb]{1.00,0.00,1.00}{$\ 8+,\ 4\times$} $L_{p} = v_{p,j}^{\left(1\right)};\  R_{p} = \Delta x v_{p,j}^{\left(2\right)};\ \left(v_{x}\right)_{p,j+\frac{1}{2}}^{-} = L_{p} + R_{p};\ \left(v_{x}\right)_{p,j-\frac{1}{2}}^{+} = L_{p} - R_{p};$
  \EndFor
  \State $v_{\frac{1}{2}}^{-} = 0;\ v_{J+\frac{1}{2}}^{+} = 0;\ \left(v_{x}\right)_{\frac{1}{2}}^{-} = 0;\ \left(v_{x}\right)_{J+\frac{1}{2}}^{+} = 0;$
  \For{$j = 0:J$}
  \State \textcolor[rgb]{1.00,0.00,0.00}{$\ 8+,10 \times$} $\widetilde{\rho}^{\pm} = m - 2\lambda \rho_{j+\frac{1}{2}}^{\pm} = m - \Lambda \sum \limits_{p=1}^{4} \left( -1 \right)^{p+1} \left( v_{p,j+\frac{1}{2}}^{\pm} \right)^{2};$
  \State \textcolor[rgb]{1.00,0.00,0.00}{$\ 4+,\ 4\times$} $\left(v_{t}\right)_{1,j+\frac{1}{2}}^{\pm} = - \left(v_{x}\right)_{2,j+\frac{1}{2}}^{\pm} + \widetilde{\rho}^{\pm} v_{3,j+\frac{1}{2}}^{\pm};\ \left(v_{t}\right)_{2,j+\frac{1}{2}}^{\pm} = - \left(v_{x}\right)_{1,j+\frac{1}{2}}^{\pm} - \widetilde{\rho}^{\pm} v_{4,j+\frac{1}{2}}^{\pm};$
  \State \textcolor[rgb]{1.00,0.00,0.00}{$\ 4+,\ 4\times$} $\left(v_{t}\right)_{3,j+\frac{1}{2}}^{\pm} = - \left(v_{x}\right)_{4,j+\frac{1}{2}}^{\pm} - \widetilde{\rho}^{\pm} v_{1,j+\frac{1}{2}}^{\pm};\ \left(v_{t}\right)_{4,j+\frac{1}{2}}^{\pm} = - \left(v_{x}\right)_{3,j+\frac{1}{2}}^{\pm} + \widetilde{\rho}^{\pm} v_{2,j+\frac{1}{2}}^{\pm};$
  \State \textcolor[rgb]{1.00,0.00,0.00}{$\ 3+,\ 1\times$} $\widehat{\mathcal{F}}_{3,1,j+\frac{1}{2}} = \frac{1}{2} \left[ \left(v_{t}\right)_{2,j+\frac{1}{2}}^{-} + \left(v_{t}\right)_{2,j+\frac{1}{2}}^{+} - \left( v_{1,j+\frac{1}{2}}^{+} - v_{1,j+\frac{1}{2}}^{-}\right) \right];$
  \State \textcolor[rgb]{1.00,0.00,0.00}{$\ 3+,\ 1\times$} $\widehat{\mathcal{F}}_{3,2,j+\frac{1}{2}} = \frac{1}{2} \left[ \left(v_{t}\right)_{1,j+\frac{1}{2}}^{-} + \left(v_{t}\right)_{1,j+\frac{1}{2}}^{+} - \left( v_{2,j+\frac{1}{2}}^{+} - v_{2,j+\frac{1}{2}}^{-}\right) \right];$
  \State \textcolor[rgb]{1.00,0.00,0.00}{$\ 3+,\ 1\times$} $\widehat{\mathcal{F}}_{3,3,j+\frac{1}{2}} = \frac{1}{2} \left[ \left(v_{t}\right)_{4,j+\frac{1}{2}}^{-} + \left(v_{t}\right)_{4,j+\frac{1}{2}}^{+} - \left( v_{3,j+\frac{1}{2}}^{+} - v_{3,j+\frac{1}{2}}^{-}\right) \right];$
  \State \textcolor[rgb]{1.00,0.00,0.00}{$\ 3+,\ 1\times$} $\widehat{\mathcal{F}}_{3,4,j+\frac{1}{2}} = \frac{1}{2} \left[ \left(v_{t}\right)_{3,j+\frac{1}{2}}^{-} + \left(v_{t}\right)_{3,j+\frac{1}{2}}^{+} - \left( v_{4,j+\frac{1}{2}}^{+} - v_{4,j+\frac{1}{2}}^{-}\right) \right];$
  \EndFor
  \For{$j=1:J$}
  \State \textcolor[rgb]{1.00,0.00,1.00}{$\ 8+,\ 0\times$} $Q_{1,p,j} = \widehat{\mathcal{F}}_{3,p,j+\frac{1}{2}} - \widehat{\mathcal{F}}_{3,p,j-\frac{1}{2}},\ Q_{2,p,j} = \widehat{\mathcal{F}}_{3,p,j+\frac{1}{2}} + \widehat{\mathcal{F}}_{3,p,j-\frac{1}{2}};\ F_{p,j}^{\left(l\right)} = 0,\ l=0,1,2;$
  \For{$k=1:P$}
  \State \textcolor[rgb]{1.00,0.00,1.00}{$12 +,12 \times$} $u_{p,j,k} = v_{p,j}^{\left(0\right)} + v_{p,j}^{\left(1\right)} D_{1k} + v_{p,j}^{\left(2\right)} D_{2k} + v_{p,j}^{\left(3\right)} D_{3k};$
  \State \textcolor[rgb]{1.00,0.00,1.00}{$\ 8+,\ 8\times$} $\left(v_{x}\right)_{p,j,k} = v_{p,j}^{\left(1\right)} + v_{p,j}^{\left(2\right)} D_{4k} + v_{p,j}^{\left(3\right)} D_{5k};$
  \State \textcolor[rgb]{1.00,0.00,0.00}{$\ 4+,\ 5\times$} $\widetilde{\rho} = m - 2\lambda \rho_{j,k} = m - \Lambda \left( \left( v_{1,j,k} \right)^{2} + \left( v_{3,j,k} \right)^{2} - \left( v_{2,j,k} \right)^{2} - \left( v_{4,j,k} \right)^{2} \right);$
  \State \textcolor[rgb]{1.00,0.00,0.00}{$\ 0+,\ 4\times$} $\widetilde{M}_{1} = \widetilde{\rho} v_{3,j,k};\ \widetilde{M}_{2} = -\widetilde{\rho} v_{4,j,k};\ \widetilde{M}_{3} = -\widetilde{\rho} v_{1,j,k};\ \widetilde{M}_{4} = \widetilde{\rho} v_{2,j,k};$
  \State \textcolor[rgb]{1.00,0.00,0.00}{$\ 2+,\ 0\times$} $\left(v_{t}\right)_{1,j,k} = - \left(v_{x}\right)_{2,j,k} + \widetilde{M}_{1};\ \left(v_{t}\right)_{2,j,k} = - \left(v_{x}\right)_{1,j,k} + \widetilde{M}_{2};$
  \State \textcolor[rgb]{1.00,0.00,0.00}{$\ 2+,\ 0\times$} $\left(v_{t}\right)_{3,j,k} = - \left(v_{x}\right)_{4,j,k} + \widetilde{M}_{3};\ \left(v_{t}\right)_{4,j,k} = - \left(v_{x}\right)_{3,j,k} + \widetilde{M}_{4};$
  \State \textcolor[rgb]{1.00,0.00,0.00}{$\ 3+,\ 5\times$} $\widetilde{\rho}_{t} = 2\lambda \left(\rho_{t}\right)_{j,k} = \widehat{\Lambda} \sum \limits_{p=1}^{4} \left( -1 \right)^{p+1} \left( v_{p,j,k} \left(v_{t}\right)_{p,j,k} \right);$
  \State \textcolor[rgb]{1.00,0.00,0.00}{$\ 2+,\ 4\times$} $\widetilde{M}_{1,t} = - \widetilde{\rho}_{t} v_{3,j,k} + \widetilde{\rho} \left(v_{t}\right)_{3,j,k};\ \widetilde{M}_{2,t} = \widetilde{\rho}_{t} v_{4,j,k} - \widetilde{\rho} \left(v_{t}\right)_{4,j,k};$
  \State \textcolor[rgb]{1.00,0.00,0.00}{$\ 2+,\ 4\times$} $\widetilde{M}_{3,t} = \widetilde{\rho}_{t} v_{1,j,k} - \widetilde{\rho} \left(v_{t}\right)_{1,j,k};\ \widetilde{M}_{4,t} = - \widetilde{\rho}_{t} v_{2,j,k} + \widetilde{\rho} \left(v_{t}\right)_{2,j,k};$
  \State \textcolor[rgb]{1.00,0.00,0.00}{$\ 0+,\ 4\times$} $\mathrm{temp}_{1,1} = W_{1k} \left(v_{t}\right)_{2,j,k};\ \mathrm{temp}_{1,2} = W_{1k} \left(v_{t}\right)_{1,j,k};\ \mathrm{temp}_{1,3} = W_{1k} \left(v_{t}\right)_{4,j,k};\ \mathrm{temp}_{1,4} = W_{1k} \left(v_{t}\right)_{3,j,k};$
  \State \textcolor[rgb]{1.00,0.00,1.00}{$\ 4+,\ 4\times$} $\mathrm{temp}_{2,p} = W_{1k} \widetilde{M}_{p,t};\ F_{p,j}^{\left(0\right)} = F_{p,j}^{\left(0\right)} + \mathrm{temp}_{2,p};$
  \State \textcolor[rgb]{1.00,0.00,1.00}{$16 +,12 \times$} $F_{p,j}^{\left(1\right)} = F_{p,j}^{\left(1\right)} + \mathrm{temp}_{1,p} + D_{1k} \mathrm{temp}_{2,p};\ F_{p,j}^{\left(2\right)} = F_{p,j}^{\left(2\right)} + D_{4k} \mathrm{temp}_{1,p} + D_{2k} \mathrm{temp}_{2,p};$
  \EndFor
  \State \textcolor[rgb]{1.00,0.00,1.00}{$12+,\ 8\times$} $\widetilde{I}_{p,0}= F_{p,j}^{\left(0\right)} - Q_{1,p,j};\ \widetilde{I}_{p,1}= F_{p,j}^{\left(1\right)} - Q_{2,p,j}C_{1};\ \widetilde{I}_{p,2}= F_{p,j}^{\left(2\right)} - Q_{1,p,j}C_{2};$
  \State \textcolor[rgb]{1.00,0.00,1.00}{$36+,36\times$} $u_{p,j}^{\left(l\right)} \left( t_{k+1} \right) = u_{p,j}^{\left(l\right)} \left( t_{k} \right) + \widetilde{T}_{l} \left( I_{p,l} + t_{3} \widehat{I}_{p,l} + t_{4} \widetilde{I}_{p,l} \right),\ l=0,1,2;$
  \EndFor
  \State \textcolor[rgb]{1.00,0.00,0.00}{$1+,0\times$} time=time+$\tau$;
  \EndWhile
  \end{algorithmic}
\end{breakablealgorithm}

\begin{breakablealgorithm}
  \caption{Pseudo codes for $P^{2}$-RKDG}
  \begin{algorithmic}[1]
  \Require The given initial data $u_{p,j}^{\left(l\right)} \left( t_{0}=0 \right),\ \textcolor[rgb]{1.00,0.00,0.00}{p=1,2,3,4},\ l=0,1,2;$
  \Ensure $u_{p,j}^{\left(l\right)} \left( T \right),$ $T$: the final time;
  \State Set $a_{j}^{\left(0\right)} =a_{0} = \Delta x;\ a_{j}^{\left(1\right)} = a_{1} = \frac{\Delta x^{3}}{12};\ a_{j}^{\left(2\right)} = a_{2} = \frac{\Delta x^{5}}{180};\ C_{1}=\frac{\Delta x}{2};\ C_{2}=\frac{\Delta x^{2}}{6};\ C_{6}=\frac{C_{2}}{2};\ \Lambda=2 \lambda;\ \tau = \frac{\mu \Delta x}{2*2+1};$
  \For{$k=1:P$}
  \State $D_{1k}=C_{1} \tilde{x}_{k};\ D_{2k}=D_{1k}^{2} - C_{6};\ D_{4k}=2D_{1k};\ W_{1k}=C_{1} \omega_{k};$
  \EndFor
  \State Set time$=0$; $k=-1$;
  \While{time$<T$}
  \State \textcolor[rgb]{1.00,0.00,0.00}{$1+,0\times$} $k=k+1$;
  \If{time$+\tau>T$}
  \State $\tau = T-$time;
  \EndIf
  \State \textcolor[rgb]{1.00,0.00,0.00}{$0+,7\times$} $t_{1}=\frac{\tau}{2};\ T_{0}=\frac{t_{1}}{a_{0}};\ T_{1}=\frac{t_{1}}{a_{1}};\ T_{2}=\frac{t_{1}}{a_{2}};\ \widetilde{T}_{0}=\frac{\tau}{a_{0}};\ \widetilde{T}_{1}=\frac{\tau}{a_{1}};\ \widetilde{T}_{2}=\frac{\tau}{a_{2}};$
  \State \textcolor[rgb]{1.00,0.00,0.00}{Stage 1}
  \For{$j = 1:J$}
  \State \textcolor[rgb]{1.00,0.00,1.00}{$12 +,\ 8\times$} $L_{p} = u_{p,j}^{\left(0\right)} \left(t_{k} \right) + C_{2} u_{p,j}^{\left(2\right)} \left(t_{k} \right);\  R_{p} = C_{1} u_{p,j}^{\left(1\right)} \left(t_{k} \right);\ u_{p,j+\frac{1}{2}}^{-} = L_{p} + R_{p};\ u_{p,j-\frac{1}{2}}^{+} = L_{p} - R_{p};$
  \EndFor
  \State $u_{\frac{1}{2}}^{-} = 0;\ u_{J+\frac{1}{2}}^{+} = 0;$
  \For{$j = 0:J$}
  \State \textcolor[rgb]{1.00,0.00,0.00}{$\ 3+,\ 1\times$} $\widehat{\mathcal{F}}_{1,j+\frac{1}{2}} = \frac{1}{2} \left[ u_{2,j+\frac{1}{2}}^{-} + u_{2,j+\frac{1}{2}}^{+} - \left( u_{1,j+\frac{1}{2}}^{+} - u_{1,j+\frac{1}{2}}^{-}\right) \right];$
  \State \textcolor[rgb]{1.00,0.00,0.00}{$\ 3+,\ 1\times$} $\widehat{\mathcal{F}}_{2,j+\frac{1}{2}} = \frac{1}{2} \left[ u_{1,j+\frac{1}{2}}^{-} + u_{1,j+\frac{1}{2}}^{+} - \left( u_{2,j+\frac{1}{2}}^{+} - u_{2,j+\frac{1}{2}}^{-}\right) \right];$
  \State \textcolor[rgb]{1.00,0.00,0.00}{$\ 3+,\ 1\times$} $\widehat{\mathcal{F}}_{3,j+\frac{1}{2}} = \frac{1}{2} \left[ u_{4,j+\frac{1}{2}}^{-} + u_{4,j+\frac{1}{2}}^{+} - \left( u_{3,j+\frac{1}{2}}^{+} - u_{3,j+\frac{1}{2}}^{-}\right) \right];$
  \State \textcolor[rgb]{1.00,0.00,0.00}{$\ 3+,\ 1\times$} $\widehat{\mathcal{F}}_{4,j+\frac{1}{2}} = \frac{1}{2} \left[ u_{3,j+\frac{1}{2}}^{-} + u_{3,j+\frac{1}{2}}^{+} - \left( u_{4,j+\frac{1}{2}}^{+} - u_{4,j+\frac{1}{2}}^{-}\right) \right];$
  \EndFor
  \For{$j=1:J$}
  \State \textcolor[rgb]{1.00,0.00,1.00}{$\ 8+,\ 0\times$} $Q_{1,p,j} = \widehat{\mathcal{F}}_{p,j+\frac{1}{2}} - \widehat{\mathcal{F}}_{p,j-\frac{1}{2}},\ Q_{2,p,j} = \widehat{\mathcal{F}}_{p,j+\frac{1}{2}} + \widehat{\mathcal{F}}_{p,j-\frac{1}{2}};\ F_{p,j}^{\left(l\right)} = 0,\ l=0,1,2;$
  \For{$k=1:P$}
  \State \textcolor[rgb]{1.00,0.00,1.00}{$\ 8+,\ 8\times$} $u_{p,j,k} = u_{p,j}^{\left(0\right)} \left(t_{k} \right) + u_{p,j}^{\left(1\right)} \left(t_{k} \right) D_{1k} + u_{p,j}^{\left(2\right)} \left(t_{k} \right) D_{2k};$
  \State \textcolor[rgb]{1.00,0.00,0.00}{$\ 4+,\ 5\times$} $\widetilde{\rho} = m - 2\lambda \rho_{j,k} = m - \Lambda \left( \left( u_{1,j,k} \right)^{2} + \left( u_{3,j,k} \right)^{2} - \left( u_{2,j,k} \right)^{2} - \left( u_{4,j,k} \right)^{2} \right);$
  \State \textcolor[rgb]{1.00,0.00,0.00}{$\ 0+,\ 4\times$} $\mathrm{temp}_{1,1} = W_{1k} u_{2,j,k};\ \mathrm{temp}_{1,2} = W_{1k} u_{1,j,k};\ \mathrm{temp}_{1,3} = W_{1k} u_{4,j,k};\ \mathrm{temp}_{1,4} = W_{1k} u_{3,j,k};$
  \State \textcolor[rgb]{1.00,0.00,0.00}{$\ 0+,\ 8\times$} $\mathrm{temp}_{2,1} = W_{1k} \widetilde{\rho} u_{3,j,k};\ \mathrm{temp}_{2,2} = -W_{1k} \widetilde{\rho} u_{4,j,k};\ \mathrm{temp}_{2,3} = -W_{1k} \widetilde{\rho} u_{1,j,k};\ \mathrm{temp}_{2,4} = W_{1k} \widetilde{\rho} u_{2,j,k};$
  \State \textcolor[rgb]{1.00,0.00,1.00}{$20 +,12 \times$} $F_{p,j}^{\left(0\right)} = F_{p,j}^{\left(0\right)} + \mathrm{temp}_{2,p};\ F_{p,j}^{\left(1\right)} = F_{p,j}^{\left(1\right)} + \mathrm{temp}_{1,p} + D_{1k} \mathrm{temp}_{2,p};\ F_{p,j}^{\left(2\right)} = F_{p,j}^{\left(2\right)} + D_{4k} \mathrm{temp}_{1,p} + D_{2k} \mathrm{temp}_{2,p};$
  \EndFor
  \State \textcolor[rgb]{1.00,0.00,1.00}{$16 +,12 \times$} $v_{p,j}^{\left(0\right)} = u_{p,j}^{\left(0\right)} \left( t_{k} \right) + T_{0} \left[ F_{p,j}^{\left(0\right)} - Q_{1,p,j} \right];\ v_{p,j}^{\left(1\right)} = u_{p,j}^{\left(1\right)} \left( t_{k} \right) + T_{1} \left[ F_{p,j}^{\left(1\right)} - Q_{2,p,j}C_{1} \right];$
  \State \textcolor[rgb]{1.00,0.00,1.00}{$\ 8+,\ 8\times$} $v_{p,j}^{\left(2\right)} = u_{p,j}^{\left(2\right)} \left( t_{k} \right) + T_{2} \left[ F_{p,j}^{\left(2\right)} - Q_{1,p,j}C_{2} \right];$
  \EndFor
  \State \textcolor[rgb]{1.00,0.00,0.00}{Stage 2}
  \For{$j = 1:J$}
  \State \textcolor[rgb]{1.00,0.00,1.00}{$12 +,\ 8\times$} $L_{p} = v_{p,j}^{\left(0\right)} + C_{2} v_{p,j}^{\left(2\right)};\  R_{p} = C_{1} v_{p,j}^{\left(1\right)};\ v_{p,j+\frac{1}{2}}^{-} = L_{p} + R_{p};\ v_{p,j-\frac{1}{2}}^{+} = L_{p} - R_{p};$
  \EndFor
  \State $v_{\frac{1}{2}}^{-} = 0;\ v_{J+\frac{1}{2}}^{+} = 0;$
  \For{$j = 0:J$}
  \State \textcolor[rgb]{1.00,0.00,0.00}{$\ 3+,\ 1\times$} $\widehat{\mathcal{F}}_{1,j+\frac{1}{2}} = \frac{1}{2} \left[ v_{2,j+\frac{1}{2}}^{-} + v_{2,j+\frac{1}{2}}^{+} - \left( v_{1,j+\frac{1}{2}}^{+} - v_{1,j+\frac{1}{2}}^{-}\right) \right];$
  \State \textcolor[rgb]{1.00,0.00,0.00}{$\ 3+,\ 1\times$} $\widehat{\mathcal{F}}_{2,j+\frac{1}{2}} = \frac{1}{2} \left[ v_{1,j+\frac{1}{2}}^{-} + v_{1,j+\frac{1}{2}}^{+} - \left( v_{2,j+\frac{1}{2}}^{+} - v_{2,j+\frac{1}{2}}^{-}\right) \right];$
  \State \textcolor[rgb]{1.00,0.00,0.00}{$\ 3+,\ 1\times$} $\widehat{\mathcal{F}}_{3,j+\frac{1}{2}} = \frac{1}{2} \left[ v_{4,j+\frac{1}{2}}^{-} + v_{4,j+\frac{1}{2}}^{+} - \left( v_{3,j+\frac{1}{2}}^{+} - v_{3,j+\frac{1}{2}}^{-}\right) \right];$
  \State \textcolor[rgb]{1.00,0.00,0.00}{$\ 3+,\ 1\times$} $\widehat{\mathcal{F}}_{4,j+\frac{1}{2}} = \frac{1}{2} \left[ v_{3,j+\frac{1}{2}}^{-} + v_{3,j+\frac{1}{2}}^{+} - \left( v_{4,j+\frac{1}{2}}^{+} - v_{4,j+\frac{1}{2}}^{-}\right) \right];$
  \EndFor
  \For{$j=1:J$}
  \State \textcolor[rgb]{1.00,0.00,1.00}{$\ 8+,\ 0\times$} $Q_{1,p,j} = \widehat{\mathcal{F}}_{p,j+\frac{1}{2}} - \widehat{\mathcal{F}}_{p,j-\frac{1}{2}},\ Q_{2,p,j} = \widehat{\mathcal{F}}_{p,j+\frac{1}{2}} + \widehat{\mathcal{F}}_{p,j-\frac{1}{2}};\ F_{p,j}^{\left(l\right)} = 0,\ l=0,1,2;$
  \For{$k=1:P$}
  \State \textcolor[rgb]{1.00,0.00,1.00}{$\ 8+,\ 8\times$} $v_{p,j,k} = v_{p,j}^{\left(0\right)} \left(t_{k} \right) + v_{p,j}^{\left(1\right)} \left(t_{k} \right) D_{1k} + v_{p,j}^{\left(2\right)} \left(t_{k} \right) D_{2k};$
  \State \textcolor[rgb]{1.00,0.00,0.00}{$\ 4+,\ 5\times$} $\widetilde{\rho} = m - 2\lambda \rho_{j,k} = m - \Lambda \left( \left( v_{1,j,k} \right)^{2} + \left( v_{3,j,k} \right)^{2} - \left( v_{2,j,k} \right)^{2} - \left( v_{4,j,k} \right)^{2} \right);$
  \State \textcolor[rgb]{1.00,0.00,0.00}{$\ 0+,\ 4\times$} $\mathrm{temp}_{1,1} = W_{1k} v_{2,j,k};\ \mathrm{temp}_{1,2} = W_{1k} v_{1,j,k};\ \mathrm{temp}_{1,3} = W_{1k} v_{4,j,k};\ \mathrm{temp}_{1,4} = W_{1k} v_{3,j,k};$
  \State \textcolor[rgb]{1.00,0.00,0.00}{$\ 0+,\ 8\times$} $\mathrm{temp}_{2,1} = W_{1k} \widetilde{\rho} v_{3,j,k};\ \mathrm{temp}_{2,2} = -W_{1k} \widetilde{\rho} v_{4,j,k};\ \mathrm{temp}_{2,3} = -W_{1k} \widetilde{\rho} v_{1,j,k};\ \mathrm{temp}_{2,4} = W_{1k} \widetilde{\rho} v_{2,j,k};$
  \State \textcolor[rgb]{1.00,0.00,1.00}{$20 +,12 \times$} $F_{p,j}^{\left(0\right)} = F_{p,j}^{\left(0\right)} + \mathrm{temp}_{2,p};\ F_{p,j}^{\left(1\right)} = F_{p,j}^{\left(1\right)} + \mathrm{temp}_{1,p} + D_{1k} \mathrm{temp}_{2,p};\ F_{p,j}^{\left(2\right)} = F_{p,j}^{\left(2\right)} + D_{4k} \mathrm{temp}_{1,p} + D_{2k} \mathrm{temp}_{2,p};$
  \EndFor
  \State \textcolor[rgb]{1.00,0.00,1.00}{$16 +,12 \times$} $p_{p,j}^{\left(0\right)} = u_{p,j}^{\left(0\right)} \left( t_{k} \right) + T_{0} \left[ F_{p,j}^{\left(0\right)} - Q_{1,p,j} \right];\ p_{p,j}^{\left(1\right)} = u_{p,j}^{\left(1\right)} \left( t_{k} \right) + T_{1} \left[ F_{p,j}^{\left(1\right)} - Q_{2,p,j}C_{1} \right];$
  \State \textcolor[rgb]{1.00,0.00,1.00}{$\ 8+,\ 8\times$} $p_{p,j}^{\left(2\right)} = u_{p,j}^{\left(2\right)} \left( t_{k} \right) + T_{2} \left[ F_{p,j}^{\left(2\right)} - Q_{1,p,j}C_{2} \right];$
  \EndFor
  \State \textcolor[rgb]{1.00,0.00,0.00}{Stage 3}
  \For{$j = 1:J$}
  \State \textcolor[rgb]{1.00,0.00,1.00}{$12 +,\ 8\times$} $L_{p} = p_{p,j}^{\left(0\right)} + C_{2} p_{p,j}^{\left(2\right)};\  R_{p} = C_{1} p_{p,j}^{\left(1\right)};\ p_{p,j+\frac{1}{2}}^{-} = L_{p} + R_{p};\ p_{p,j-\frac{1}{2}}^{+} = L_{p} - R_{p};$
  \EndFor
  \State $p_{\frac{1}{2}}^{-} = 0;\ p_{J+\frac{1}{2}}^{+} = 0;$
  \For{$j = 0:J$}
  \State \textcolor[rgb]{1.00,0.00,0.00}{$\ 3+,\ 1\times$} $\widehat{\mathcal{F}}_{1,j+\frac{1}{2}} = \frac{1}{2} \left[ p_{2,j+\frac{1}{2}}^{-} + p_{2,j+\frac{1}{2}}^{+} - \left( p_{1,j+\frac{1}{2}}^{+} - p_{1,j+\frac{1}{2}}^{-}\right) \right];$
  \State \textcolor[rgb]{1.00,0.00,0.00}{$\ 3+,\ 1\times$} $\widehat{\mathcal{F}}_{2,j+\frac{1}{2}} = \frac{1}{2} \left[ p_{1,j+\frac{1}{2}}^{-} + p_{1,j+\frac{1}{2}}^{+} - \left( p_{2,j+\frac{1}{2}}^{+} - p_{2,j+\frac{1}{2}}^{-}\right) \right];$
  \State \textcolor[rgb]{1.00,0.00,0.00}{$\ 3+,\ 1\times$} $\widehat{\mathcal{F}}_{3,j+\frac{1}{2}} = \frac{1}{2} \left[ p_{4,j+\frac{1}{2}}^{-} + p_{4,j+\frac{1}{2}}^{+} - \left( p_{3,j+\frac{1}{2}}^{+} - p_{3,j+\frac{1}{2}}^{-}\right) \right];$
  \State \textcolor[rgb]{1.00,0.00,0.00}{$\ 3+,\ 1\times$} $\widehat{\mathcal{F}}_{4,j+\frac{1}{2}} = \frac{1}{2} \left[ p_{3,j+\frac{1}{2}}^{-} + p_{3,j+\frac{1}{2}}^{+} - \left( p_{4,j+\frac{1}{2}}^{+} - p_{4,j+\frac{1}{2}}^{-}\right) \right];$
  \EndFor
  \For{$j=1:J$}
  \State \textcolor[rgb]{1.00,0.00,1.00}{$\ 8+,\ 0\times$} $Q_{1,p,j} = \widehat{\mathcal{F}}_{p,j+\frac{1}{2}} - \widehat{\mathcal{F}}_{p,j-\frac{1}{2}},\ Q_{2,p,j} = \widehat{\mathcal{F}}_{p,j+\frac{1}{2}} + \widehat{\mathcal{F}}_{p,j-\frac{1}{2}};\ F_{p,j}^{\left(l\right)} = 0,\ l=0,1,2;$
  \For{$k=1:P$}
  \State \textcolor[rgb]{1.00,0.00,1.00}{$\ 8+,\ 8\times$} $p_{p,j,k} = p_{p,j}^{\left(0\right)} \left(t_{k} \right) + p_{p,j}^{\left(1\right)} \left(t_{k} \right) D_{1k} + p_{p,j}^{\left(2\right)} \left(t_{k} \right) D_{2k};$
  \State \textcolor[rgb]{1.00,0.00,0.00}{$\ 4+,\ 5\times$} $\widetilde{\rho} = m - 2\lambda \rho_{j,k} = m - \Lambda \left( \left( p_{1,j,k} \right)^{2} + \left( p_{3,j,k} \right)^{2} - \left( p_{2,j,k} \right)^{2} - \left( p_{4,j,k} \right)^{2} \right);$
  \State \textcolor[rgb]{1.00,0.00,0.00}{$\ 0+,\ 4\times$} $\mathrm{temp}_{1,1} = W_{1k} p_{2,j,k};\ \mathrm{temp}_{1,2} = W_{1k} p_{1,j,k};\ \mathrm{temp}_{1,3} = W_{1k} p_{4,j,k};\ \mathrm{temp}_{1,4} = W_{1k} p_{3,j,k};$
  \State \textcolor[rgb]{1.00,0.00,0.00}{$\ 0+,\ 8\times$} $\mathrm{temp}_{2,1} = W_{1k} \widetilde{\rho} p_{3,j,k};\ \mathrm{temp}_{2,2} = -W_{1k} \widetilde{\rho} p_{4,j,k};\ \mathrm{temp}_{2,3} = -W_{1k} \widetilde{\rho} p_{1,j,k};\ \mathrm{temp}_{2,4} = W_{1k} \widetilde{\rho} p_{2,j,k};$
  \State \textcolor[rgb]{1.00,0.00,1.00}{$20 +,12 \times$} $F_{p,j}^{\left(0\right)} = F_{p,j}^{\left(0\right)} + \mathrm{temp}_{2,p};\ F_{p,j}^{\left(1\right)} = F_{p,j}^{\left(1\right)} + \mathrm{temp}_{1,p} + D_{1k} \mathrm{temp}_{2,p};\ F_{p,j}^{\left(2\right)} = F_{p,j}^{\left(2\right)} + D_{4k} \mathrm{temp}_{1,p} + D_{2k} \mathrm{temp}_{2,p};$
  \EndFor
  \State \textcolor[rgb]{1.00,0.00,1.00}{$16 +,12 \times$} $q_{p,j}^{\left(0\right)} = u_{p,j}^{\left(0\right)} \left( t_{k} \right) + \widetilde{T}_{0} \left[ F_{p,j}^{\left(0\right)} - Q_{1,p,j} \right];\ q_{p,j}^{\left(1\right)} = u_{p,j}^{\left(1\right)} \left( t_{k} \right) + \widetilde{T}_{1} \left[ F_{p,j}^{\left(1\right)} - Q_{2,p,j}C_{1} \right];$
  \State \textcolor[rgb]{1.00,0.00,1.00}{$\ 8+,\ 8\times$} $q_{p,j}^{\left(2\right)} = u_{p,j}^{\left(2\right)} \left( t_{k} \right) + \widetilde{T}_{2} \left[ F_{p,j}^{\left(2\right)} - Q_{1,p,j}C_{2} \right];$
  \EndFor
  \State \textcolor[rgb]{1.00,0.00,0.00}{Stage 4}
  \For{$j = 1:J$}
  \State \textcolor[rgb]{1.00,0.00,1.00}{$12 +,\ 8\times$} $L_{p} = q_{p,j}^{\left(0\right)} + C_{2} q_{p,j}^{\left(2\right)};\  R_{p} = C_{1} q_{p,j}^{\left(1\right)};\ q_{p,j+\frac{1}{2}}^{-} = L_{p} + R_{p};\ q_{p,j-\frac{1}{2}}^{+} = L_{p} - R_{p};$
  \EndFor
  \State $q_{\frac{1}{2}}^{-} = 0;\ q_{J+\frac{1}{2}}^{+} = 0;$
  \For{$j = 0:J$}
  \State \textcolor[rgb]{1.00,0.00,0.00}{$\ 3+,\ 1\times$} $\widehat{\mathcal{F}}_{1,j+\frac{1}{2}} = \frac{1}{2} \left[ q_{2,j+\frac{1}{2}}^{-} + q_{2,j+\frac{1}{2}}^{+} - \left( q_{1,j+\frac{1}{2}}^{+} - q_{1,j+\frac{1}{2}}^{-}\right) \right];$
  \State \textcolor[rgb]{1.00,0.00,0.00}{$\ 3+,\ 1\times$} $\widehat{\mathcal{F}}_{2,j+\frac{1}{2}} = \frac{1}{2} \left[ q_{1,j+\frac{1}{2}}^{-} + q_{1,j+\frac{1}{2}}^{+} - \left( q_{2,j+\frac{1}{2}}^{+} - q_{2,j+\frac{1}{2}}^{-}\right) \right];$
  \State \textcolor[rgb]{1.00,0.00,0.00}{$\ 3+,\ 1\times$} $\widehat{\mathcal{F}}_{3,j+\frac{1}{2}} = \frac{1}{2} \left[ q_{4,j+\frac{1}{2}}^{-} + q_{4,j+\frac{1}{2}}^{+} - \left( q_{3,j+\frac{1}{2}}^{+} - q_{3,j+\frac{1}{2}}^{-}\right) \right];$
  \State \textcolor[rgb]{1.00,0.00,0.00}{$\ 3+,\ 1\times$} $\widehat{\mathcal{F}}_{4,j+\frac{1}{2}} = \frac{1}{2} \left[ q_{3,j+\frac{1}{2}}^{-} + q_{3,j+\frac{1}{2}}^{+} - \left( q_{4,j+\frac{1}{2}}^{+} - q_{4,j+\frac{1}{2}}^{-}\right) \right];$
  \EndFor
  \For{$j=1:J$}
  \State \textcolor[rgb]{1.00,0.00,1.00}{$\ 8+,\ 0\times$} $Q_{1,p,j} = \widehat{\mathcal{F}}_{p,j+\frac{1}{2}} - \widehat{\mathcal{F}}_{p,j-\frac{1}{2}},\ Q_{2,p,j} = \widehat{\mathcal{F}}_{p,j+\frac{1}{2}} + \widehat{\mathcal{F}}_{p,j-\frac{1}{2}};\ F_{p,j}^{\left(l\right)} = 0,\ l=0,1,2;$
  \For{$k=1:P$}
  \State \textcolor[rgb]{1.00,0.00,1.00}{$\ 8+,\ 8\times$} $q_{p,j,k} = q_{p,j}^{\left(0\right)} \left(t_{k} \right) + q_{p,j}^{\left(1\right)} \left(t_{k} \right) D_{1k} + q_{p,j}^{\left(2\right)} \left(t_{k} \right) D_{2k};$
  \State \textcolor[rgb]{1.00,0.00,0.00}{$\ 4+,\ 5\times$} $\widetilde{\rho} = m - 2\lambda \rho_{j,k} = m - \Lambda \left( \left( q_{1,j,k} \right)^{2} + \left( q_{3,j,k} \right)^{2} - \left( q_{2,j,k} \right)^{2} - \left( q_{4,j,k} \right)^{2} \right);$
  \State \textcolor[rgb]{1.00,0.00,0.00}{$\ 0+,\ 4\times$} $\mathrm{temp}_{1,1} = W_{1k} q_{2,j,k};\ \mathrm{temp}_{1,2} = W_{1k} q_{1,j,k};\ \mathrm{temp}_{1,3} = W_{1k} q_{4,j,k};\ \mathrm{temp}_{1,4} = W_{1k} q_{3,j,k};$
  \State \textcolor[rgb]{1.00,0.00,0.00}{$\ 0+,\ 8\times$} $\mathrm{temp}_{2,1} = W_{1k} \widetilde{\rho} q_{3,j,k};\ \mathrm{temp}_{2,2} = -W_{1k} \widetilde{\rho} q_{4,j,k};\ \mathrm{temp}_{2,3} = -W_{1k} \widetilde{\rho} q_{1,j,k};\ \mathrm{temp}_{2,4} = W_{1k} \widetilde{\rho} q_{2,j,k};$
  \State \textcolor[rgb]{1.00,0.00,1.00}{$20 +,12 \times$} $F_{p,j}^{\left(0\right)} = F_{p,j}^{\left(0\right)} + \mathrm{temp}_{2,p};\ F_{p,j}^{\left(1\right)} = F_{p,j}^{\left(1\right)} + \mathrm{temp}_{1,p} + D_{1k} \mathrm{temp}_{2,p};\ F_{p,j}^{\left(2\right)} = F_{p,j}^{\left(2\right)} + D_{4k} \mathrm{temp}_{1,p} + D_{2k} \mathrm{temp}_{2,p};$
  \EndFor
  \State \textcolor[rgb]{1.00,0.00,1.00}{$20+,12\times$} $u_{p,j}^{\left(0\right)} \left( t_{k+1} \right) = \frac{1}{3} \left( v_{p,j}^{\left(0\right)} + 2 p_{p,j}^{\left(0\right)} + q_{p,j}^{\left(0\right)} - u_{p,j}^{\left(0\right)} \left( t_{k} \right) + T_{0} \left[ F_{p,j}^{\left(0\right)} - Q_{1,p,j} \right] \right);$
  \State \textcolor[rgb]{1.00,0.00,1.00}{$20+,16\times$} $u_{p,j}^{\left(1\right)} \left( t_{k+1} \right) = \frac{1}{3} \left( v_{p,j}^{\left(1\right)} + 2 p_{p,j}^{\left(1\right)} + q_{p,j}^{\left(1\right)} - u_{p,j}^{\left(1\right)} \left( t_{k} \right) + T_{1} \left[ F_{p,j}^{\left(1\right)} - Q_{2,p,j}C_{1} \right] \right);$
  \State \textcolor[rgb]{1.00,0.00,1.00}{$20+,16\times$} $u_{p,j}^{\left(2\right)} \left( t_{k+1} \right) = \frac{1}{3} \left( v_{p,j}^{\left(2\right)} + 2 p_{p,j}^{\left(2\right)} + q_{p,j}^{\left(2\right)} - u_{p,j}^{\left(2\right)} \left( t_{k} \right) + T_{2} \left[ F_{p,j}^{\left(2\right)} - Q_{1,p,j}C_{2} \right] \right);$
  \EndFor
  \State \textcolor[rgb]{1.00,0.00,0.00}{$1+,0\times$} time=time+$\tau$;
  \EndWhile
  \end{algorithmic}
\end{breakablealgorithm}
}




\end{document}